\documentclass[onefignum,onetabnum]{siamonline190516}

\usepackage{amsfonts,comment}
\usepackage{graphicx,bm,bbm}
\usepackage{epstopdf,lipsum}
\usepackage{algorithmic}
\ifpdf
\DeclareGraphicsExtensions{.eps,.pdf,.png,.jpg}
\else
\DeclareGraphicsExtensions{.eps}
\fi

\usepackage{enumitem}
\setlist[enumerate]{leftmargin=.5in}
\setlist[itemize]{leftmargin=.5in}

\newcommand{\h}[1]{\mathbf{#1}}
\usepackage{multirow,dsfont}
\newsiamremark{remark}{Remark}
\newsiamremark{hypothesis}{Hypothesis}
\crefname{hypothesis}{Hypothesis}{Hypotheses}
\newsiamthm{claim}{Claim}

\usepackage{eqnarray}
\newcommand{\nn}{\nonumber}

\usepackage{cases}

\newcommand{\tr}[1]{\textcolor{black}{#1}}

\headers{Limited-angle CT via $L_1/L_2$}{Wang, Tao, Nagy, and Lou}

\title{Limited-angle CT reconstruction via the $L_1/L_2$ minimization\thanks{
		Submitted to the editors on May 31st, 2020. Revision submitted on September 24th, 2020.
		\funding{C.~Wang was partially supported by HKRGC Grant No.CityU11301120 and NSF CCF HDR TRIPODS grant 1934568. M.~Tao was  supported in part by \tr{the National Key Research and Development Program of China (No.
2018AAA0101100), } the Natural Science Foundation of China (No. 11971228) and  the Jiangsu Provincial National Natural Science Foundation of China (No. BK20181257). J.~Nagy was partially supported by NSF DMS-1819042 and NIH 5R01CA181171-04. Y.~Lou was partially supported by NSF grant CAREER 1846690. }}}

\author{Chao Wang\thanks{Department of Mathematical Sciences, University of Texas at Dallas, Richardson, TX 75080, USA; Department of Radiation Oncology, University of Texas Southwestern Medical Center, Dallas, TX 75390, USA (\email{chaowang.hk@gmail.com}).  }
	\and Min Tao\thanks{Department of Mathematics, National Key Laboratory for Novel Software Technology, Nanjing University, Nanjing 210093, China (\email{taom@nju.edu.cn}).}
	\and James Nagy\thanks{Department of Mathematics, Emory University, Atlanta, GA 30322, USA (\email{jnagy@emory.edu}) }
	\and Yifei Lou\thanks{Department of Mathematical Sciences, University of Texas at Dallas, Richardson, TX 75080, USA (\email{yifei.lou@utdallas.edu})}. }
\usepackage{amsopn}

\ifpdf
\hypersetup{
	pdftitle={L1dL2_CT_limited_angle},
	pdfauthor={C. Wang, et al.}
}
\fi

\begin{document}
	
	\maketitle
	
	% REQUIRED
	\begin{abstract}
		In this paper, we consider to minimize the $L_1/L_2$ term on the gradient for a limited-angle scanning problem in computed tomography (CT) reconstruction. We design a specific splitting framework for an unconstrained optimization model so that the alternating direction method of multipliers (ADMM) has guaranteed convergence under certain conditions. In addition, we incorporate a box constraint that is reasonable for imaging applications, and the convergence for the additional box constraint can also be established. Numerical results 
		on both synthetic and experimental datasets demonstrate the \tr{effectiveness and efficiency} of our proposed approach, showing significant improvements over the state-of-the-art methods in the limited-angle CT reconstruction. 
%		Specifically worth noticing  is an exact recovery of the Shepp-Logan phantom from noiseless projection data with  $30^\circ$ scanning angle. 
	\end{abstract}
	
	% REQUIRED
	
	% REQUIRED
	
	\section{Introduction}

	Recent developments in science and technology have led to a revolution in data processing, as large datasets are becoming increasingly available and useful. In medical imaging, a series of imaging modalities, such as x-ray computed tomography (CT) \cite{pric_CT,essential_ph_medimag,de2010statistical,elbakri2002statistical}, magnetic resonance imaging (MRI) \cite{sparseMRI}, and electroencephalography (EEG) \cite{liuRLHSW17,liuRQLW18}, offer different perspectives to facilitate diagnostics. On the other hand, however, one often faces ``small data,'' e.g., only a small number of  CT scans are allowed for the sake of radiation dose.
	In this paper, we are particularly interested in a limited-angle CT reconstruction problem, which often occurs in many medical imaging applications.  
%	For example, mammography \cite{wang2011low,zhang2006comparative} can not have a complete scanning if an object is too large.
In breast imaging, a technique gaining wide interests is tomosynthesis (sometimes referred to as 3D mammography) \cite{wang2011low,zhang2006comparative}, which is a limited angle tomography approach designed to produce pseudo three-dimensional images while keeping the radiation exposure to approximately that of traditional two-dimensional mammograms.
% \tb{[delete?]	 Other issues to cause  limited-angle scanning include  short exposure time \cite{jin2010motion},  engine exam \cite{quinto1998exterior},  and restricted scanning \cite{gao2007volumetric}.
% }	

The CT data collection is a nonlinear process due to the polychromatic nature \cite{AM01,gu2015polychromatic} of the x-ray source. A common practice in CT adopts some linearization and discretization schemes that express the formation model as
	$
	f = Au,
	$
	where $f$ denotes the measurement data, $u$ is the attenuation coefficients to be recovered, and $A$ is  a projection matrix.
Specifically for this paper, we consider two types of projection geometries: 	parallel beam and fan beam, which are popular in the CT reconstruction literature. For parallel beam, the complete scanning angle is $180^\circ,$ while it is $360^\circ$ for fan beam. If we restrict the maximum scanning angle, it becomes the so-called limited-angle scanning, which is much more challenging than the CT reconstruction from the complete scanning.
	Some conventional methods in the CT reconstruction   include filtered back projection (FBP) \cite{feldkamp1984practical,tang2006three}, simultaneous iterative reconstruction technique (SIRT) \cite{AM01}, and simultaneous algebraic reconstruction technique (SART) \cite{SART,jiang2003convergence}. These approaches do not involve any regularization, and perform poorly in the case of  limited-angle and/or noisy data, resulting in severe streaking artifacts \cite{frikel2013characterization,nguyen2015strong}.

	When data is insufficient, one often requires reasonable assumptions to be imposed as a regularization term in order to reconstruct a desired solution.
	As such, the CT reconstruction can be formulated as minimizing an objective function that consists of a data fidelity term and a regularization term.  There are two commonly \tr{studied} data fitting terms  for CT reconstruction: least squares (LS) 
	 \cite{han2010algorithm,jia2010gpu,sidky2008image} and \tr{weighted least-squares (WLS)} \cite{thibault2007three}. In this paper, we focus on the LS data fitting with a discussion of \tr{WLS} in Section \ref{sect:exp_syn}.
%	 , the  as a data fitting term can be addressed in many kinds of CT reconstruction, especially for the data follows Gaussian distribution. In the CT literature,  is also widely used, which is to measure the misfit in Poisson noise.  and will discuss the behavior of these two 
	% We  review some literature on image  regularization techniques for  CT reconstruction.
	As for  regularizations,  the celebrated total variation (TV) \cite{gwang_limitedct,jia2010gpu,sidkyKP06,sidky2008image,carola_dirctv,yu2010soft} prefers piece-wise constant images. However, two noticeable drawbacks for TV are loss of contrast and staircasing artifacts. To resolve its limitations, Jia et al.~\cite{jia2011gpu} utilized  a  tight frame regularization and implemented the algorithm on graphics processing units (GPU) to achieve fast computation. Recently, a combination of TV and wavelet tight frame  was discussed in \cite{weigthed_limitedct}.
	The extension of TV in a nonlocal fashion by exploiting patch similarity was examined in \cite{louZOB10} for the regular CT reconstruction and in \cite{xiaoqun_limitedct} for the limited-angle case.  
	%Thanks to the recent success of deep learning,  machine learning techniques have been introduced to CT reconstruction \cite{gwang2020book}.  
% 	Other related applications included sinogram inpainting via a directional TV \cite{carola_dirctv}  as well as a joint approach of reconstruction and segmentation   \cite{bdong_limtedct}  for x-ray tomography.
	
	The TV semi-norm is equivalent to the $L_1$ norm on the gradient. It is well-known that $L_1$ is the tightest convex approximation to the $L_0$ \textit{norm}\footnote{Note that $L_0$ is not a norm, but often called this way.}, which is used to enforce sparsity for  signals of interest.
	There are several alternatives to approximate the $L_0$ norm, such as  $L_p$ with $0<p<1$ \cite{chartrand07,Xu2012}, transformed $L_1$ \cite{lv2009unified,wang2019nonconvex,zhangX17,zhangX18}, and $L_1$-$L_2$ \cite{louY18,louYHX14,louYX16,maLH17,yinEX14}.
	Algorithmically, Cand\'es et al.~\cite{candes2008enhancing} proposed an  iteratively  reweighted $L_1$ (IRL1) algorithm to solve for the $L_0$ minimization. This idea was reformulated as a scale-space algorithm in \cite{huang2018scale}. % for limited-angle tomography.
	%\cite{li2019image}.

	Motivated by recent works  of using $L_1/L_2$ \cite{l1dl2,TaoLou2020,L1dL2_accelerated} for sparse signal recovery, we apply the $L_1/L_2$ form on the gradient, leading to a new regularization term. This regularization is rather generic in image processing, and we find it works particularly well for piece-wise constant images, owing to its scale-invariant property when approximating $L_0$. In addition, the proposed regularization can mitigate the staircasing artifacts produced by TV, as the $L_2$ norm of the gradient in the denominator should be away from zero. Extensive experiments demonstrate that our method outperforms the state-of-the-art in  CT reconstruction and significant improvements are achieved for the limited-angle case.
$L_1/L_2$ on the gradient was originally proposed in  \cite{l1dl2}, which included the MRI reconstruction as  a proof-of-concept example under the noiseless setting, but it lacks practicality and convergence analysis of the algorithm.
 The contributions of this work are three-fold:
	\begin{enumerate}
		\item We propose a novel regularization  together with a box constraint for the limited-angle CT reconstruction.
		\item We design a specific splitting scheme for solving several related models so that the convergence of ADMM can be established under certain conditions.
		\item We present extensive CT reconstruction results (using phantoms/experimental data under parallel/fan beam) to demonstrate the practicality  of the proposed approach.
	\end{enumerate}
	
	The rest of the paper is organized as follows. In \Cref{sect:prelim}, we present some preliminary materials, such as notations, TV definition, and a previous work of $L_1/L_2$ \cite{l1dl2}. We discuss the proposed models and algorithms in  \Cref{sect:model}, followed by convergence analysis in \Cref{sect:conv}. Experimental studies are conducted in \Cref{sec:experiments} using the projection data of two phantoms subject to two types of noises (Gaussian noise and Poisson noise) as well as two experimental datasets.  Finally, conclusions and future works are given in \Cref{sec:conclusions}.
	
	%In \Cref{fig:illu_CT},  the sinograme is the blue rectangle when  the limited-angle scanning is . Owing to limited amount data, the problem is very ill-posed and challenging.
	
	%Here are some references  ,. And
	%CT has various applications in materials science and medical imaging . Here, we consider the .  The
	%mathematical forward model for the CT is provided by the Radon transform which is defined as an operator $\mathcal{A}$.

	%\begin{figure}
	%	\begin{center}
	%		\begin{tabular}{ccc}
	%90$^\circ$ & 135$^\circ$ & 180$^\circ$ \\
	%\includegraphics[width=0.3\textwidth]{fig/90_sinogram_SL} &
	%			\includegraphics[width=0.3\textwidth]{fig/135_sinogram_SL}  	&			\includegraphics[width=0.3\textwidth]{fig/180_sinogram_SL} \\
	%				\includegraphics[width=0.3\textwidth]{fig/90_sart_SL} &
	%			\includegraphics[width=0.3\textwidth]{fig/135_sart_SL}  	&			\includegraphics[width=0.3\textwidth]{fig/180_sart_SL} \\
	%			\end{tabular}
	%	\end{center}
	%	\caption{Restored image from different ranges of angle with the same amount of projection. \tb{\it [Actually, I don't think we need this figure as the numerical tests already tell the phenomenon. ]}}\label{fig:sart_diff_angle}
	%\end{figure}
	
	\section{Preliminaries}
	\label{sect:prelim}
	
	%In this subsection, we review some previous works in the convex and nonconvex optimization for the related image reconstruction problem.
	Suppose an underlying image is defined on an $m\times n$ Cartesian grid and denote  the Euclidean space $\mathbb{R}^{m\times n}$ as $X$. We adopt a linear index for the 2D image, i.e., for $u\in X$,  $u_{ij}\in \mathbb{R}$ is the $((i-1)m+j)$-th component of $u$. We define a discrete gradient operator,
	\begin{equation}\label{eq:gradient}
		\nabla u:=
	(\nabla_x u, \nabla_y u),
	\end{equation}
%	$ \nabla u:=
%	(\nabla_x u, \nabla_y u)$
	%\{[u_{ij}-u_{(i+1)j}]_{i=1}^{n}\}_{j=1}^m, \{[u_{ij}-u_{i(j+1)}]_{j=1}^{m}\}_{i=1}^n$
with $\nabla_z$ being the forward difference operator in the  $z$-direction for $z\in \{x,y\}$. Denote $Y=X\times X.$ Then $\nabla u\in Y$, and for any $\h p\in Y$, its $((i-1)m+j)$-th component  is  $p_{ij}=(p_{ij,1},p_{ij,2})$. We use a bold letter $\h p$ to indicate that it contains two elements in each component.
	With these notations, we define the inner products
	by
	\begin{equation}
	\langle x,y \rangle_X=\sum_{i,j=1}^{m,n}x_{ij}y_{ij} \quad \text{ and } \quad\langle \h p,\h q \rangle_Y=\sum_{i,j=1}^{m,n}\sum_{k=1}^2 p_{ij,k}q_{ij,k},
	\end{equation}
	as well as the corresponding norms
	\begin{equation}\label{equ:norm_general}
	\|x\|_2=\sqrt{\langle x,x\rangle_X} \quad \text{ and } \quad \|\h p\|_2 = \sqrt{\langle \h p,\h p\rangle_Y}.
	\end{equation}
	
	\medskip
	\subsection{Total variation}
	By incorporating  the TV regularization \cite{rudinOF92} into the data fitting terms, we can obtain the following two models,
	\begin{eqnarray}\label{eq:grad_con_l1}
	&&\min_{u}  {\| \nabla u \|_1} \quad \mathrm{s.t.} \ \  A u =f \\
	\label{eq:grad_uncon_l1}
	&&\min_{u}  {\| \nabla u \|_1} + \frac{\lambda}{2} \|A u -f\|_2^2,
	\end{eqnarray}
	where $\nabla$ is defined in \eqref{eq:gradient}.
	We refer \eqref{eq:grad_con_l1} as a constrained formulation, while \eqref{eq:grad_uncon_l1} as an unconstrained one. The latter is often used  when the noise is present and the parameter $\lambda>0$ in \eqref{eq:grad_uncon_l1} shall be tuned according to the noise level. Note that the TV term, $\|\nabla u\|_1$, is equivalent to the $L_1$ norm of the gradient, which can be  formulated as the \textit{anisotropic TV},
	\begin{equation}\label{equ:anisotropic}
	\|\nabla u\|_1 = \|\nabla_x u\|_1+ \|\nabla_y u\|_1,
	\end{equation}
	or the \textit{isotropic TV}, defined by
	$
	\sum_{i,j=1}^{m,n}\sqrt{(\nabla_x u)_{ij}^2+ (\nabla_y u)_{ij}^2}.
	$
	The anisotropic TV was shown to be superior  over the isotropic one for CT reconstruction \cite{gwang_limitedct}. Here, we also adopt the anisotropic TV to define the $L_1$ norm on the gradient. Besides, the difference of anisotropic and isotropic TV was proposed in  \cite{louZOX15} for general imaging applications.
	There are many efficient algorithms to minimize  \eqref{eq:grad_con_l1} or \eqref{eq:grad_uncon_l1}, including dual projection \cite{chambolle2004algorithm},  primal-dual \cite{chambolleP11}, split Bregman \cite{GO}, and  the alternating direction method of multipliers (ADMM) \cite{boydPCPE11admm}.

	\subsection{$L_1/L_2$ on the gradient}\label{sect:review_con}

	We review a model of $L_1/L_2$ on the gradient in a constrained formulation  \cite{l1dl2},
	\begin{equation}\label{eq:grad_con}
	\min_{u}  \frac{\| \nabla u \|_1}{\| \nabla u \|_2} \quad \mathrm{s.t.} \ \  A u =f,
	\end{equation}
	which is referred to as $L_1/L_2$-con. Here  $\|\cdot\|_1$ and $ \|\cdot\|_2$ are defined by \eqref{equ:anisotropic} and \eqref{equ:norm_general},  respectively.
	%, for any vectorized input, i.e., $$\|(p, q)\|_1 = \sum_{i=1}^n\sum_{j=1}^m \Big(|p_{ij}|+|q_{ij}|\Big), \ \|(p, q)\|_2 = \sqrt{\sum_{i=1}^n\sum_{j=1}^m \Big(p_{ij}^2+q_{ij}^2\Big)},\ \|p\|_2 = \sqrt{\sum_{i=1}^n\sum_{j=1}^m p_{ij}^2}.$$
	We  apply the ADMM framework \cite{boydPCPE11admm}  to minimize \eqref{eq:grad_con} by rewriting it into an equivalent form
	\begin{equation}\label{eq:grad_con2}
	\min_{u, \h d, \h h }  \frac{\| \h d \|_1}{\| \h h \|_2} \quad \mathrm{s.t.} \ \  A u =f, \  \h d = \nabla u,\  \h h = \nabla u,
	\end{equation}
	with  two auxiliary variables $\h d$ and  $\h h$.
	Note that we denote $\h d$ and  $\h h$ in bold to indicate that they have two components corresponding to   $x$ and $y$ derivatives.
	The augmented Lagrangian for \eqref{eq:grad_con2} is given by 
	\begin{equation}\label{eq:grad_aug_1}
	\begin{split}
	\mathcal{L}(u,\h d, \h h; w, \h b_1, \h b_2) = \frac{\|\h d\|_1 }{\| \h h\|_2} + & \langle \lambda w, f-Au\rangle +\frac{\lambda}{2}\|  Au-f\|_2^2\\
	&+ \langle \rho_1 \h b_1,\nabla u -\h d\rangle + \frac{\rho_1}{2}\| \h d - \nabla u \|_2^2 \\
	&+ \langle \rho_2 \h b_2,\nabla u -\h h\rangle + \frac{\rho_2}{2}\| \h h - \nabla u \|_2^2,
	\end{split}
	\end{equation}
	where $w, \h b_1, \h b_2$ are Lagrange multipliers (or dual variables) and $\lambda, \rho_1, \rho_2$ are positive parameters.
	The ADMM iterations proceed as follows,
	\begin{equation}\label{ADMM_l1dl2_2blocks}
	\left\{\begin{array}{l}
	u^{(k+1)} =\arg\min\limits_u \mathcal{L}(u, \h d^{(k)},\h h^{(k)};w^{(k)}, \h b_1^{(k)},\h b_2^{(k)})\\
	\h d^{(k+1)} = \arg\min\limits_{\h d} \mathcal{L}(u^{(k+1)}, \h d, \h h^{(k)};w^{(k)}, \h b_1^{(k)},\h b_2^{(k)}) \\
	\h h^{(k+1)}=\arg\min\limits_{\h h} \mathcal{L} (u^{(k+1)}, \h d^{(k+1)},\h h; w^{(k)},\h b^{(k)}_1,\h b_2^{(k)})\\
	w^{(k+1)} = w^{(k)} +f- Au^{(k+1)}\\
	\h b_1^{(k+1)} = \h b_1^{(k)} + \nabla u^{(k+1)} - \h d^{(k+1)}\\
	\h b_2^{(k+1)} = \h  b_2^{(k)} + \nabla u^{(k+1)} - \h h^{(k+1)}.
	\end{array}\right.
	\end{equation}
	For more details, please refer to \cite{l1dl2} that presented a proof-of-concept example  when $A^TA$ and $\nabla^T\nabla$ can be simultaneously diagonalizable by fast Fourier transform (FFT). In this paper, the matrix $A$ corresponds to a projection matrix, where the inverse of $\lambda A^TA+(\rho_1+\rho_2)\nabla^T\nabla$ can not be computed via FFT.
	
As the splitting scheme \eqref{eq:grad_con2} involves two-block variables of $u$ and $(\h d, \h h)$, it is hard to establish the convergence of  \eqref{ADMM_l1dl2_2blocks}. 
% Tao's suggestion to remove: In particular, 
% the direct extension of ADMM into multi-block does not necessarily converge for convex problems \cite{3blocksfail}, not to mention the nonconvex minimization. 
To prove for the convergence of ADMM, the existing literature \cite{guo2017convergence,pang2018decomposition,wang2019global} requires some associated function (e.g. objective function, merit function, and augmented Lagrangian function) to be coercive, separable, or Lipschitz differentiable (on a certain domain), neither of which holds for the $L_1/L_2$ functional.

	\section{The proposed models}
	\label{sect:model}
	\setcounter{equation}{0} Here we consider an unconstrained formulation of $L_1/L_2$ in order to deal with   noisy data. As opposed to \eqref{eq:grad_con2}, we propose a different splitting scheme, under which we can establish the ADMM convergence.   We then discuss a variant in Section~\ref{sect:box} to incorporate a box constraint, which is reasonable for the CT reconstruction problems.

	\subsection{Unconstrained formulation}\label{sect:uncon}
	The unconstrained $L_1/L_2$ formulation is given by
	\begin{equation}\label{eq:grad_uncon}
	\min_{u}  \frac{\| \nabla u \|_1}{\| \nabla u \|_2} +\frac{\lambda}{2} \|A u -f\|_2^2,
	\end{equation}
	which  is referred to as  $L_1/L_2$-uncon.

	We design a specific splitting scheme that reformulates \eqref{eq:grad_uncon} into
	\begin{equation}\label{equ:split_model_uncon}
	\min_{u, \h h}  \frac{\| \nabla u \|_1}{\| \h h \|_2}+\frac{\lambda}{2} \|A u -f\|_2^2 \quad \mathrm{s.t.} \quad \h h = \nabla u.
	\end{equation}
	The corresponding augmented Lagrangian function is expressed as
	\begin{equation}\label{eq:AL4L1/L2uncon}
	\mathcal{L}_{\rm{uncon}}(u, \h h; \h b_2) =\frac{\| \nabla u \|_1}{\| \h h \|_2}+\frac{\lambda}{2} \|A u -f\|_2^2+\langle \rho_2 \h b_2,\nabla u -\h h\rangle + \frac{\rho_2}{2}\| \h h - \nabla u \|_2^2,
	\end{equation}
	with a dual variable $\h b_2$ and a positive parameter $\rho_2.$ The ADMM framework involves the following iterations,
	\begin{equation} \label{ADMML1overL2_uncon}
	\left\{\begin{array}{l}
	\h u^{(k+1)}=\arg\min_u \mathcal{L}_{\rm{uncon}}(u, \h h^{(k)}; \h b^{(k)}_2)\\
	\h h^{(k+1)}=\arg\min_{\h h} \mathcal{L}_{\rm{uncon}}(u^{(k+1)}, \h h; \h b^{(k)}_2)\\
	\h b_2^{(k+1)} = \h  b_2^{(k)} + \nabla u^{(k+1)} - \h h^{(k+1)}.
	\end{array}\right.
	\end{equation}
%	Notice that it has the same structure as   \eqref{ADMML1overL2}  in the constrained case, especially the same updates for $\h h$ and $\h b_2$. 
Same as in \cite{l1dl2},	the $\h h$-update  has a closed-form solution given by
	\begin{equation}\label{eq:update-h}
	\h h^{(k+1)} =
	\begin{cases}
	\tau^{(k)}  \h g^{(k)} &  \mbox{if } \h g^{(k)} \neq \h 0\\
	\h e^{(k)}	& \text{otherwise},
	\end{cases}
	\end{equation}
	where $\h g^{(k)}  = \nabla u^{(k+1)} + \h b_2^{(k)},$ $\h e^{(k)}$ is a random vector with its $L_2$ norm being $\sqrt[3]{\frac{\|\nabla u^{(k+1)}\|_1}{\rho_2}},$ and $\tau^{(k)} = \frac{1}{3} + \frac{1}{3}(C^{(k)} + \frac{1}{C^{(k)}})$ with
	\begin{equation*}\label{eq:update_root}
		C^{(k)} = \sqrt[3]{\frac{27D^{(k)} + 2 + \sqrt{(27D^{(k)}+2)^2 - 4}}{2} } \quad\mbox{and}\quad D^{(k)} = \frac{\|\nabla u^{(k+1)}\|_1}{\rho_2 \|\h g^{(k)}\|_2^3}.
	\end{equation*}

	 The $u$-subproblem can be expressed  as
	\begin{equation}\label{eq:uncon_u}
	\min_{u}\displaystyle{\frac{\| \h \nabla u \|_1}{\| \h h^{(k)} \|_2}+\frac{\lambda}{2} \|A u -f\|_2^2+ \frac{\rho_2}{2}\| \h h^{(k)} - \nabla u - \h b_2^{(k)}\|_2^2}. 
	\end{equation}
%	Again we adopt the ADMM scheme to solve for \eqref{eq:uncon_u}, which is similar to \eqref{ADMM_l1_argmin_con}. In fact, the only difference is by fixing $w_j=0$ for all $j$. The details are thus omitted and 
With $ \h h^{(k)} $ and $\h b_2^{(k)}$ fixed,  we can apply ADMM to find the optimal solution of \eqref{eq:uncon_u}. Specifically by introducing an auxiliary variable $\h d,$ we rewrite \eqref{eq:uncon_u} as
	\begin{eqnarray}\label{xsubinner}
	\min_{u,\h d}&\displaystyle{\frac{\| \h d \|_1}{\| \h h^{(k)} \|_2}+\frac{\lambda}{2} \|A u -f\|_2^2+ \frac{\rho_2}{2}\| \h h^{(k)} - \nabla u - \h b_2^{(k)}\|_2^2}
	\quad \mathrm{s.t.} \quad  \ \h d=\nabla u.
	\end{eqnarray}
	The augmented Lagrangian corresponding to \eqref{xsubinner} is given by
	\begin{equation*}
		\begin{split}
			\mathcal{L}_{\rm{uncon}}^{(k)}(u, \h d;\h b_1) =&\frac{\|\h  d \|_1}{\| \h h^{(k)} \|_2}+\frac{\lambda}{2}\| Au-f\|_2^2+ \frac{\rho_2}{2}\| \h h^{(k)} - \nabla u - \h b_2^{(k)}\|_2^2\\
			& + \langle \rho_1 \h b_1,\nabla u -\h d\rangle + \frac{\rho_1}{2}\| \h d - \nabla u \|_2^2,
		\end{split}
	\end{equation*}
	where $ \h b_1$ is a dual variable and $\lambda, \rho_1$ are positive parameters. Here we have $k$ in the superscript of $\mathcal L_{\rm{uncon}}$ to indicate that it is the Lagrangian for the $u$-subproblem in \eqref{ADMML1overL2_uncon} at the $k$-th iteration.
	The ADMM framework to minimize \eqref{xsubinner} leads to the following iterations,
	\begin{equation}\label{ADMM_l1_argmin_con}
	\left\{\begin{array}{l}
	u_{j+1} =\arg\min_u \mathcal{L}_{\rm{uncon}}^{(k)}(u, \h d_j;(\h b_1)_j)\\
	\h d_{j+1} = \arg\min_{\h d} \mathcal{L}_{\rm{uncon}}^{(k)}(u_{j+1}, \h d;(\h b_1)_j) \\
	(\h b_1)_{j+1} = (\h b_1)_{j} + \nabla u_{j+1} - \h d_{j+1},
	\end{array}\right.
	\end{equation}
	where the subscript $j$ represents the inner loop index, as opposed to
	the superscript $k$ for outer iterations  in \eqref{ADMML1overL2_uncon}.  Note that $ \mathcal{L}_{\rm{uncon}}^{(k)}(u, \h d; \h b_1)$ resembles the augmented Lagrangian
	$\mathcal{L}(u,\h d, \h h^{(k)}; w, \h b_1, \h b_2^{(k)})$ with $w=0$ defined in  \eqref{eq:grad_aug_1}, and hence \eqref{ADMML1overL2_uncon} with one iteration of \eqref{ADMM_l1_argmin_con} for the $u$-subproblem is equivalent to  the previous approach \cite{l1dl2}. If we can reach to the optimal solution of  the $u$-subproblem, the convergence can be guaranteed; see Section~\ref{sect:conv}.

	We then elaborate on how to solve the two subproblems in \eqref{ADMM_l1_argmin_con}.
	By taking the derivative of $\mathcal{L}_{\rm{uncon}}^{(k)}$ with respect to $u$, we obtain a closed-form solution, 
	\begin{equation}\label{ADMM_l1con_u}
	u_{j+1} = \Big(\lambda A^TA-(\rho_1+\rho_2)\triangle\Big)^{-1}\Big(\lambda A^T f  + \rho_1 \nabla^T(\h d_j -(\h b_1)_{j})  +\rho_2\nabla^T(\h h^{(k)}-\h b_2^{(k)})\Big),
	\end{equation}
where $\triangle =-\nabla^T\nabla$ denotes the Laplacian operator.
	For a general system matrix $A$ that can not be diagonalized by Fourier transform, we adopt the conjugate gradient (CG) descent iterations \cite{Optimization_2006book_Nocedal} to solve for \eqref{ADMM_l1con_u}. The $\h d$-subproblem in \eqref{ADMM_l1_argmin_con} has a closed-form solution, i.e.,
	\begin{equation}\label{ADMM_l1con_d}
	\h d_{j+1} = \mathbf{ shrink}\left(\nabla u_{j+1} + (\h b_1)_{j}, \frac{1}{\rho_1 \|\h h^{(k)} \|_2}\right),
	\end{equation}
	where  $\mathbf{shrink}(\h v, \mu) = \mathrm{sign}(\h v)\max\left\{|\h v|-\mu, 0\right\}.$

	We summarize  in \Cref{alg:l1dl2_unconst}  for minimizing the $L_1/L_2$-uncon model \eqref{eq:grad_uncon}. Admittedly, \Cref{alg:l1dl2_unconst} involves 3 levels of iterations:  outer/inner ADMM and CG for solving \eqref{ADMM_l1con_u}, which is not computationally appealing. An alternative is the linearized ADMM \cite{nien2014fast} so as to avoid the CG iterations, which will be explored in the future.

	\begin{algorithm}[t]
		\caption{The $L_1/L_2$ unconstrained minimization ($L_1/L_2$-uncon). }
		\label{alg:l1dl2_unconst}
		\begin{algorithmic}[1]
			\STATE{Input: projection matrix $A$  and observed data $f$ }
			\STATE{Parameters: $\rho_1, \rho_2, \lambda,\bar \epsilon \in \mathbb{R}^+$, and kMax, jMax$\in\mathbb Z^+$}
			\STATE{Initialize: $ \h h, \h b_1, \h b_2, \h d$,   and $k,j = 0$}
			
			\WHILE{$k < $ kMax or $|u^{(k)}-u^{(k-1)}|/|u^{(k)}| > \bar\epsilon$}
			\WHILE{$j < $ jMax or $|u_{j}-u_{j-1}|/|u_{j}| > \bar\epsilon$}
			\STATE{
				%			\begin{equation*}
				%				\begin{cases}
				%				\text{Update } \{\h x^{(k+1)}, \alpha^{(k+1)} \}\text{ by } \eqref{equ:a1} & \text{ for A1}\\
				%				\text{Update } \{\h x^{(k+1)}, \alpha^{(k+1)}\} \text{ by } \eqref{equ:a2} & \text{ for A2}\\
				%				
				%				\end{cases}
				%				\end{equation*}
				$u_{j+1} = (\lambda A^TA-(\rho_1+\rho_2)\triangle)^{-1}(\lambda A^T f + \rho_1 \nabla^T(\h d_j -(\h b_1)_{j})  $  \\ $ \quad \quad \quad  +\rho_2\nabla^T(\h h^{(k)}-\h b_2^{(k)}))$}
			\STATE{$\h d_{j+1} = \mathbf{ shrink}\left(\nabla u_{j+1} + (\h b_1)_{j}, \frac{1}{\rho_1 \|\h h^{(k)} \|_2}\right) $}
			\STATE{$(\h b_1)_{j+1} = (\h b_1)_{j}+ \nabla u_{j+1} - \h d_{j+1} $}
			\STATE{$j = j + 1$}
			\ENDWHILE
			\RETURN $u^{(k+1)} = u_j$
			\STATE{$\h h^{(k+1)} =
				\begin{cases}
				\tau^{(k)} \left( \nabla u^{(k+1)} + \h b_2^{(k)} \right) & \nabla u^{(k+1)} + \h b_2^{(k)} \neq 0\\
				\h e^{(k)}	& \nabla u^{(k+1)} + \h b_2^{(k)} = 0
				\end{cases} $}
			\STATE $\h b_2^{(k+1)} = \h  b_2^{(k)} + \nabla u^{(k+1)} - \h h^{(k+1)} $
			\STATE{$k = k+1$ and $j = 0$}
			\ENDWHILE
			\RETURN $\h u^\ast = \h u^{(k)}$ \end{algorithmic} 
	\end{algorithm}

	\subsection{Box constraint}\label{sect:box}
	It is reasonable to incorporate a box constraint for image processing applications \cite{chan2012multiplicative,kk_bound_const}, since pixel values are usually bounded by $[0,1]$ or $[0, 255]$. Specifically for CT, the pixel value has physical meanings and hence the bound can often be estimated in advance \cite{pric_CT,essential_ph_medimag}.  
	The box constraint is particularly helpful  for the $L_1/L_2$ model to prevent its divergence \cite{L1dL2_accelerated}. We add a general box constraint $u\in [c,d]$ to \eqref{eq:grad_uncon}, thus leading to
	\begin{eqnarray}\label{eq:grad_con_l1l2_box}
%	&&	\min_{u} \frac{\| \nabla u \|_1}{\| \nabla u \|_2} \quad \mathrm{s.t.} \quad Au=f, u\in [c,d] \\% \quad \quad\qquad \text{ for } L_1/L_2\text{-con-box} \\
	&&\min_{u}  \frac{\| \nabla u \|_1}{\| \nabla u \|_2}+ \frac{\lambda}{2}\| A u -f\|_2^2 \quad \mathrm{s.t.} \quad  u\in [c,d], \label{eq:grad_uncon_l1l2_box}
	% \quad \text{ for } L_1/L_2\text{-uncon-box} .
	\end{eqnarray}
	referred to as $L_1/L_2$-box.
	To derive an algorithm for solving the $L_1/L_2$-box model, we rewrite \eqref{eq:grad_con_l1l2_box} equivalently as
	\begin{equation}\label{equ:split_model_box}
	\min_{u,\h h}  \frac{\| \nabla u \|_1}{\| \h h \|_2} +\frac{\lambda}{2}\| A u -f\|_2^2 + \Pi_{[c,d]}(u) \quad \mathrm{s.t.} \quad  \h h = \nabla u,
	\end{equation}
	where $\Pi_{S}(t)$ is an indicator function enforcing $t$ into the feasible set $S$, i.e.,
	\begin{equation}\label{eq:indicator}
	\Pi_S(t) =
	\begin{cases}
	0	&	 \mbox{if } t \in S
	\\
	+\infty	&	\text{otherwise}.
	\end{cases}
	\end{equation}
	
	 The augmented Lagrangian function for \eqref{equ:split_model_box} can be expressed as
	 \begin{equation}\label{eq:Lboxvsuncon}
	     \mathcal{L}_{\rm box}(u,\h h; \h b_2) =  \mathcal{L}_{\rm uncon}(u,\h h; \h b_2)+\Pi_{[c,d]}(u).
	 \end{equation}
	  By using ADMM, 
	 we have the same update rules for $\h h$ and $\h b_2$ as in \eqref{ADMML1overL2_uncon}, while 
	 the $u$-subproblem is given by
	\begin{equation}\label{eq:u-update4box}
	u^{(k+1)} = \arg\min_u \frac{\| \nabla u \|_1}{\| \h h^{(k)} \|_2} +\frac{\lambda}{2}\| A u -f\|_2^2+   \frac{\rho_2}{2}\| \h h^{(k)} - \nabla u -\h b_2^{(k)}\|_2^2+\Pi_{[c,d]}(u).
	\end{equation}
	We introduce two variables, $\h d$ for the gradient and $v$ for the box constraint, thus  getting
	\begin{equation}\label{eq:u-update4box-split}
	\min_{u,\h d,v} \frac{\|\h d \|_1}{\| \h h^{(k)} \|_2} +\frac{\lambda}{2}\| A u -f\|_2^2+   \frac{\rho_2}{2}\| \h h^{(k)} - \nabla u -\h b_2^{(k)}\|_2^2 +\Pi_{[c,d]}(v) \quad \mathrm{s.t.} \quad  \h d = \nabla u, u = v.
	\end{equation}
	The  augmented Lagrangian corresponding to \eqref{eq:u-update4box-split} becomes
	\begin{equation}\label{eq:grad_aug_l1_box}
	\begin{split}
	\mathcal{L}_{\rm{box}}^{(k)}(u,\h d, v; \h b_1, e)=& \frac{\|\h d\|_1}{\|\h h^{(k)}\|_2}  +\frac{\rho_2}{2} \|\nabla u-\h h^{(k)} +\h b_2^{(k)} \|_2^2  + \Pi_{[c, d]}(v)  +\frac{\lambda}{2}\|Au - f\|_2^2\\
	%	& +\frac{\rho_1}{2} \|\nabla u-\h d+\h b_1 \|_2^2 +  \frac{\beta}{2} \|v-  u - e\|_2^2,\\
	& + \langle \rho_1 \h b_1,\nabla u -\h d\rangle + \frac{\rho_1}{2}\| \h d - \nabla u \|_2^2	 + \langle \beta e, u-v\rangle + \frac \beta 2 \|v-u\|_2^2,
	\end{split}
	\end{equation}
	where $\h b_1, e$ are dual variables and $\lambda, \rho_1, \beta $ are positive parameters.
	Similar to \eqref{ADMM_l1_argmin_con}, there is a closed-form solution of the $u$-subproblem,
% 	\begin{equation}\label{eq:u_con_sub_box}
% 	\begin{split}
% 	u_{j+1} &= \Big(\lambda A^TA-(\rho_1+\rho_2)\triangle\Big + \beta I\Big)^{-1}\Big(\lambda A^T f + \rho_1 \nabla^T(\h d_j -(\h b_1)_{j})  \\
% 	&  \quad  +\rho_2\nabla^T(\h h^{(k)}-\h b_2^{(k)})+ \beta (v^{(k)}-e^{(k)})\Big).
% 	\end{split}
% 	\end{equation}
	
\begin{equation}\label{eq:u_con_sub_box}
	\begin{split}
 	u_{j+1} &= \Big(\lambda A^TA+ (\rho_1+\rho_2) \triangle+\beta I\Big)^{-1}\Big(\lambda A^Tf +\rho_1 \nabla^T(\h d_j -(\h b_1)_{j}) \\
	& \quad + \rho_2 \nabla^T(\h h^{(k)} -\h b_2^{(k)})   + \beta(v^{(k)}-e^{(k)})\Big),
	\end{split}
	\end{equation}
	
	The update for $\h d$ is the same as \eqref{ADMM_l1con_d}, and we update $v$ by projecting it onto  $[c,d],$ i.e.,
$
		v_{j+1}=\min\left\{\max\{u_{j+1}+e_{j}, c\}, d\right\}.
$	
	%\begin{equation}\label{eq: u_uncon_sub_box}
	%	u_{j+1} = \big(\lambda A^TA-(\rho_1+\rho_2)\triangle\big)^{-1}\big(\lambda A^T (f+w_j) + \rho_1 \nabla^T(\h d_j -(\h b_1)_{j})   +\rho_2\nabla^T(\h h^{(k)}-\h b_2^{(k)})+ \beta (v^{(k)}-e^{(k)})\big)
	%\end{equation}
	%We also have a closed-form solution \cite{beck2017first} for $v$ given by 
	%\begin{equation*}
	%	v_{j+1}=\min\left\{\max(u_{j+1}+e_{j}, c), d\right\} 
	%\end{equation*}
	%The algorithm for the $L_1/L_2$-con with a box constrain can be deduced in a similar way. Basically, we update $u$ as
	%
	%\tb{and add the update of $w$  
	%\begin{equation}
	%	\label{eq:w}
	%	w_{j+1} = w_{j} + f - A u_{j+1}
	%\end{equation}
	%}
	The pseudo-code with the additional box constraint is summarized in \Cref{alg:l1dl2_box}.
	
	\begin{algorithm}[t]
		\caption{The $L_1/L_2$ minimization with a box constraint ($L_1/L_2$-box). }
		\label{alg:l1dl2_box}
		\begin{algorithmic}[1]
			\STATE{Input: projection matrix $A$, observed data $f$,  and a bound  $[c,d]$ for the original image}
			\STATE{Parameters: $\rho_1, \rho_2, \lambda,\beta,\bar\epsilon \in \mathbb{R}^+$, and kMax, jMax $\in\mathbb Z^+$}
			\STATE{Initialize: $ \h h, \h b_1, \h b_2, \h d, w=0, e$,   and $k,j = 0$}
			
			\WHILE{$k < $ kMax or $|u^{(k)}-u^{(k-1)}|/|u^{(k)}| > \bar\epsilon$}
			
			\WHILE{$j < $ jMax or $|u_{j}-u_{j-1}|/|u_{j}| > \bar\epsilon$}
			\STATE{$u_{j+1} = (\lambda A^TA-(\rho_1+\rho_2)\triangle + \beta I)^{-1}(\lambda A^T f + \rho_1 \nabla^T(\h d_j -(\h b_1)_{j})  $  \\ $ \quad \quad \quad  +\rho_2\nabla^T(\h h^{(k)}-\h b_2^{(k)})+\beta (v^{(k)}-e^{(k)}))$
			}
			\STATE{$\h d_{j+1} = \mathbf{ shrink}\left(\nabla u_{j+1} + (\h b_1)_{j}, \frac{1}{\rho_1 \|\h h^{(k)} \|_2}\right) $}
			\STATE{$v_{j+1}=\min\left\{\max\{u_{j+1}+e_{j}, c\}, d\right\} $ }
			\STATE{$(\h b_1)_{j+1} = (\h b_1)_{j}+ \nabla u_{j+1} - \h d_{j+1} $}
			
			\STATE{$e_{j+1} = e_{j} + u_{j+1} - v_{j+1}$}
			\STATE{$j = j + 1$}
			\ENDWHILE
			\RETURN $u^{(k+1)} = u_j$
			\STATE{$\h h^{(k+1)} =
				\begin{cases}
				\tau^{(k)} \left( \nabla u^{(k+1)} + \h b_2^{(k)} \right) & \nabla u^{(k+1)} + \h b_2^{(k)} \neq 0\\
				\h e^{(k)}	& \nabla u^{(k+1)} + \h b_2^{(k)} = 0
				\end{cases} $}
			\STATE $\h b_2^{(k+1)} = \h  b_2^{(k)} + \nabla u^{(k+1)} - \h h^{(k+1)} $
			%			\STATE{$w^{(k+1)} = w^{(k)} + f - A u^{(k+1)}\quad \quad $(for $L_1/L_2\text{-con}$ only) }
			\STATE{$k = k+1$ and $j = 0$}
			\ENDWHILE
			\RETURN $\h u^\ast = \h u^{(k)}$ \end{algorithmic} 
	\end{algorithm}
	
	\section{Convergence analysis}\label{sect:conv}

	We intend to establish the convergence of  Algorithms~\ref{alg:l1dl2_unconst}-\ref{alg:l1dl2_box}. We
	observe that the ADMM framework for both models share the same structure
	\begin{equation} \label{ADMML1overL2all}
	\left\{\begin{array}{l}
	u^{(k+1)}=\arg\min_u \mathcal{L}(u, \h h^{(k)};\h b^{(k)}_2)\\
	\h h^{(k+1)}=\arg\min_{\h h} \mathcal{L}(u^{(k+1)}, \h h;\h b^{(k)}_2)\\
	\h b_2^{(k+1)} = \h  b_2^{(k)} + \nabla u^{(k+1)} - \h h^{(k+1)},
	\end{array}\right.
	\end{equation}
	where $\mathcal L$ is either  $\mathcal{L}_{\rm{uncon}}$ or $\mathcal{L}_{\rm{box}}$. 
	We show the sequence generated by ADMM for  $L_1/L_2$-uncon either diverges due to unboundedness or has a convergent subsequence, while the sequence for $L_1/L_2$-box always has a convergent subsequence. 
	For this purpose, we introduce \Cref{lem43} for an upper bound of $\|\h b_2^{(k+1)}-\h b_2^{(k)}\|_2$ in terms of $\|u^{(k+1)}-u^{(k)}\|_2$ and $\|\h h^{(k+1)}-\h h^{(k)}\|_2$. 
	\Cref{lem:suff_decr} and \Cref{thm:stantionary_con} are standard in convergence analysis \cite{hong2016convergence,li2015global,wang2018convergences,wang2019global} to guarantee that  the augmented Lagrangian decreases sufficiently and the subgradient at each iteration is bounded by successive errors, respectively. The lemmas require the following three assumptions,
	\begin{enumerate}
		\item[A1]:  ${\cal N}(\nabla)\bigcap{\cal N}(A)=\{0\}$, 
		where $\cal N$ denotes the null space and $\nabla$ is defined in \eqref{eq:gradient}. 
		\item[A2]: The sequence $\{ u^{(k)}\}$ generated by \eqref{ADMML1overL2all} is bounded, then so is $\{\nabla u^{(k)}\}$ and
		  we denote $M=\sup_{k}\{\|\nabla u^{(k)}\|_1\}$. 
		\item[A3]:  The norm of $\{ \h h^{(k)}\}$ generated by \eqref{ADMML1overL2all} has a  lower bound, i.e., there exists a positive constant $\epsilon$ such that $\|\h h^{(k)}\|_2\geq \epsilon, \ \forall k$. 
	\end{enumerate}
	\begin{remark}
		The Assumption	A1 is standard in image processing \cite{chan2005image,louZOX15}. The Assumption A2 requires the boundedness of $\{u^{(k)}\},$ and hence the convergence results can be interpreted as the sequence either diverges (due to unboundedness) or converges to a critical point. 
		To make the $L_1/L_2$ regularization well-defined,
		we shall have  $\|\h h\|_2>0$. Certainly, $\|\h h\|_2>0$ does not imply a uniform lower bound of $\epsilon$, but we can redefine the divergence of an algorithm by including the case of $\|\h h^{(k)}\|_2< \epsilon$, which can be checked numerically with a pre-set value of $\epsilon$.
	\end{remark}
	
	Please refer to Appendix for the proofs of these lemmas, based on which we can establish the convergence in \Cref{theo2,theo:box}  for \Cref{alg:l1dl2_unconst,alg:l1dl2_box}, respectively. Furthermore, \Cref{theo2inr,theo3inr} extend the convergence analysis to the case when the $u$-subproblem in \eqref{ADMML1overL2all} can be solved inexactly.

	\begin{lemma}\label{lem43} Under the Assumptions A1 and A2, the sequence $\{u^{(k)},\h h^{(k)},\h b_2^{(k)}\}$  generated by \eqref{ADMML1overL2all}
		%if $\|\nabla u^{(k)}\|_2\leq M$ and $\|\h h^{(k)}\|_2\geq \epsilon$, 
		satisfies
		\begin{equation}\label{ineq:lemma}
		\left\|\h b_2^{(k+1)}-\h b_2^{(k)}\right\|_2^2\le\left(\frac{32 mn}{\rho_2^2\epsilon^{4}}\right)\left\|u^{(k+1)}-u^{(k)}\right\|_2^2+ \left(\frac{8M^2}{\rho_2^2\epsilon^6}\right)\left\|\h h^{(k+1)}-\h h^{(k)}\right\|_2^2.
		\end{equation}
	\end{lemma}

	%
	%Using \Cref{lem4.4}, we show in \Cref{lem:suff_decr} that the augmented Lagrangian has sufficient decreasing under some mild conditions.
	
	\begin{lemma}{(sufficient descent)}\label{lem:suff_decr}
		Under the Assumptions A1-A3 and a sufficiently large $\rho_2$,
		the sequence $\{u^{(k)},\h h^{(k)},\h b_2^{(k)}\}$  generated by 
		%Algorithm 3.1, 3.2 or 3.3
		\eqref{ADMML1overL2all}
		% . If $\|\nabla u^{(k)}\|_2\leq M, \|\h h\|_2\geq \epsilon,$ and ${\cal N}(\nabla)\bigcap{\cal N}(A)={0}$, then
		satisfies 
		\begin{equation}\label{ineq:lem_suff_decr}
		\mathcal{L}(u^{(k+1)}, \h h^{(k+1)};\h b_2^{(k+1)})\le \mathcal{L}(u^{(k)}, \h h^{(k)}; \h b_2^{(k)})-c_1\|u^{(k+1)}-u^{(k)}\|_2^2-c_2\|\h h^{(k+1)} - \h h^{(k)}\|_2^2,
		\end{equation}
		where  $c_1$ and $c_2$ are two positive constants.
	\end{lemma}

	\begin{lemma}{(subgradient bound)}\label{thm:stantionary_con}
		Under the Assumptions A1-A3 and a sufficiently large $\rho_2$, there exists  a vector $\bm\eta^{(k+1)} \in \partial {\cal L}(u^{(k+1)}, \h h^{(k+1)}; \h b^{(k+1)}_2)$ and a constant $\gamma>0$
		such that
		\begin{eqnarray}\label{eq2}\|\bm\eta^{(k+1)}\|_2^2\le \gamma \left( \|\h h^{(k+1)}-\h h^{(k)}\|_2^2 + \|\h b_2^{(k+1)}-\h b_2^{(k)}\|_2^2\right).
		\end{eqnarray}
		%where
		%$$ w^T:=\left(\begin{array}{c} u^T, \h h^T, \h b_2^T\end{array}\right).$$
		%\end{itemize}
	\end{lemma}

	\begin{theorem}\label{theo2}
		(convergence of $L_1/L_2$-uncon)
		Under the Assumptions A1-A3 and a sufficiently large $\rho_2$,  the sequence $\{u^{(k)}, \h h^{(k)}\}$  generated by \eqref{ADMML1overL2_uncon}
		% with sufficiently large $\rho_2$, if the assumption A1 is satisfied and $\{ u^{(k)}\}$ is bounded, 
		%If  the sequence $ \{u^{(k)}\}$  generated by \eqref{ADMML1overL2_uncon} is bounded, ${\cal N}(\nabla)\bigcap{\cal N}(A)={0},$ and $\rho_2$ is sufficiently large, 
		%then  the sequence  $\{u^{(k)},\h h^{(k)}\}$  
		has a subsequence convergent to a critical point of \eqref{equ:split_model_uncon}.
			\end{theorem}

		\begin{proof}
			We first show that if $\{u^{(k)}\}$ is bounded, then  $\{\h h^{(k)},\h b_2^{(k)}\}$ is also bounded. As	$\|u^{(k)}\|_2$ is bounded, so is $\|\nabla u^{(k)}\|_1$. 
		%To ease the convergence analysis, we  further assume a lower bound of $\|\h h^{(k)}\|_2$, i.e., $\|\h h^{(k)}\|_2\geq \epsilon, \forall k$ with sufficiently small $\epsilon>0$.
		It  follows from the Assumption A2 and the optimality condition for $\h b_2$ in \eqref{equ:b_2} that we have
		\[
		\|\h b_2^{(k)}\|_2=\left\|\frac{\|\nabla u^{(k)}\|_1}{\rho_2}
		\frac{\h h^{(k)}}{\|\h h^{(k)}\|^3}\right\|_2\leq\frac{\|\nabla u^{(k)}\|_1}{\rho_2\epsilon^2}.
		\]
		Therefore, $\{\h b_2^{(k)}\}$ is bounded and hence $\{\h h^{(k)}\}$ is also bounded due to the $\h h$-update \eqref{eq:update-h} and boundedness of $\nabla u$. 
	Then it follows from the Bolzano-Weierstrass Theorem that the sequence  $\{u^{(k)},\h h^{(k)},\h b_2^{(k)}\}$ has a  convergent subsequence,
		denoted by $(u^{(k_j)},\h h^{(k_j)},\h b_2^{(k_j)})\rightarrow (u^*,\h h^*,\h b_2^*),$ as $k_j\rightarrow \infty$.	
		In addition, we can estimate that
		\begin{equation*}
			\begin{split}
				& \mathcal{L}_{\rm{uncon}}(u^{(k)}, \h h^{(k)}; \h b_2^{(k)})\\ 
%				= & \frac{\| \nabla u^{(k)} \|_1}{\| \h h^{(k)} \|_2}+\Pi_{Au=f}(u^{(k)})+ \frac{\rho_2}{2}\| \h h^{(k)} - \nabla u^{(k)} \|_2^2+\langle \rho_2 \h b_2^{(k)},\nabla u^{(k)} -\h h^{(k)}\rangle\\
				= & \frac{\| \nabla u^{(k)} \|_1}{\| \h h^{(k)} \|_2}+\frac{\lambda}{2}\|Au-f\|_2^2 + \frac{\rho_2}{2}\| \h h^{(k)} - \nabla u^{(k)}-\h b_2 \|_2^2- \frac{\rho_2}{2}\|\h b^{(k)}_2\|_2^2\\
				\geq & \frac{\| \nabla u^{(k)} \|_1}{\| \h h^{(k)} \|_2} - \frac{\|\nabla u^{(k)}\|_1^2}{\rho_2\epsilon^4},
			\end{split}
		\end{equation*}
		which gives a lower bound of ${\cal L}_{\rm{uncon}}$ \tr{owing to the boundedness of $u^{(k)}$}.
		Therefore, we have $\mathcal L_{\rm{uncon}}(u^{(k)},\h h^{(k)},\h b_2^{(k)})$ converges due to its monotonic decreasing by \Cref{lem:suff_decr}. 	
		%  Then it follows from \Cref{lem4.4} that the gradient of $\frac a{\|\h h\|_2}$ has a Lipschitz constant of $\frac {4M}{\epsilon^3}.$
		% Due to the boundedness of $\|\nabla u^{(k)}\|,$ \Cref{lem43} holds, which leads to  \Cref{lem:suff_decr}. 	
		%\Cref{lem:suff_decr} shows that $\mathcal L_{\rm{con}}(u^{(k)},\h h^{(k)},\h b_2^{(k)})$ decreases monotonically and hence it converges. 
		
		We then  sum the inequality \eqref{ineq:lem_suff_decr}  from $k=0$ to $K$, thus getting
		\begin{eqnarray*}
			&&	\mathcal{L}_{\rm{uncon}}(u^{(K+1)}, \h h^{(K+1)}; \h b_2^{(K+1)})\\
			&\le& \mathcal{L}_{\rm{uncon}}(u^{(0)}, \h h^{(0)}; \h b_2^{(0)})-c_1\sum_{k=0}^K\|u^{(k+1)}-u^{(k)}\|_2^2-c_2\sum_{k=0}^K\|\h h^{(k+1)} - \h h^{(k)}\|_2^2.
		\end{eqnarray*}
		Let $K\rightarrow \infty,$ we have both $\sum_{k=0}^{\infty}\|u^{(k+1)}-u^{(k)}\|_2^2$ and $\sum_{k=0}^{\infty}\|\h h^{(k+1)} - \h h^{(k)}\|_2^2$ are  finite, indicating that
		$u^{(k)}-u^{(k+1)}\rightarrow 0$, $\h h^{(k)} - \h h^{(k+1)}\rightarrow 0$. Then by Lemma \ref{lem43}, we get $\h b_2^{(k)}-\h b_2^{(k+1)} \rightarrow 0 $.
		By  $(u^{(k_j)},\h h^{(k_j)},\h b_2^{(k_j)})\rightarrow (u^*,\h h^*,\h b_2^*),$ we have  $(u^{(k_j+1)},\h h^{(k_j+1)},\h b_2^{(k_j+1)})\rightarrow (u^*,\h h^*,\h b_2^*)$, and $\nabla u^*=\h h^*$ (by the update of $\h b_2$). Here, by Lemma \ref{thm:stantionary_con}, we have $\h 0\in \partial\mathcal L_{\rm{uncon}}(u^*,\h h^*, \h b_2^*)$ and hence $(u^*,\h h^*)$ is a critical point of \eqref{equ:split_model_uncon}.
		\end{proof}

	For the box  model \eqref{eq:grad_uncon_l1l2_box} with an explicit bounded assumption on $u$, we can prove that the ADMM framework has the same convergence results  as in \Cref{theo2} without the Assumption A2. The proof is thus omitted.
	
	\begin{theorem}\label{theo:box}
		%Suppose that Assumptions \ref{Ass2}-\ref{Ass1} holds, and $\rho_2>{\hat \rho}$.
		(convergence of $L_1/L_2$-box)
		Under the Assumptions A1, A3, and a sufficiently large $\rho_2$,  the sequence $\{u^{(k)}, \h h^{(k)}\}$  generated by
		\Cref{alg:l1dl2_box} 
		%with sufficiently large $\rho_2$. If 
		%	the assumption A1 is satisfied, 
		%	${\cal N}(\nabla)\bigcap{\cal N}(A)={0}$
		%then  the sequence 
		always has a subsequence  convergent to a critical point of   \eqref{eq:grad_uncon_l1l2_box}.
	\end{theorem}

\begin{theorem}\label{theo2inr} 
(convergence of inexact scheme in $L_1/L_2$-uncon) \tr{Under the Assumption A1}
	and a sufficiently large $\rho_2$,  one can solve the $u$-subproblem in \eqref{eq:uncon_u} within an error tolerance \tr{$\varepsilon_{k+1},$} i.e.,
		\begin{eqnarray} %\label{ADMML1overL2_unconin}
% \|\tilde{u}^{(k+1)}-\arg\min_u \mathcal{L}(u, \h h^{(k)}; \h b^{(k)}_2)\|_2^2\leq \varepsilon_k.
\|\tilde{u}^{(k+1)}-u^{(k+1)}\|_2^2\leq \tr{\varepsilon_{k+1}},
		\end{eqnarray}
	and the 	sequence  $\{\tilde{\h h}^{(k+1)},	\tilde{\h b}_2^{(k+1)}\}$ is
generated \tr{by an inexact ADMM scheme, i.e., 
\begin{equation} \label{ADMML1overL2allx}
	\left\{\begin{array}{l}
%	{\tilde u}^{(k+1)}\approx\arg\min_u \mathcal{L}(u, \tilde{\h h}^{(k)};\tilde{\h b}^{(k)}_2)\\
	\tilde{\h h}^{(k+1)}=\arg\min_{\h h} \mathcal{L}_{\rm{uncon}}({\tilde u}^{(k+1)}, \h h;\tilde{\h b}^{(k)}_2)\\
	\tilde{\h b}_2^{(k+1)} =\tilde{\h  b}_2^{(k)} + \nabla {\tilde u}^{(k+1)} - \tilde{\h h}^{(k+1)}.
	\end{array}\right.
	\end{equation}}
	If \tr{$\{\tilde u^{(k)}, \tilde{\h h}^{(k)}\}$ satisfies the Assumptions A2-A3
% is bounded, $\|\tilde{\h h}^{(k)}\|_2\geq \epsilon,$ 
	and
	}
	$\sum_k\varepsilon_k<+\infty$, then this 
 sequence  
%  $\{\tilde u^{(k)},\tilde{\h h}^{(k)}\}$ 
% 		has a subsequence convergent to a critical point  of \eqref{equ:split_model_uncon}.
has a subsequence convergent to a critical point  of
\eqref{equ:split_model_uncon}.
	\end{theorem}

\begin{proof}
\tr{
As $\tilde u^{(k)}$ is bounded, we denote $\tilde M:=\sup_{k}\|\nabla {\tilde u}^{k}\|_1$ and $\hat M:=\sup_{k}\{\|{\tilde u}^{(k)}\|_2\}$. 
Following the proof of \Cref{lem:suff_decr}, we have
\begin{equation*}
    \begin{split}
% \label{xx-sub-des}
&\mathcal{L}_{\rm{uncon}}({ u}^{(k+1)},\tilde{ \h h}^{(k)};\tilde{\h b}_2^{(k)})\le \mathcal{L}_{\rm{uncon}}({\tilde u}^{(k)}, \tilde{\h h}^{(k)}; \tilde{\h b}_2^{(k)})
-\frac{\sigma\lambda}{2}\|u^{(k+1)}-{\tilde u}^{(k)}\|_2^2\\
&\mathcal{L}_{\rm{uncon}}(\tilde u^{(k+1)}, \tilde{\h h}^{(k+1)};\tilde{\h b}^{(k)}_2)\leq \mathcal{L}_{\rm{uncon}}(\tilde u^{(k+1)}, \tilde{\h h}^{(k)};\tilde{\h b}_2^{(k)})  -\frac{\rho_2-3\tilde L}{2}\|\tilde{\h h}^{(k+1)}-\tilde{\h h}^{(k)}\|_2^2,
\end{split}
\end{equation*}
where $\tilde L = \frac {2\tilde M}{\epsilon^3}.$
 Simple calculations lead to
\begin{equation*}
\begin{split}
    \|u^{(k+1)}-{\tilde u}^{(k)}\|_2^2 &
=\| (u^{(k+1)}-{\tilde u}^{(k+1)})+ ({\tilde u}^{(k+1)}-{\tilde u}^{(k)})\|_2^2 \\
&\ge\|u^{(k+1)}-{\tilde u}^{(k+1)}\|^2+\|{\tilde u}^{(k+1)}-{\tilde u}^{(k)}\|^2  -2\|{\tilde u}^{(k+1)}-{\tilde u}^{(k)}\|\|u^{(k+1)}-{\tilde u}^{(k+1)}\|\\
&\ge\|u^{(k+1)}-{\tilde u}^{(k+1)}\|^2+\|{\tilde u}^{(k+1)}-{\tilde u}^{(k)}\|^2-4{\hat M}\varepsilon_{k+1}\\
&\ge\|{\tilde u}^{(k+1)}-{\tilde u}^{(k)}\|^2-4{\hat M}\varepsilon_{k+1}.
\end{split}
\end{equation*}
It follows from \Cref{lemma:lip} that 
\begin{equation*}
    \begin{split}
        \mathcal{L}_{\rm{uncon}}(\tilde u^{(k+1)}, \tilde{\h h}^{(k)};\tilde{ \h b}^{(k)}_2) & \le \mathcal{\tilde L}_{\rm{uncon}}(u^{(k+1)}, \tilde{\h h}^{(k)};\tilde{ \h b}_2^{(k)})+\frac{\tilde L}{2}\|\tilde u^{(k+1)}-u^{(k+1)}\|^2_2 \\
        &\leq \mathcal{L}_{\rm{uncon}}(u^{(k+1)}, \tilde{\h h}^{(k)}; \tilde{\h b}_2^{(k)})+\frac{\tilde L \varepsilon_{k+1}}{2}.
    \end{split}
\end{equation*}
% \begin{eqnarray*}\lefteqn{\|u^{(k+1)}-{\tilde u}^{(k)}\|_2^2}
% =\| (u^{(k+1)}-{\tilde u}^{(k+1)})+ ({\tilde u}^{(k+1)}-{\tilde u}^{(k)})\|_2^2\\
% &\ge&\|u^{(k+1)}-{\tilde u}^{(k+1)}\|_2^2+\|{\tilde u}^{(k+1)}-{\tilde u}^{(k)}\|_2^2-2\|{\tilde u}^{(k+1)}-{\tilde u}^{(k)}\|_2\|u^{(k+1)}-{\tilde u}^{(k+1)}\|_2\\
% &\ge&\|u^{(k+1)}-{\tilde u}^{(k+1)}\|^2+\|{\tilde u}^{(k+1)}-{\tilde u}^{(k)}\|^2-4{\hat M}\varepsilon_{k+1}\\
% &\ge&\|{\tilde u}^{(k+1)}-{\tilde u}^{(k)}\|^2-4{\hat M}\varepsilon_{k+1}.\end{eqnarray*}
As analogous to \Cref{lem43}, it holds
\begin{eqnarray*}
	\left\|\tilde{\h b}_2^{(k+1)}-\tilde{\h b}_2^{(k)}\right\|_2^2\le\kappa_1\left\|{\tilde u}^{(k+1)}-{\tilde u}^{(k)}\right\|_2^2+ \kappa_2\left\|\tilde{\h h}^{(k+1)}-\tilde{\h h}^{(k)}\right\|_2^2,
		\end{eqnarray*}
where $\kappa_1:=\frac{32 mn}{\rho_2^2\epsilon^{4}}$ and $\kappa_2:=\frac{8\tilde M^2}{\rho_2^2\epsilon^6}$.
By combining all the inequalities, we get
\begin{eqnarray} &&\mathcal{L}_{\rm{uncon}}({\tilde u}^{(k+1)},\tilde{\h h}^{(k+1)};\tilde{\h b}_2^{(k+1)})\notag\\&\le& \mathcal{L}_{\rm{uncon}}({\tilde u}^{(k)}, \tilde{\h h}^{(k)}; \tilde{\h b}_2^{(k)})-\tilde c_1\|{\tilde u}^{(k+1)}-{\tilde u}^{(k)}\|_2^2
-\tilde c_2\|\tilde{\h h}^{(k+1)} -\tilde{ \h h}^{(k)}\|_2^2+\tilde c_3\varepsilon_{k+1},\label{ineq:inexact} \end{eqnarray}
where ${\tilde c}_1:=\frac{\sigma\lambda}{2}-\frac{16mn}{\rho_2\epsilon^{4}}, {\tilde c}_2:=\frac{\rho_2-3L}{2}-\frac{4\tilde M^2}{\rho_2\epsilon^6},$ and $\tilde c_3:=2{\hat M}\sigma\lambda+\frac{\tilde L}{2}.$
We can choose a sufficiently large $\rho_2$ such that $\tilde c_1,\tilde c_2>0$.
Summing the inequality \eqref{ineq:inexact} with $k$ from 0 to $K$ and letting $K\rightarrow \infty,$  we obtain: $\sum_{k=0}^{\infty}\|{\tilde u}^{(k+1)}-{\tilde u}^{(k)}\|_2^2$ and $\sum_{k=0}^{\infty}\|\tilde{\h h}^{(k+1)} -\tilde{\h h}^{(k)}\|_2^2$ are  finite, since $\tilde c_3\sum_k \varepsilon_{k}<+\infty$ by assumption. The rest of the proof follows the same as \Cref{theo2},  thus omitted.}
\end{proof}

Similarly, we have the convergence of inexact scheme in  $L_1/L_2$-box under a restriction that $\tilde u^{(k+1)}$ should belong to the feasible set, i.e., $\tilde u^{(k+1)}\in [c,d]. $
\begin{theorem}\label{theo3inr} 
(convergence of inexact scheme in $L_1/L_2$-box)
\tr{Under the Assumption A1,}
	and a sufficiently large $\rho_2$,  one can solve the $u$-subproblem in \eqref{eq:u-update4box}  within an error tolerance $\varepsilon_k$ and feasible set, i.e.,
		\begin{eqnarray} %\label{ADMML1overL2_unconin}
% \|\tilde{u}^{(k+1)}-\arg\min_u \mathcal{L}(u, \h h^{(k)}; \h b^{(k)}_2)\|_2^2\leq \varepsilon_k.
\|\tilde{u}^{(k+1)}-u^{(k+1)}\|_2^2\leq \tr{\varepsilon_{k+1}} \quad\mbox{and}\quad \tilde u^{(k+1)}\in [c,d].
		\end{eqnarray}
	If
	$\sum_k\varepsilon_k<+\infty$ \tr{and  $\tilde{ \h h}^{(k)}$ satisfies the Assumption A3,} then the  
sequence  $\{\tilde u^{(k)},\tilde{\h h}^{(k)}\}$
		has a subsequence convergent to a critical point of   \eqref{eq:grad_uncon_l1l2_box}.
	\end{theorem}

\begin{remark}
 Theorems \ref{theo2}-\ref{theo3inr} are about subsequential convergence, which is weaker than global convergence, i.e., the entire sequence converges. If  the augmented Lagrangian $\mathcal L$ has the Kurdyka-\L ojasiewicz (KL) property \cite{bolte2007lojasiewicz}, global convergence can be shown in a similar way as \cite[Theorem 3.1]{guo2017convergence}. Unfortunately, the KL property is an open problem for the $L_1/L_2$ functional. On the other hand, it is true that \Cref{theo2inr,theo3inr} relax the accuracy of solving the $u$-subproblem within the tolerance $\varepsilon_k$ at the $k$-th iteration, but in practice we solve for a fixed number of iterations, under which the convergence remains open in the optimization literature.
    \end{remark}

% 	\begin{theorem}\label{theowhole2}
% (global convergence) Under the assumptions A1, A3 and a sufficiently large $\rho_2$,
% 		the sequence $\{u^{(k)},\h h^{(k)},\h b_2^{(k)}\}$  generated by
% 		\Cref{alg:l1dl2_box}.
%  Suppose that the merit function $\mathcal L$ is a KL function.
%  Then $\{u^{(k)},\h h^{(k)},\h b_2^{(k)}\}$ has finite length, i.e.,
% $$\sum_{k=1}^\infty(\|{\h u}^{k+1}-{\h u}^k\|+\|\h h^{k+1} -{\h h}^{k}\| +\|\h b_2^{k+1} -{\h b}_2^{k}\|)<\infty,$$
% and hence $\{u^{(k)},\h h^{(k)},\h b_2^{(k)}\}$ converges to a stationary point of (\ref{eq:grad_con_l1l2_box}).

% \end{theorem}

% \begin{proof}The proof is similar to \cite[Theorem 3.1]{guo2017convergence}, thus we omit here. \end{proof}

% If we solve the $\h u$-subproblem of (\ref{ADMML1overL2_uncon}) inexactly to satisfy some certain criterion, i.e.,

% where
% \begin{eqnarray}\label{inexact} \sum_k\varepsilon_k<+\infty. \end{eqnarray}
% 	Similarly to the proof to Theorem \ref{theo2}, we can establish the following theorem.

	\section{Experimental results}
	\label{sec:experiments}

		\begin{figure}[t]
		\begin{center}
			\begin{tabular}{ccc}
				%		Sinogram & $L_1$-grad (RE $= 0.72\%$) & $L_1/L_2$-grad (RE $= 0.04\%$)\\
				\includegraphics[width=0.22\textwidth]{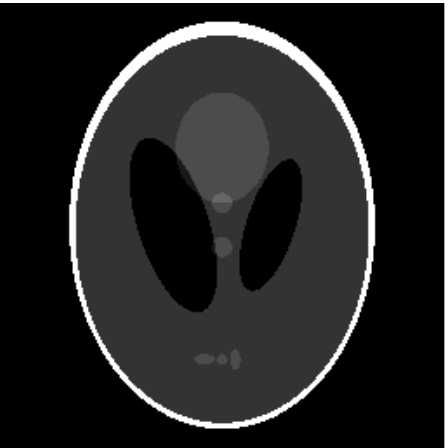} &
				\includegraphics[width=0.22\textwidth]{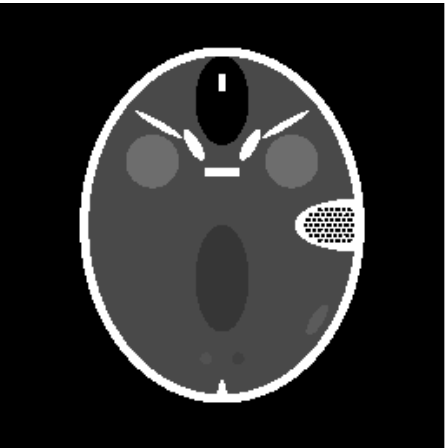} &
				\includegraphics[width=0.22\textwidth]{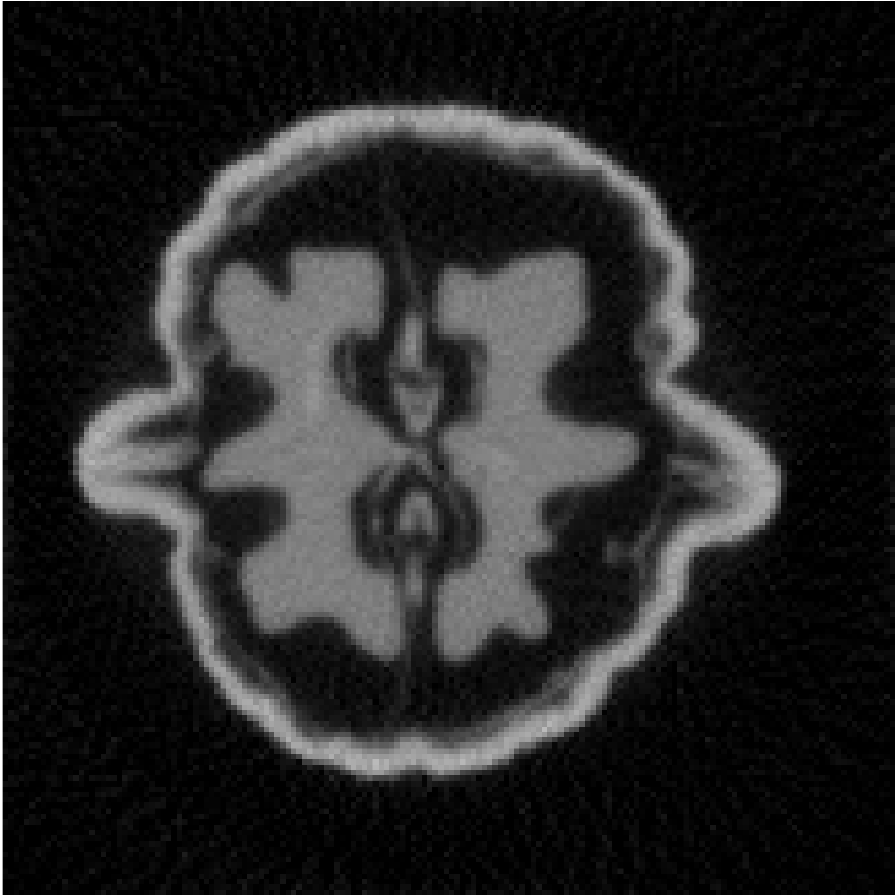} 
				\includegraphics[width=0.22\textwidth]{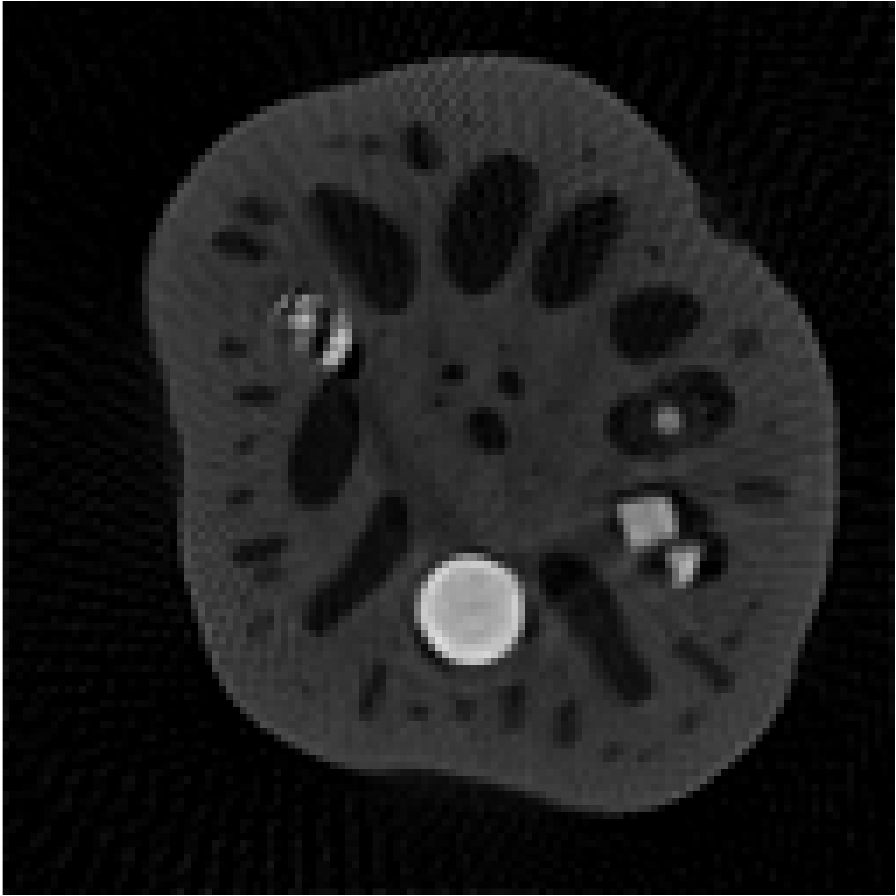} 
			\end{tabular}
		\end{center}
		\caption{Ground truth of Shepp-Logan (SL) phantom and FORBILD (FB) head phantom with the gray scale window  of $[0, \ 1]$ and  $[1.03, \   1.10], $ respectively. The last two  are reference images of a walnut and a lotus  reconstructed by using the complete projection data with the gray scale window  of $[0,\ 0.6]$.} \label{fig:SL_FB}
	\end{figure}
	
	\begin{figure}
	    \centering
	    	\begin{tabular}{cc}
	    	\includegraphics[width=0.3\textwidth]{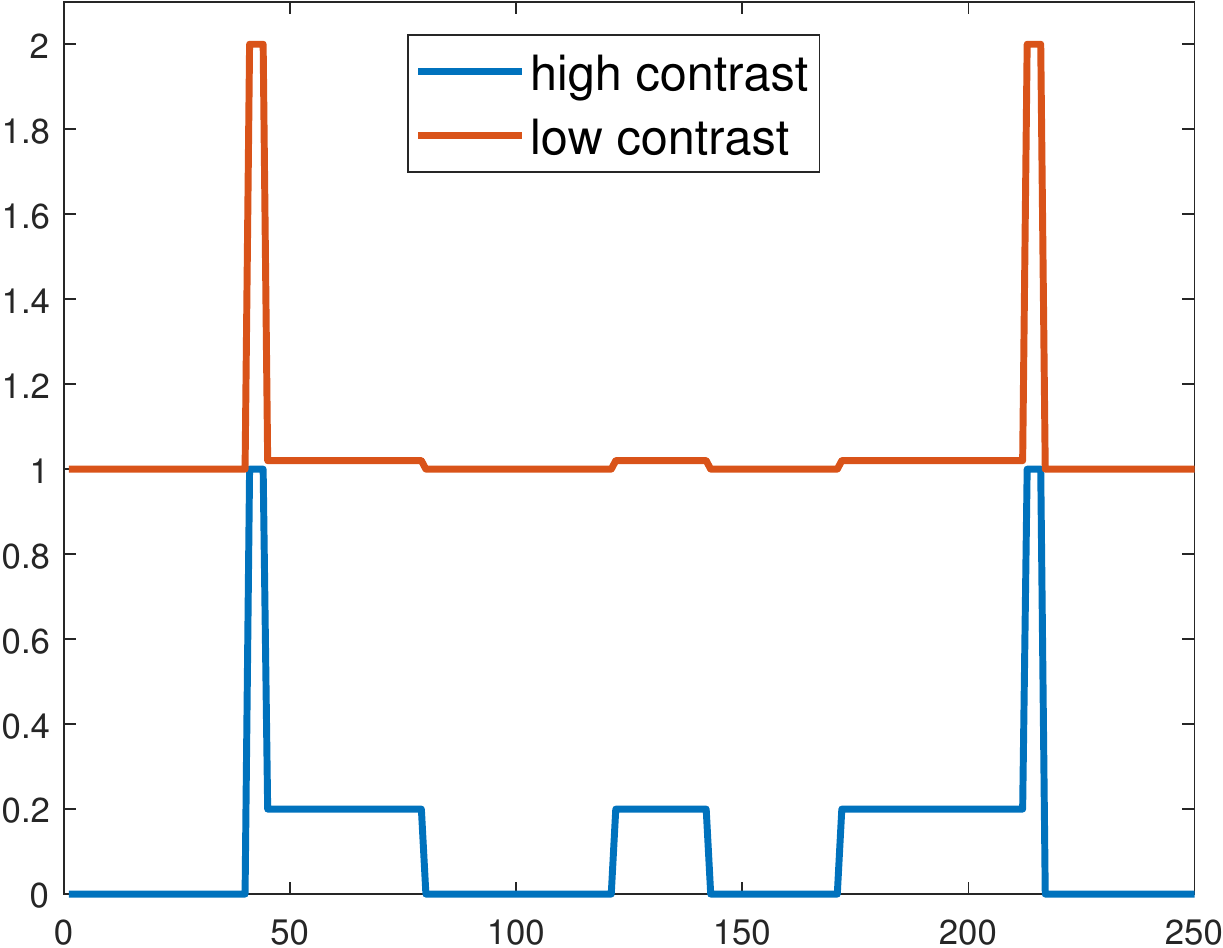} &
	    \includegraphics[width=0.3\textwidth]{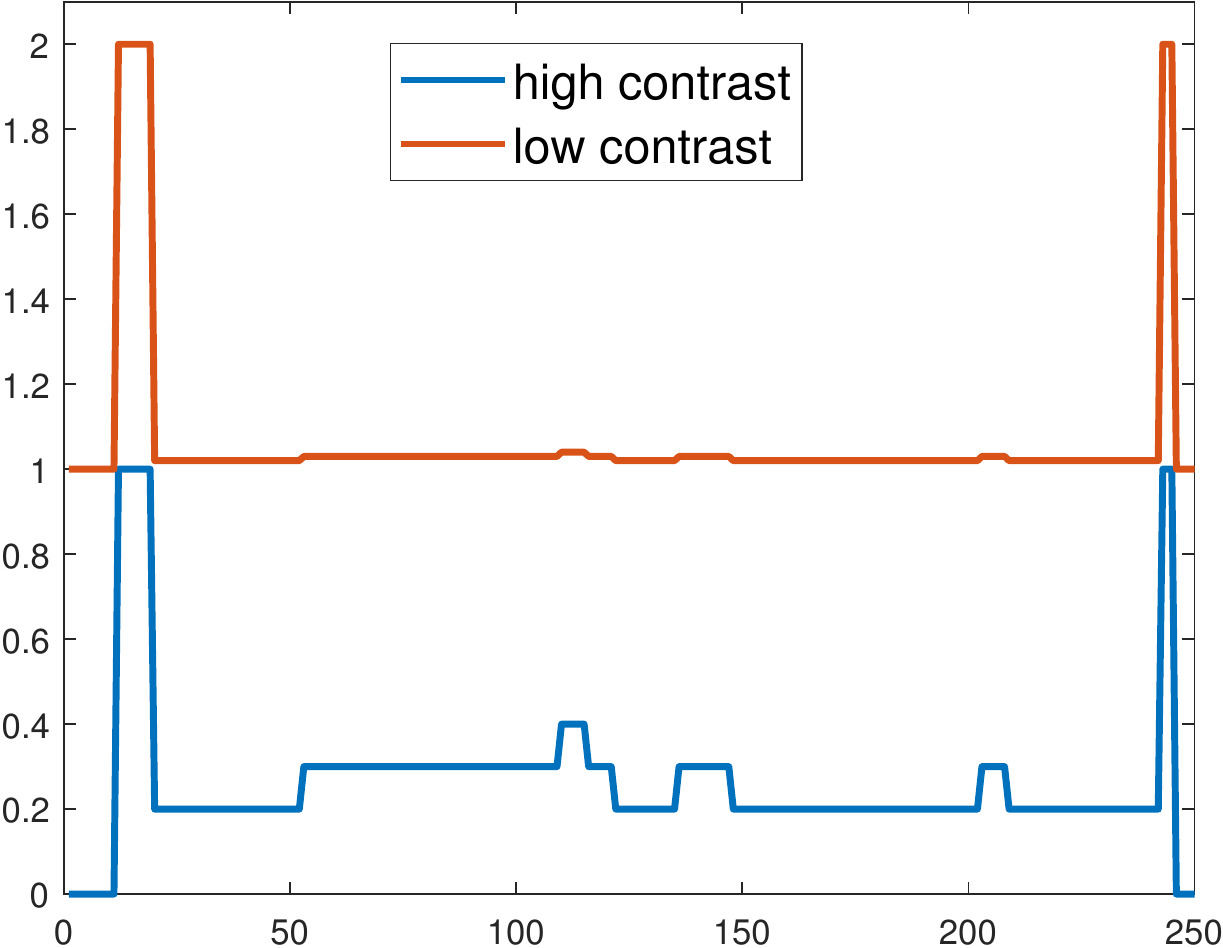}
	    			\end{tabular}
	    \caption{\tr{Horizontal (left) and vertical (right) profiles of two SL phantoms with low contrast and high contrast, the latter of which is used in this paper. }}
	    \label{fig:HV_SL}
	\end{figure}
	
We carry out extensive experiments to demonstrate the performance of the proposed approaches in comparison to the state-of-the-art. We test on two  phantoms of Shepp-Logan  (SL) by the MATLAB command {\tt phantom} and FORBILD (FB) \cite{FB_ph} as well as two experimental datasets of a walnut \cite{walnut_ct} and a lotus \cite{lotusdata}, all shown in \Cref{fig:SL_FB}. \tr{Note that the MATLAB SL phantom has a higher contrast than  the original one presented in \cite{SL_ph}, as illustrated in  \cref{fig:HV_SL} by the horizontal and vertical profiles.}	The reference images for the experimental data are reconstructed from  a complete scanning via the  Tikhonov regularization (MATLAB function is provided in \cite{lotusdata,walnut_ct}).
% , they are  open-access datasets of tomographic x-ray data . The packages contain  for CT reconstruction via , which we use to obtain the reference images by 
As the FB phantom has a very low image contrast, we display it with the grayscale window of [1.03, 1.10] in order  to reveal its structures.  Using the SL phantom, we  discuss some computational aspects of the proposed algorithms in \Cref{sect:exp_alg}. We then present numerical results on synthetic data (SL and FB) in \Cref{sect:exp_syn} and  experimental data (walnut and lotus) in \Cref{sect:exp_real}. 
	All the numerical experiments are conducted on a  desktop with CPU (Intel i7-5930K, 3.50 GHz) and MATLAB 9.7 (R2019b). 
	
 %These two images are discretized 
	%{\color{red} [Image is scaled to [0,1]?]}
	%For simplicity, we scale the pixel value  to $[0,1]$ for both phantoms and rescale the solution back after the computation. 

	To synthesize the limited-angle CT projection data, we discretize both SL and FB phantoms  at a resolution of $256\times256$. The forward operator $A$ is generated as the discrete Radon transform with the same resolution as the digital phantoms. We  use the IR and AIR toolbox \cite{Nagy_toolIR,hansen2012air} to simulate  \textit{parallel beam} and \textit{fan beam} for the CT scanning. Both settings are sampled at $ \theta_{\rm{Max}}/30$ over a range of   $\theta_{\rm{Max}},$ resulting in a sub-sampled data of size $362\times31$. 	We use the same number of projections when we vary ranges of projection angles.  
Note that complete scanning ranges for parallel beam and fan beam are  $180^\circ$ and $360^\circ,$ respectively. Therefore, fan beam is more challenging  to reconstruct than parallel beam with the same value of $\theta_{\rm{Max}}$.   
 We then add either Gaussian noise or Poisson noise to the projected data. The  Gaussian noise follows the zero mean Gaussian distribution with a standard deviation set by a noise level multiplying the maximum intensity of the projected data. We consider two Gaussian noise levels: 0.5\% and 0.1\%. \tr{A more realistic noise distribution for CT data is that the data has the Poisson distribution with the mean $I_0\exp(-f)$,
 where 
 $I_0$ denotes 	the number of incident x-ray photons and 
 $f$ is the noise-free sinogram. } We consider two Poisson noise levels: $I_0=10^4$ and $10^5$. The larger the value of $I_0$ is, the higher signal-to-noise ratio is for the measured data. 
%  , and hence less ill-posed the reconstruction problem is.
%  \yl{check this sentence}

	%\cw{[I am not sure whether it should be called as built-in function as it is simply use pcg for CT reconstruction with the Tikhonov term. ]}

% 	 \yl{[walnut and lotus have same dimension? maybe move these details to 5.2. do we need the next sentence:]} 
% 	 \cw{They are in the different dimension. Hence I only keep the dimensional information here and remove the next sentence. }
% 	 The testing data is a complete scanning with low-resolution (120-projection) sinogram sampled at every $3^\circ$.
% When we perform the limited-angle CT reconstruction, we take partial data from $f.$ As it shall contain noise, we do not add any additional noise.
% }

	%For the evaluation purpose, the image is downsampled \tb{[why downsample?]} \cw{[just want to compare with the restored image in sparse gradient]} into a $164\times 164$ matrix which corresponds the dimension of measure matrix $A$.  Here the measurement  and .  The X-ray samples at  and has $120$ projections in a complete scanning. \tb{[how to get the projection data?]} \cw{[t is provided in the dataset. I just explain the it is a complete scanning. ]} We take the projection data from ... without adding any noise by ourselves.

	We evaluate the performance in terms of the root mean squared error (RMSE) and  the overall structural similarity index (SSIM) \cite{wangBSS04}. RMSE is defined as 
	\begin{equation*}
		\text{RMSE}(u^\ast,\tilde{u}):= \frac{\|u^\ast-\tilde{u}\|_2}{N_{\rm pixel}},
	\end{equation*} 
	where $u^\ast$ is the restored image, $\tilde{u}$ is the ground truth, and $N_{\rm pixel}$ is the total number of pixels.  SSIM is the mean of local  similarity indices, 
	\begin{equation*}
		\text{SSIM}(u^\ast,\tilde{u}):= \frac{1}{N} \sum_{i=1}^{N} \text{ssim}(x_i,y_i),
	\end{equation*} 
	where  $x_i, y_i$ correspond to the $i$-th $8\times 8$  windows for $u^\ast$ and $\tilde u$, respectively, and $N$ is the number of such windows.  Note that $N\neq N_{pixel}$, if we do not consider zero-padded pixels along the edges. The local similarity index is defined as 
	\begin{equation*}
		\text{ssim}(x,y):= \frac{(2\mu_x\mu_y+c_1)(2\sigma_{xy} +c_2)}{(\mu_x^2+\mu_y^2+c_1)(\sigma^2_x+\sigma_y^2+c_2)}, 
	\end{equation*}
	where the averages/variances of $x,y$ are denoted as $\mu_x/\sigma_x^2$ and $\mu_y/\sigma_y^2$, respectively. %The corresponding variances are $\sigma_x^2$ and $\sigma^2_y$.
	 Here, $c_1$ and $c_2$ are two fixed constants to stabilize the division with weak denominator, which are set to be $c_1=c_2=0.05. $

%	 Lastly, PSNR is defined as
%	\begin{equation}\label{PSNR}
%	\hbox{PSNR}(u^\ast,\tilde{u}):=10\log_{10}\frac{mnR^2}{\|u^*-\tilde u\|_2^2},
%	\end{equation}
%	where $R$ is the maximum peak value of the original image $\tilde u\in\mathbb{R}^{m\times n}$. 
	%{\color{red} [so R = 1?]}\tb{\it [I think we don't need to specify the value as 1, since the  CT number has physics meaning. ]}

		We compare the proposed $L_1/L_2$ model  with a clinical standard approach of SART \cite{SART}, the TV model \eqref{eq:grad_con_l1},  referred to as $L_1$, as well as two  nonconvex regularizations: $L_p$ for $p=0.5$  and $L_1$-$L_2$  \cite{louZOX15} on the gradient. To solve for the $L_p$ model, we replace the soft shrinkage   by the proximal operator corresponding to $L_p,$ derived in \cite{Xu2012} and apply the same ADMM framework as the $L_1$ minimization. As for $L_1$-$L_2,$ we modify the MATLAB package provided by the authors \cite{louZOX15} for the CT reconstruction problem. All these regularization methods are solved in a constrained formulation with the same box constraint to make a fair comparison. We pose the box constraints:  $[0, 1]$ for SL and $[0, 1.8]$ for FB, since  we know the upper/lower bounds of the ground-truth images.  As for the experimental images, we set the box constraint as $[0, 0.5],$ which is estimated from the reference images.
% 		\yl{[in response letter it is up to 0.6?]} 
The initial condition of $u$ is chosen to be a zero vector for all the methods.
		We set the maximum iteration in the inner loop and outer loop for both $L_1/L_2$ and $L_1$-$L_2$ as 5 and 300, respectively, while the maximum iteration of $L_1$ and $L_p$ is 500.    The (outer) stopping criterion is $\frac{\|u^{(k)}-u^{(k-1)}\|_2}{\|u^{(k)}\|_2}\leq 10^{-5}.$ As for the other parameters in $L_1/L_2,$ we set $\rho_1=\rho_2=\rho$ and find the optimal combination among the candidate set of $\lambda\in\{10^{-3},10^{-2},10^{-1},1 \}$ and $\rho,\beta \in\{0.1, 1, 10\}$ that gives the lowest RMSE. We tune parameters at each noise level for every testing dataset. In a similar way, we  tune the parameters individually for $L_1, $ $L_p$,  and $L_1$-$L_2$.  
		
		\subsection{Algorithm behavior}\label{sect:exp_alg}
		
	\begin{figure}[t]
		\begin{center}
			\begin{tabular}{cc}
%				Noiseless & Noisy\\
%				\includegraphics[width=0.4\textwidth]{fig/conv_obj_noiseless} &
				\includegraphics[width=0.4\textwidth]{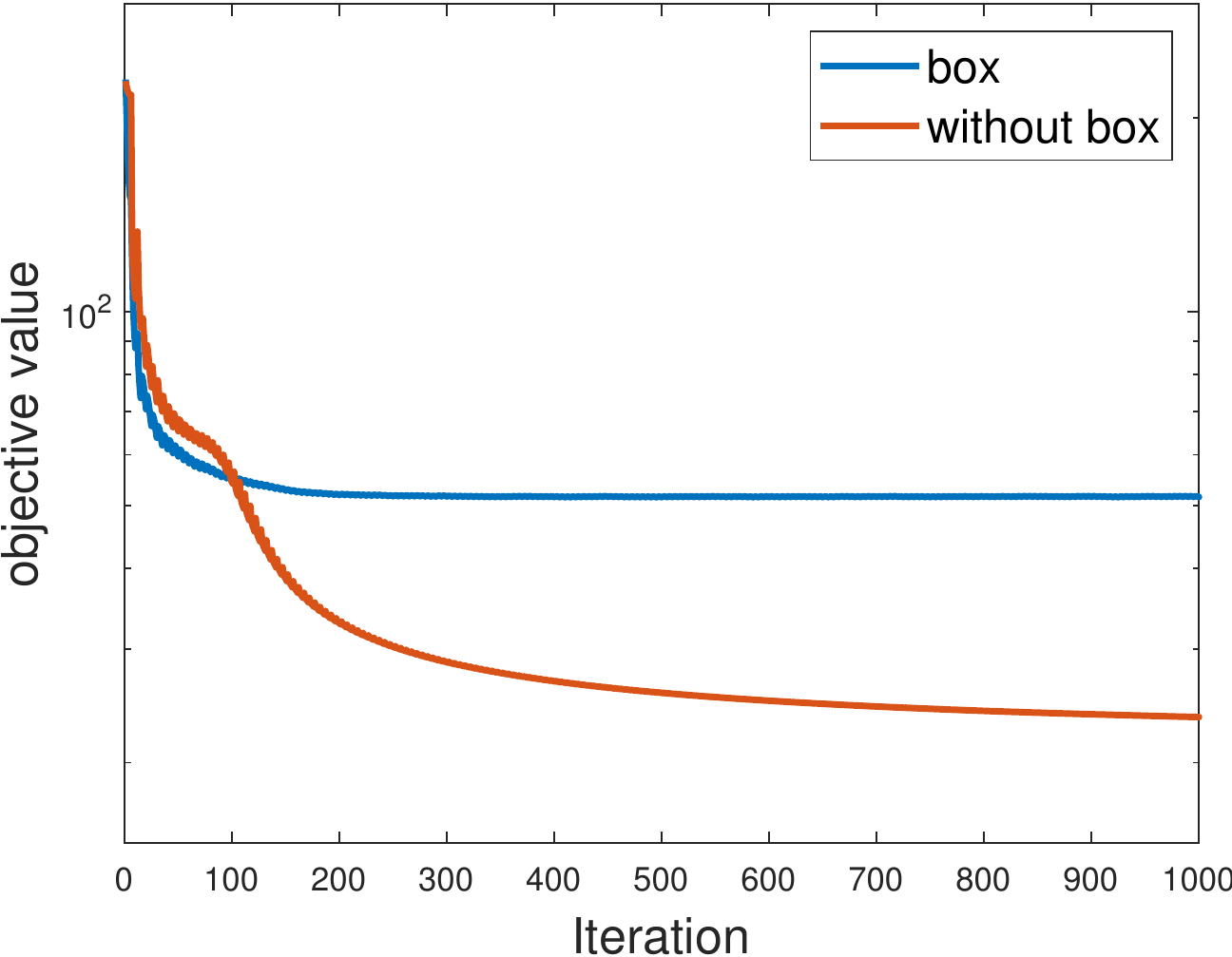}  &
				\includegraphics[width=0.4\textwidth]{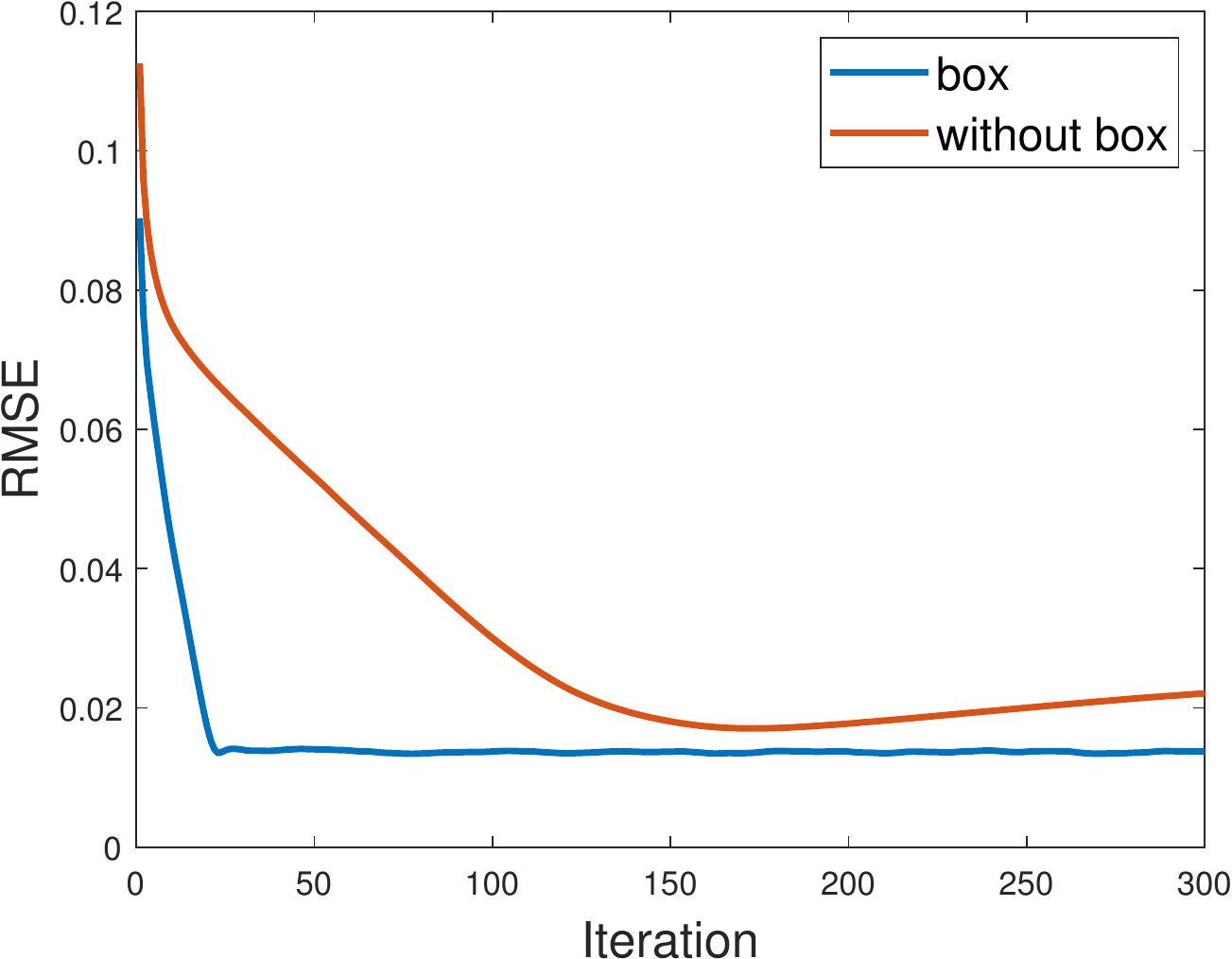} 		\end{tabular}
		\end{center}
		\caption{The effects of the box constraint in terms of the objective value (left) and RMSE (right). }\label{fig:box}
	\end{figure}

	In this section, we discuss computational  aspects of the proposed algorithms.	We first analyze the influence of the box constraint  on the reconstruction results. The analysis is based on the SL phantom from parallel beam CT projection data with the scanning range of  $135^\circ$ subject to Gaussian noise of 0.5\%.
	The fidelity of the CT reconstruction and the convergence are assessed in terms of objective values and RMSE$(u^{(k)}, \tilde u)$ versus outer iteration counter  $k$. 
	In  \Cref{fig:box}, we present algorithmic behaviors of the box constraint on the unconstrained model. 
%	Note that the objective function  is $\frac{\| \nabla u \|_1}{\| \nabla u \|_2}+ \frac{\lambda}{2}\| A u -f\|_2^2$. 
Here we set jMax to be 5 (we will discuss the effects of inner iteration number shortly.) We plot both inner and outer iterations in \Cref{fig:box}, showing that the proposed algorithms with and without the box constraint are convergent, as the objective functions decrease.
	On the other hand, the box constraint yields smaller RMSE  compared to the one without box. Moreover, the box constraint helps to avoid  local minimizers, as the RMSE of the algorithm without box increases and the objective function keeps going down.  Therefore, the box constraint plays an important role in the success of our approach for the CT reconstruction.
	%For the rest of the experiments, we only consider the algorithms with a box constraint. 

			\begin{figure}[t]
		\begin{center}
			\begin{tabular}{cc}
%				Noiseless & Noisy\\
%				\includegraphics[width=0.4\textwidth]{fig/conv_obj_noiseless} &
				\includegraphics[width=0.4\textwidth]{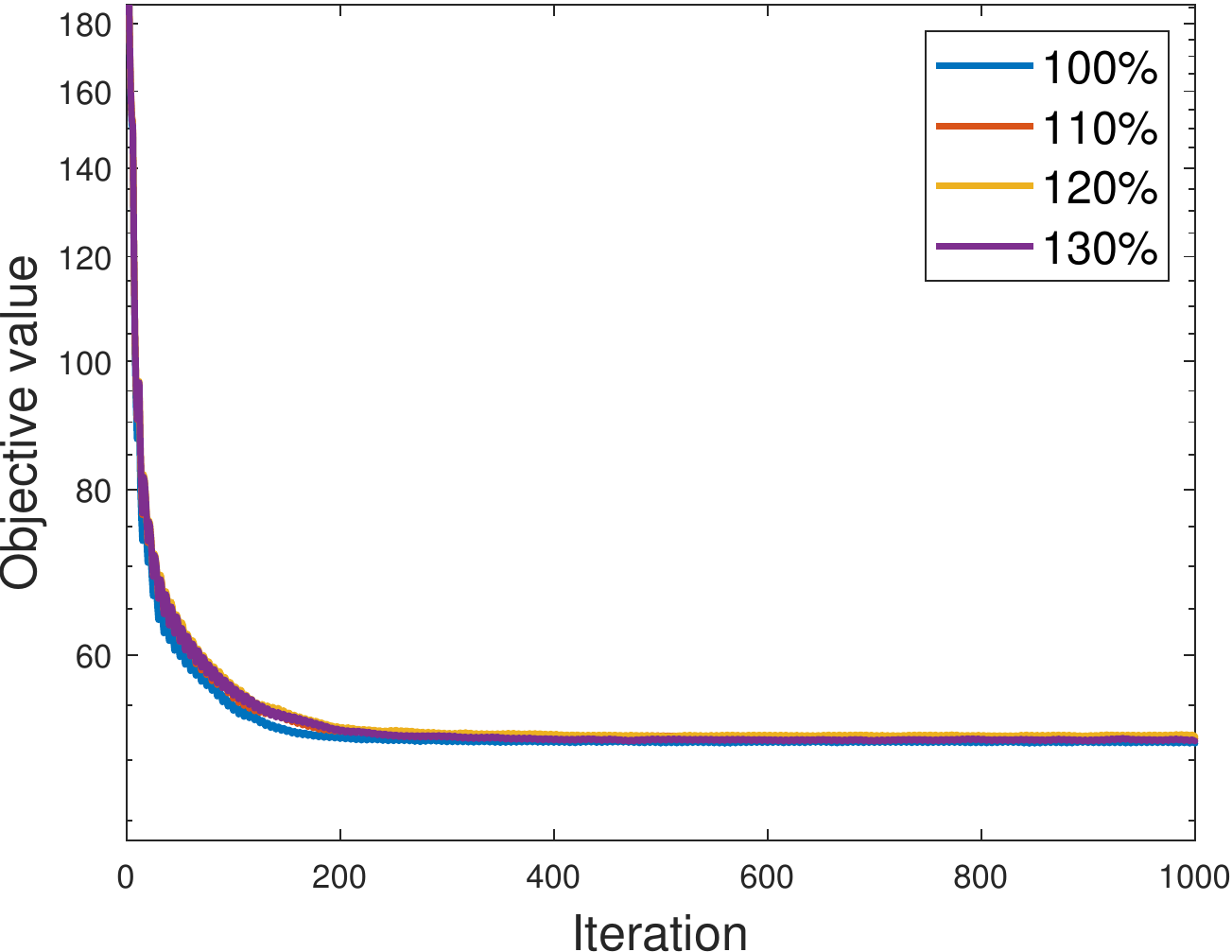}  &
				\includegraphics[width=0.4\textwidth]{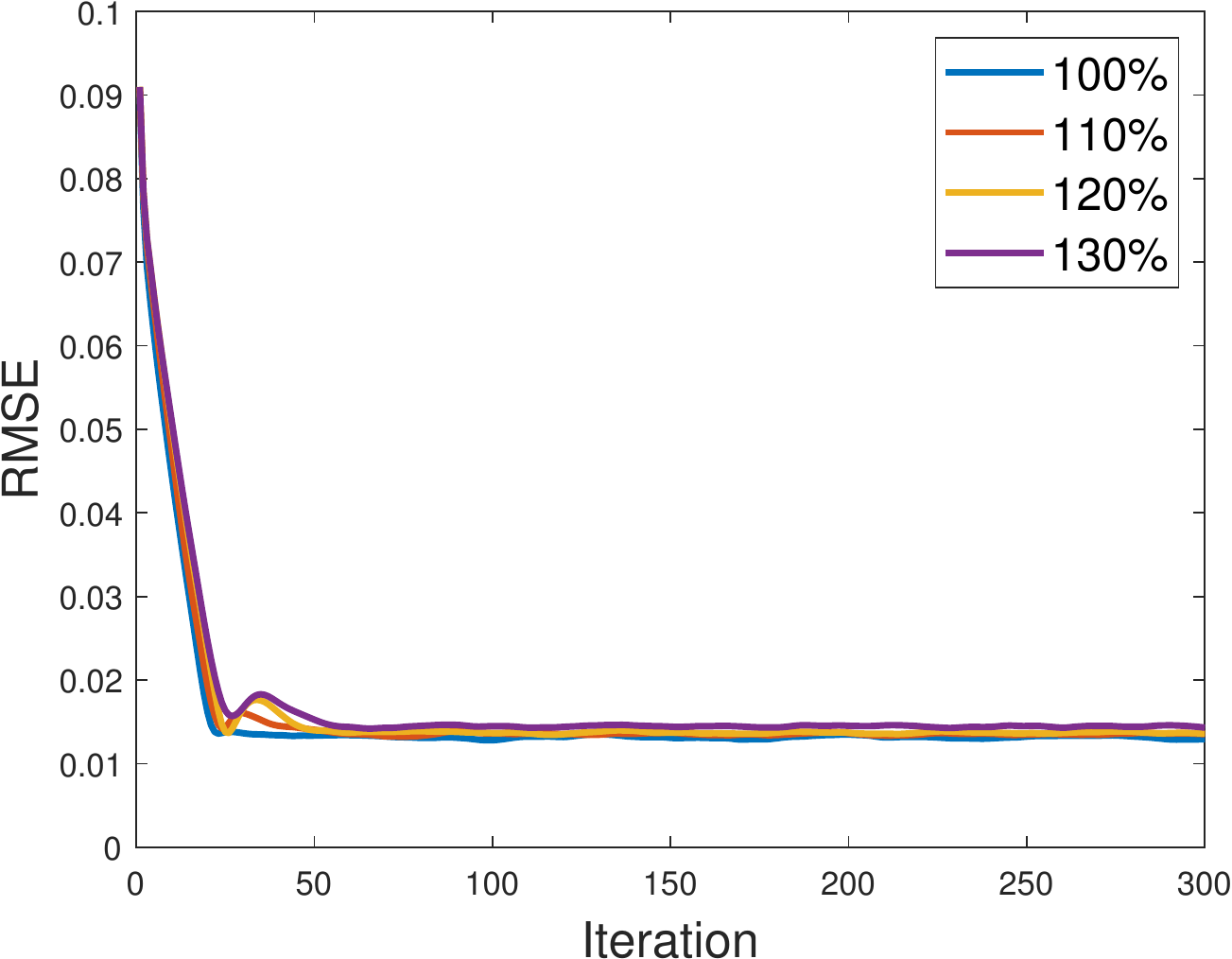} 		\end{tabular}
		\end{center}
		\caption{The effects of the upper bound of the box constraint in terms of the objective value (left) and RMSE (right).} \label{fig:box_upper}
	\end{figure}

	We then discuss the effect of upper bound of the box constraint, i.e., $d$, on the CT reconstruction performance. Again, we consider the SL phantom from paralleled beam CT projection with the scanning range of $135^\circ$ subject to a noise level of $0.5\%$. We compare the oracle upper bound (100\%) with relaxed bounds (110\%, 120\%, and 130\%). In \Cref{fig:box_upper}, we plot the objective values and RMSE with respect to iteration numbers. All the curves of objective values in different $d$ are almost the same, while the RMSE shows the accuracy is  slightly different. \Cref{fig:box_upper} demonstrates that the proposed method is  insensitive to the upper bound of the box constraint.

Finally,	we  discuss the influence of jMax on the sparse recovery performance. 
	Fixing the maximum outer iterations as 300, we examine the results of jMax$= 1,3,5,$ and 10.
	In \Cref{fig:effect_inner}, we plot the objective values and RMSE with respect to iterations (counting both inner and outer loops).   The objective function with only one inner iteration does not decrease as much as the ones with more inner iterations. RMSE reaches a lower value by fewer outer iterations when using larger jMax. 
	 Following \Cref{fig:effect_inner}, 
	we set  jMax to be  5 throughout the experiments. 
	
	\begin{figure}[t]
		\begin{center}
			\begin{tabular}{cc}
				%		Sinogram & $L_1$-grad (RE $= 0.72\%$) & $L_1/L_2$-grad (RE $= 0.04\%$)\\
%				Noiseless & Noisy\\
%				\includegraphics[width=0.4\textwidth]{fig/effect_in_obj_noiseless} &
				\includegraphics[width=0.4\textwidth]{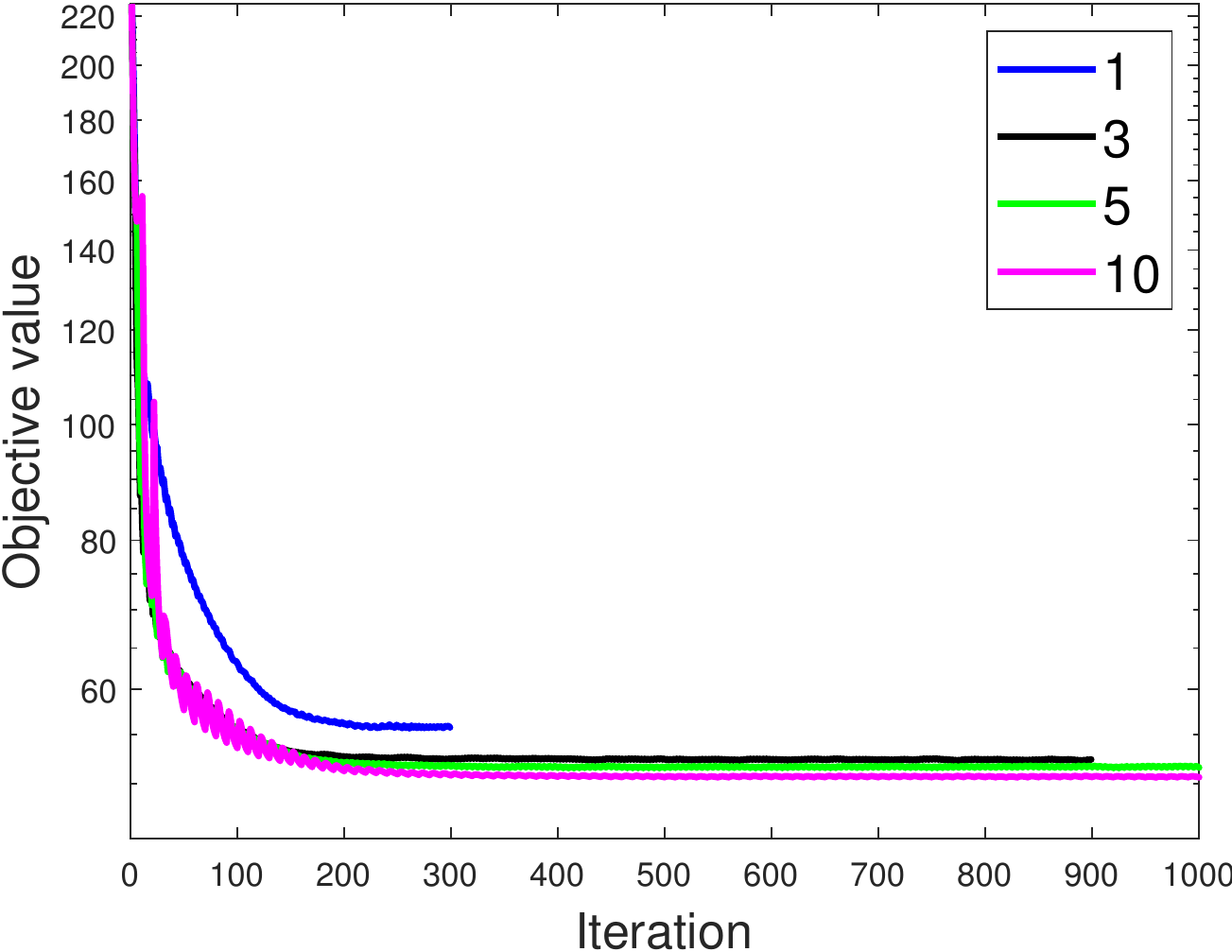}  &
				\includegraphics[width=0.4\textwidth]{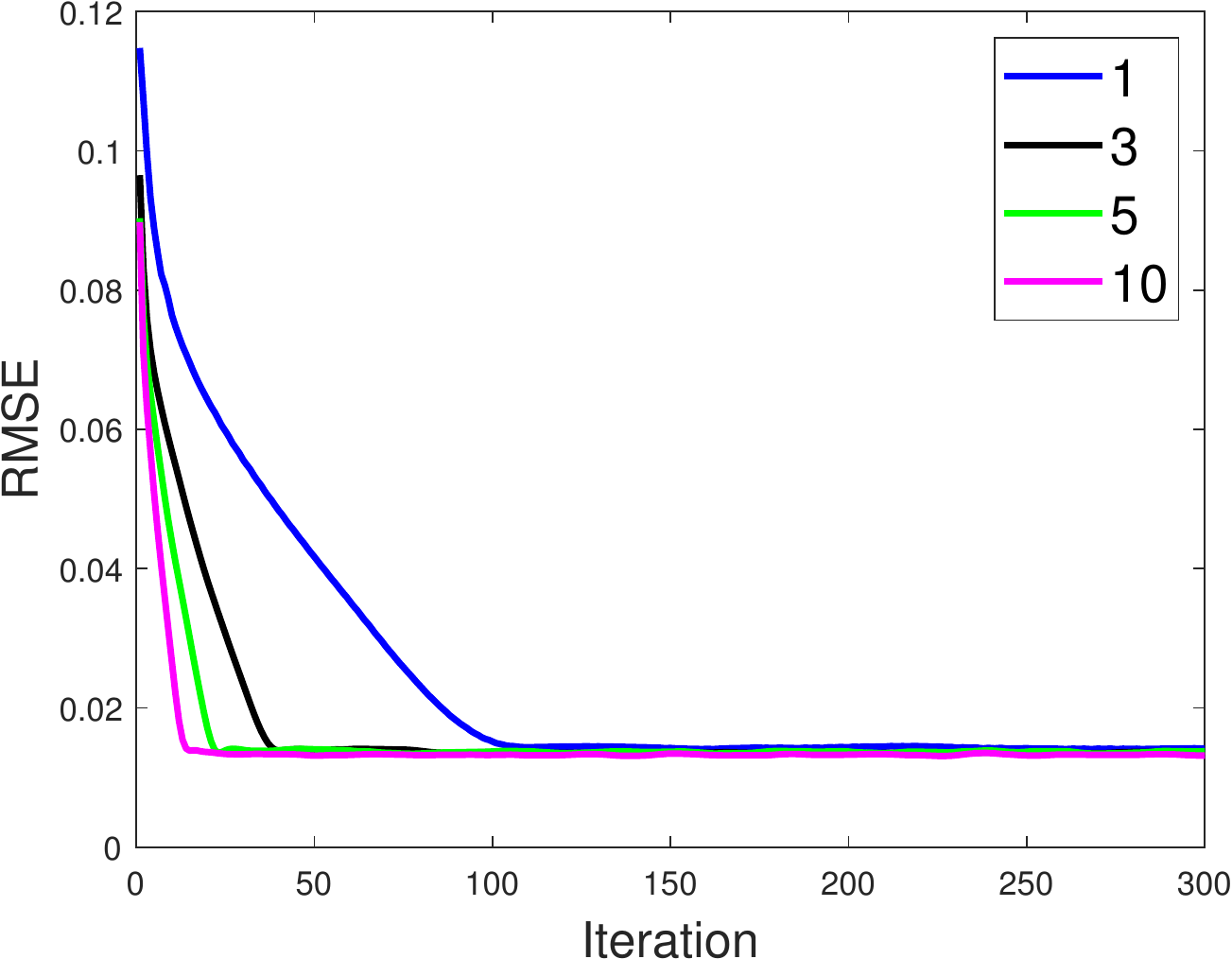} 
			\end{tabular}
		\end{center}
		\caption{The effects of  the maximum number \tr{of} the inner loops in terms of the   objective value (left) and RMSE (right). }\label{fig:effect_inner}
	\end{figure}

%	\subsection{Model comparision}

\subsection{Synthetic dataset}\label{sect:exp_syn}

We start with the parallel beam CT reconstruction of the SL phantom from  $90^\circ$ and $150^\circ$ projection range,   labelled by SL-$90^\circ$/SL-$150^\circ$,  with $0.5\%$ Gaussian noise. The quantitative results in terms of SSIM and RMSE are reported  in \Cref{Tab:noise_SL}. 
Visually in  \Cref{fig:SL150}, SART fails to recover the ellipse shape of the skull with such small ranges of projection angles.
Both $L_1$ and $L_1$-$L_2$ models are unable to restore the bottom skull and preserve details of some ellipses in the middle. \tr{The $L_p$ model leads to a nearly perfect reconstruction of the skull,} but containing a lot of salt-and-pepper artifacts inside the brain. The proposed $L_1/L_2$ method yields a reasonable recovery in a balanced manner. In the case of SL-$150^\circ$,  $L_p$ is superior over the other approaches, while the propose method is the second best.
This outcome is consistent with the CS literature \cite{yinLHX14} that  $L_p$ performs quite well for the incoherent problem, which is corresponding to a larger scanning angle in the CT reconstruction. The performance of $L_p$ decays for narrow scanning ranges, as reported in \Cref{Tab:noise_SL}.
	\begin{figure}[t]
		\begin{center}
			\begin{tabular}{ccccc}
				%		Sinogram & $L_1$-grad (RE $= 0.72\%$) & $L_1/L_2$-grad (RE $= 0.04\%$)\\
				SART & $L_1$ & $L_p$ & $L_1$-$L_2$ & $L_1/L_2$\\
				\includegraphics[width=0.15\textwidth]{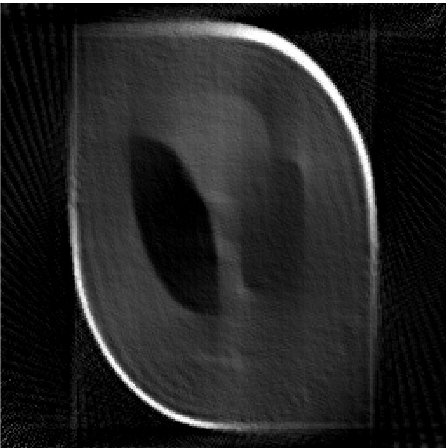} &
				\includegraphics[width=0.15\textwidth]{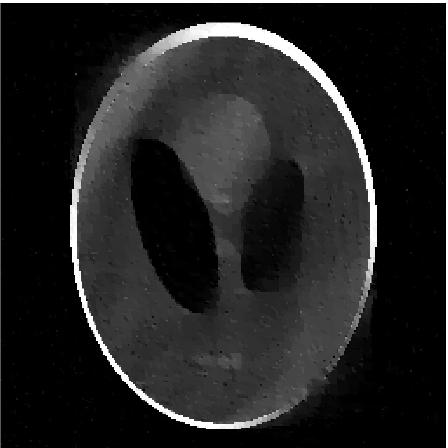} &
				\includegraphics[width=0.15\textwidth]{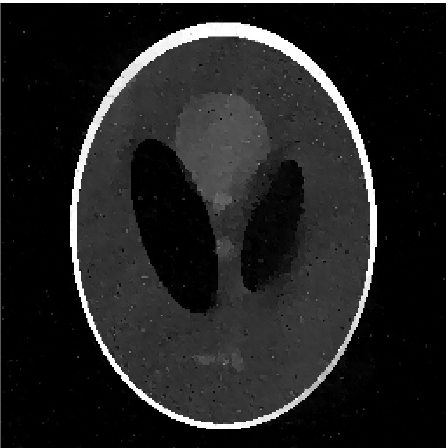} &
				\includegraphics[width=0.15\textwidth]{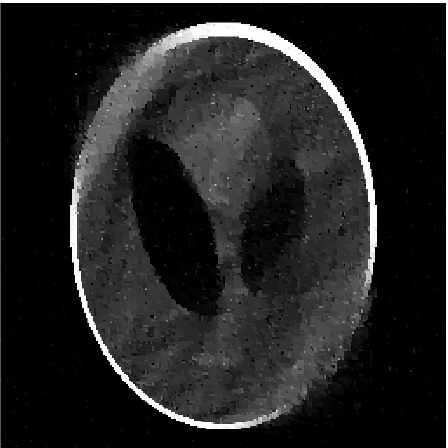}  &
				\includegraphics[width=0.15\textwidth]{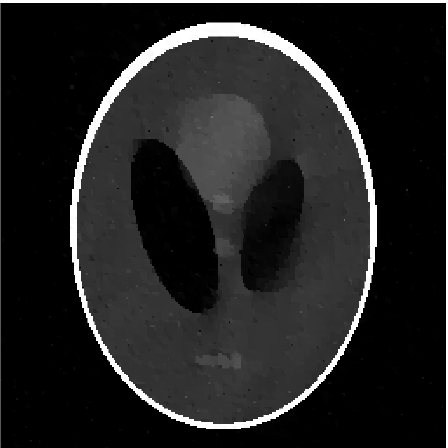}  \\
				\includegraphics[width=0.15\textwidth]{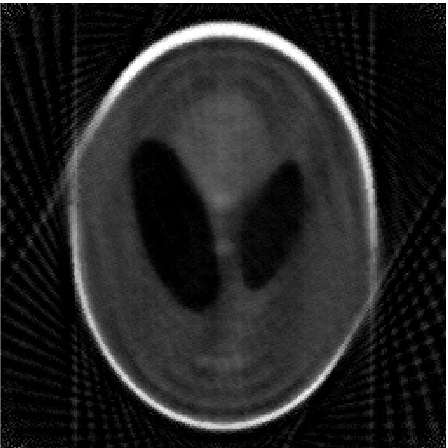} &
				\includegraphics[width=0.15\textwidth]{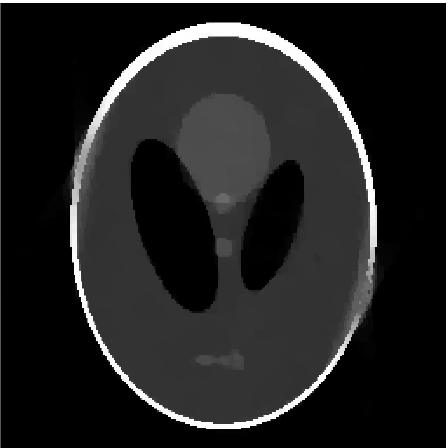} &
				\includegraphics[width=0.15\textwidth]{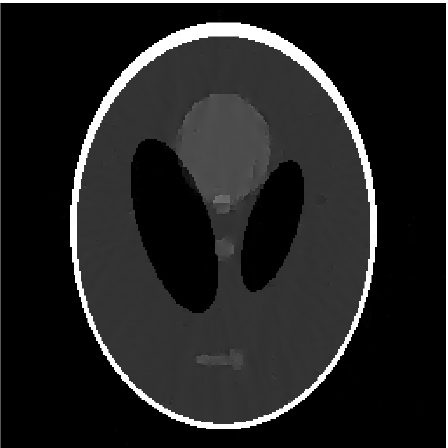} &
				\includegraphics[width=0.15\textwidth]{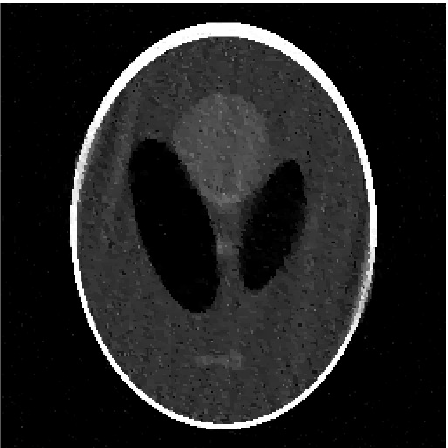}  &
				\includegraphics[width=0.15\textwidth]{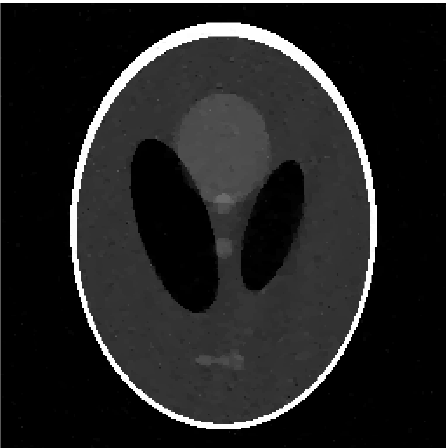}  
			\end{tabular}
		\end{center}
		\caption{CT reconstruction from $90^\circ$ (top) and $150^\circ$ (bottom) parallel beam  projection for the SL phantom with 0.5\% noise. The gray scale window is $[0,1]$. }\label{fig:SL150}
	\end{figure}

			\begin{table}[t]
		\begin{center}
			\scriptsize
			\caption{Parallel beam CT reconstruction of the  SL phantom by SART, $L_1$, $L_p$, $L_1$-$L_2$, and $L_1/L_2$. }
			\begin{tabular}{c|c|cc|cc|cc|cc|cc} 
				\hline 
				\multirow{2}{*}{noise} & \multirow{2}{*}{range} & \multicolumn{2}{c|}{SART}& \multicolumn{2}{c|}{$L_1$ }& \multicolumn{2}{c|}{$L_{p}$ }  & \multicolumn{2}{c|}{$L_1$-$L_2$ } & \multicolumn{2}{c}{$L_1/L_2$ }  \\ \cline{3-12} 
				&  & SSIM &   RMSE & SSIM &   RMSE & SSIM &   RMSE & SSIM &   RMSE & SSIM &   RMSE  \\ \hline
				\multirow{2}{*}{0.5\%}&  $90^{\circ}$ & 0.56 & 0.138 &  0.88 & 0.075 & 0.91  & 0.029 &  0.78  & 0.087 & {\bf 0.96}  & {\bf 0.017}    \\ \cline{2-12} 
				& $150^{\circ}$ & 0.58 & 0.106 &  0.98 & 0.038 & {\bf 0.99}  & {\bf 0.008} &  0.88  & 0.034 & 0.98  & 0.011   
				 \\ \hline	 
				\multirow{2}{*}{0.1\%}&  $90^{\circ}$ & 0.58 & 0.137 &  0.96 & 0.041 & {\bf 1.00} & 0.006 &  0.88  & 0.072 & {\bf 1.00}  & {\bf 0.003}     \\ \cline{2-12} 
				& $150^{\circ}$ & 0.60 & 0.104 &  0.98 & 0.035 & {\bf 1.00}  & 0.005 &  0.99  & 0.076 & {\bf 1.00}  & {\bf 0.001}  
				 \\ \hline
			\end{tabular}\label{Tab:noise_SL}
			\medskip
		\end{center}
	\end{table}

	\begin{figure}[ht]
		\begin{center}
			\begin{tabular}{ccccc}
				%		Sinogram & $L_1$-grad (RE $= 0.72\%$) & $L_1/L_2$-grad (RE $= 0.04\%$)\\
				SART & $L_1$ & $L_p$ & $L_1$-$L_2$ & $L_1/L_2$\\
				\includegraphics[width=0.15\textwidth]{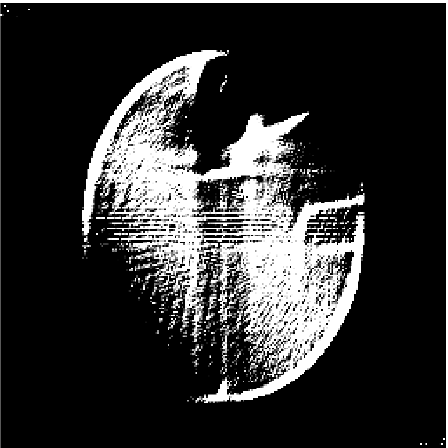} &
				\includegraphics[width=0.15\textwidth]{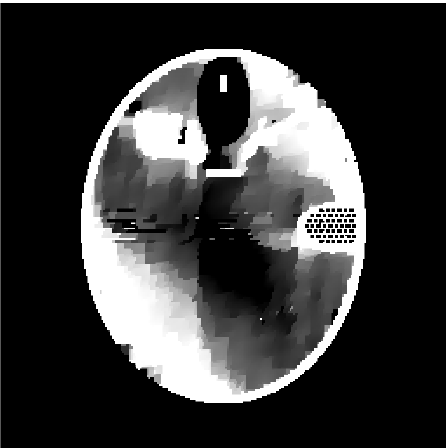} &
				\includegraphics[width=0.15\textwidth]{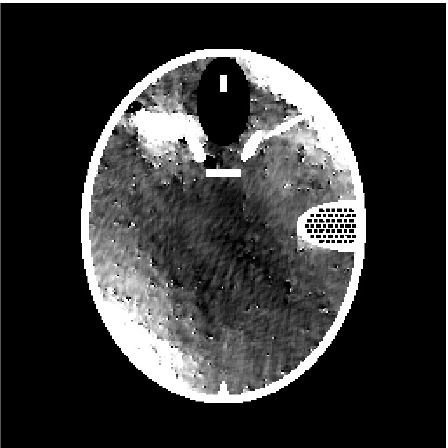} &
				\includegraphics[width=0.15\textwidth]{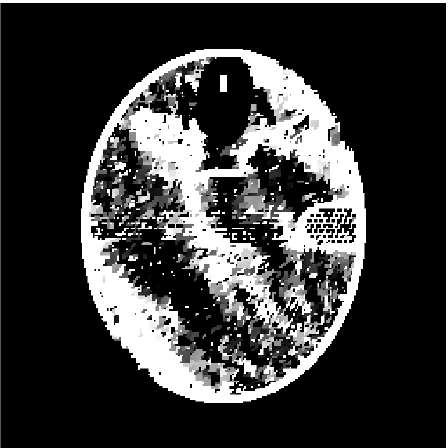}  &
				\includegraphics[width=0.15\textwidth]{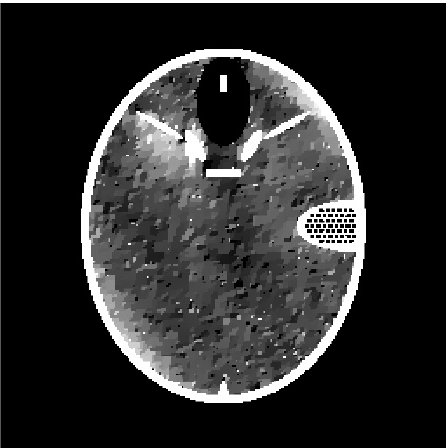}  \\
				\includegraphics[width=0.15\textwidth]{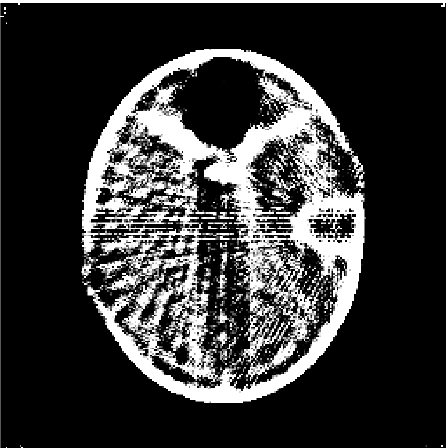} &
				\includegraphics[width=0.15\textwidth]{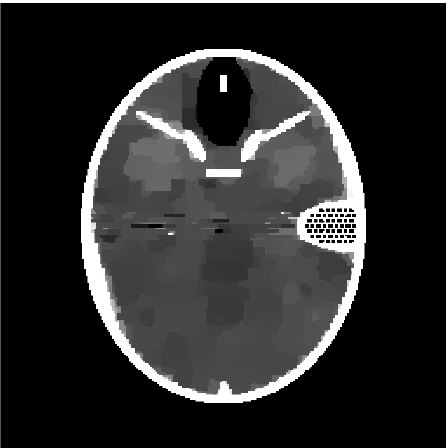} &
				\includegraphics[width=0.15\textwidth]{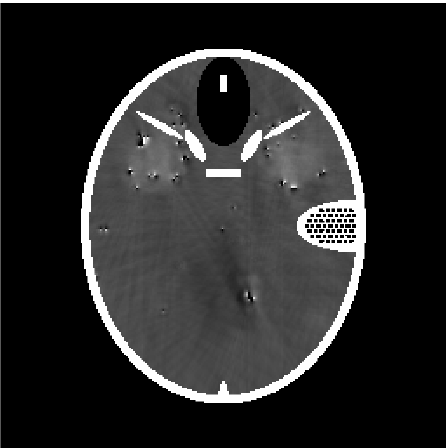} &
				\includegraphics[width=0.15\textwidth]{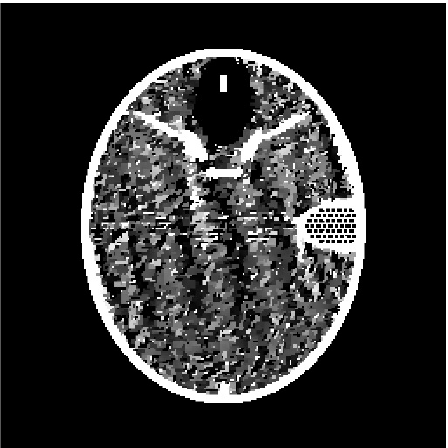}  &
				\includegraphics[width=0.15\textwidth]{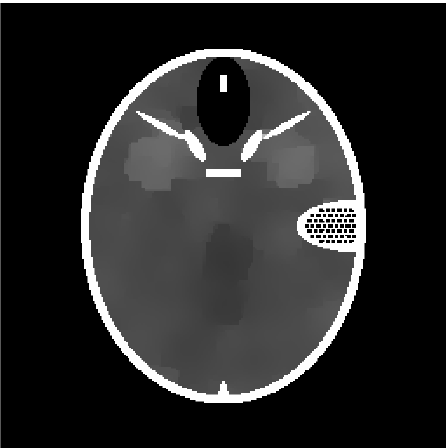}  
			\end{tabular}
		\end{center}
		\caption{CT reconstruction  from $90^\circ$ (top) and $150^\circ$ (bottom) parallel beam projection for the FB phantom with 0.1\% Gaussian  noise. The gray scale window is $[1.03,1.10]$. }\label{fig:FB150}
	\end{figure}
		\begin{figure}[!]
		\begin{center}
			\begin{tabular}{cc}
				%		Sinogram & $L_1$-grad (RE $= 0.72\%$) & $L_1/L_2$-grad (RE $= 0.04\%$)\\
				\includegraphics[width=0.4\textwidth]{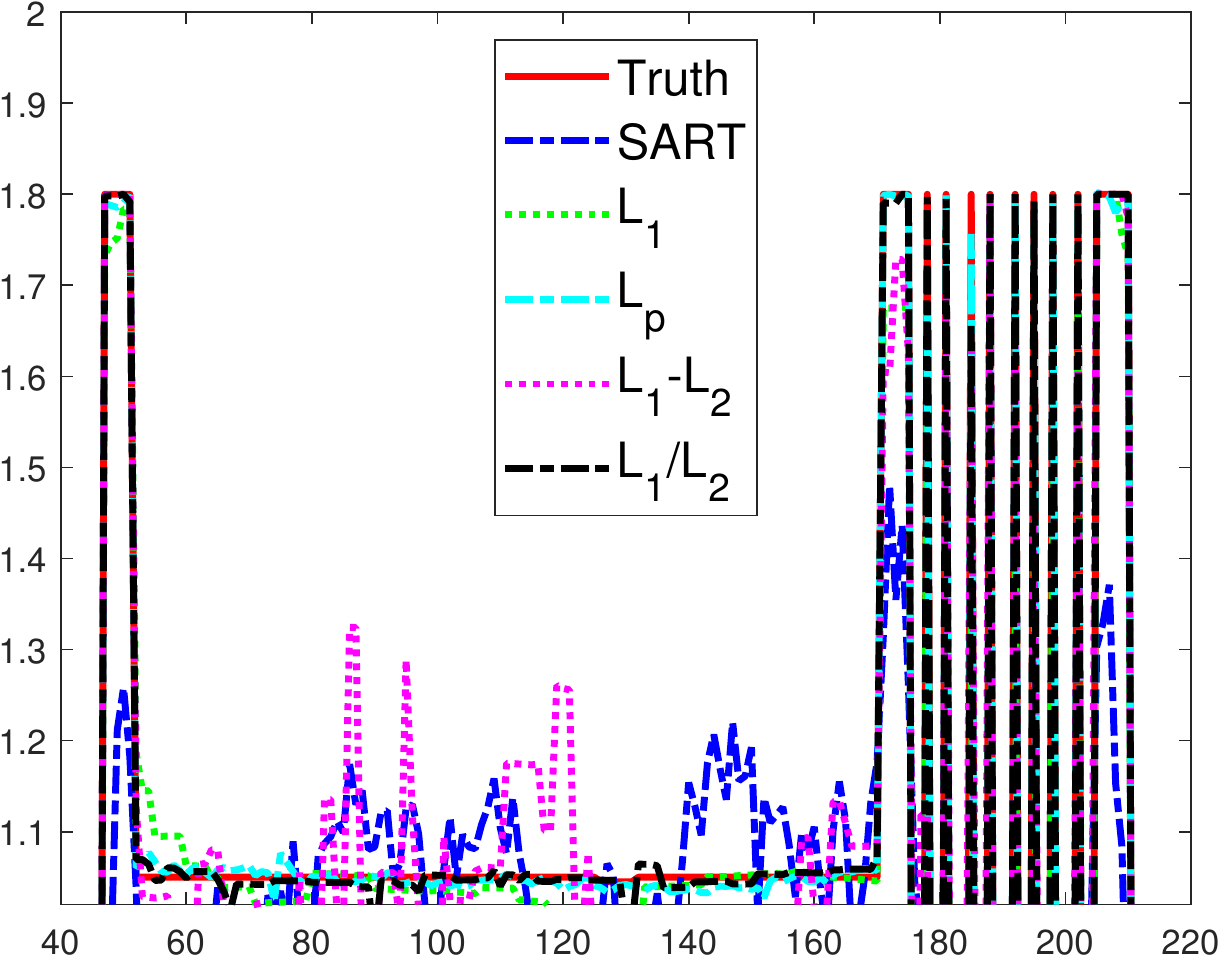} &
				\includegraphics[width=0.4\textwidth]{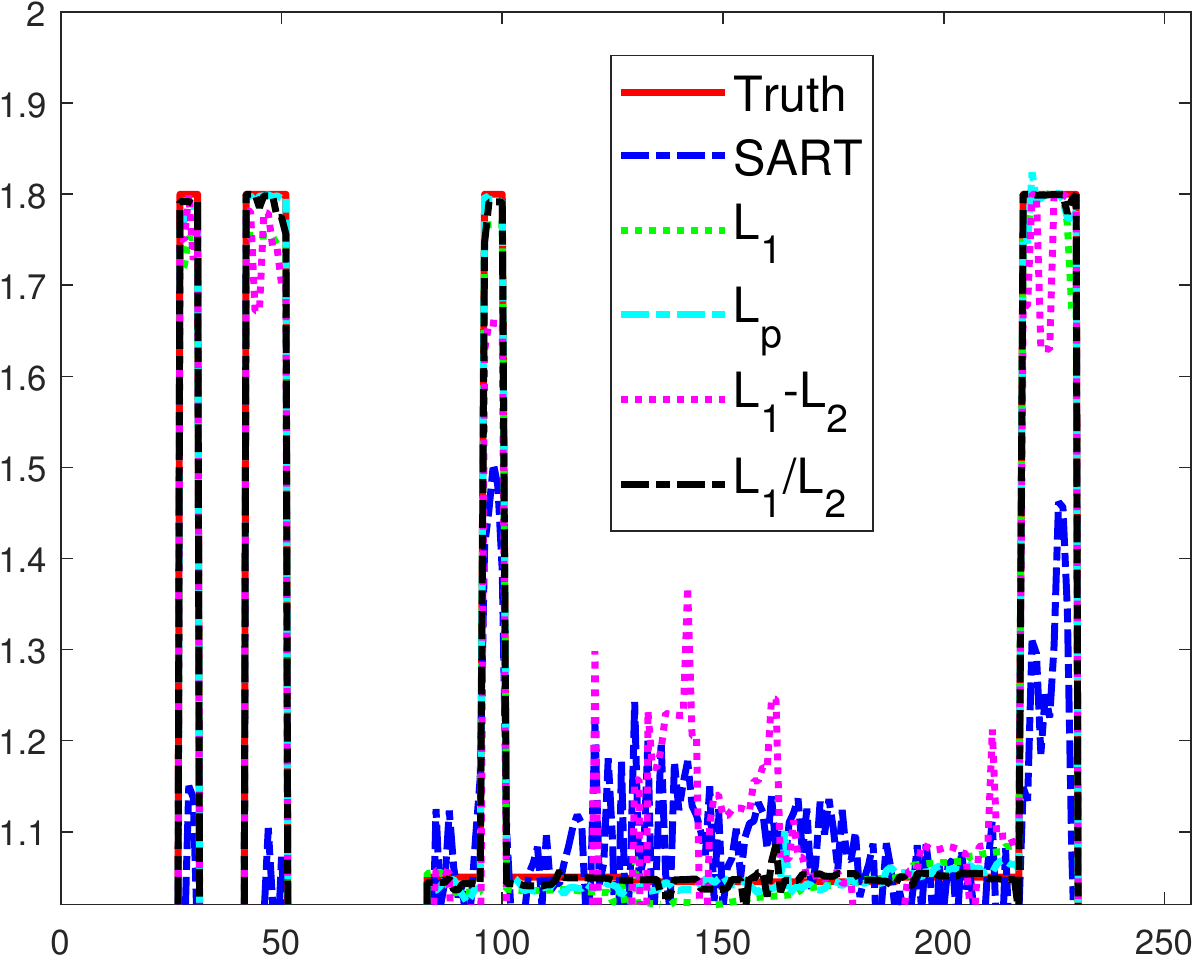} \\
				\includegraphics[width=0.4\textwidth]{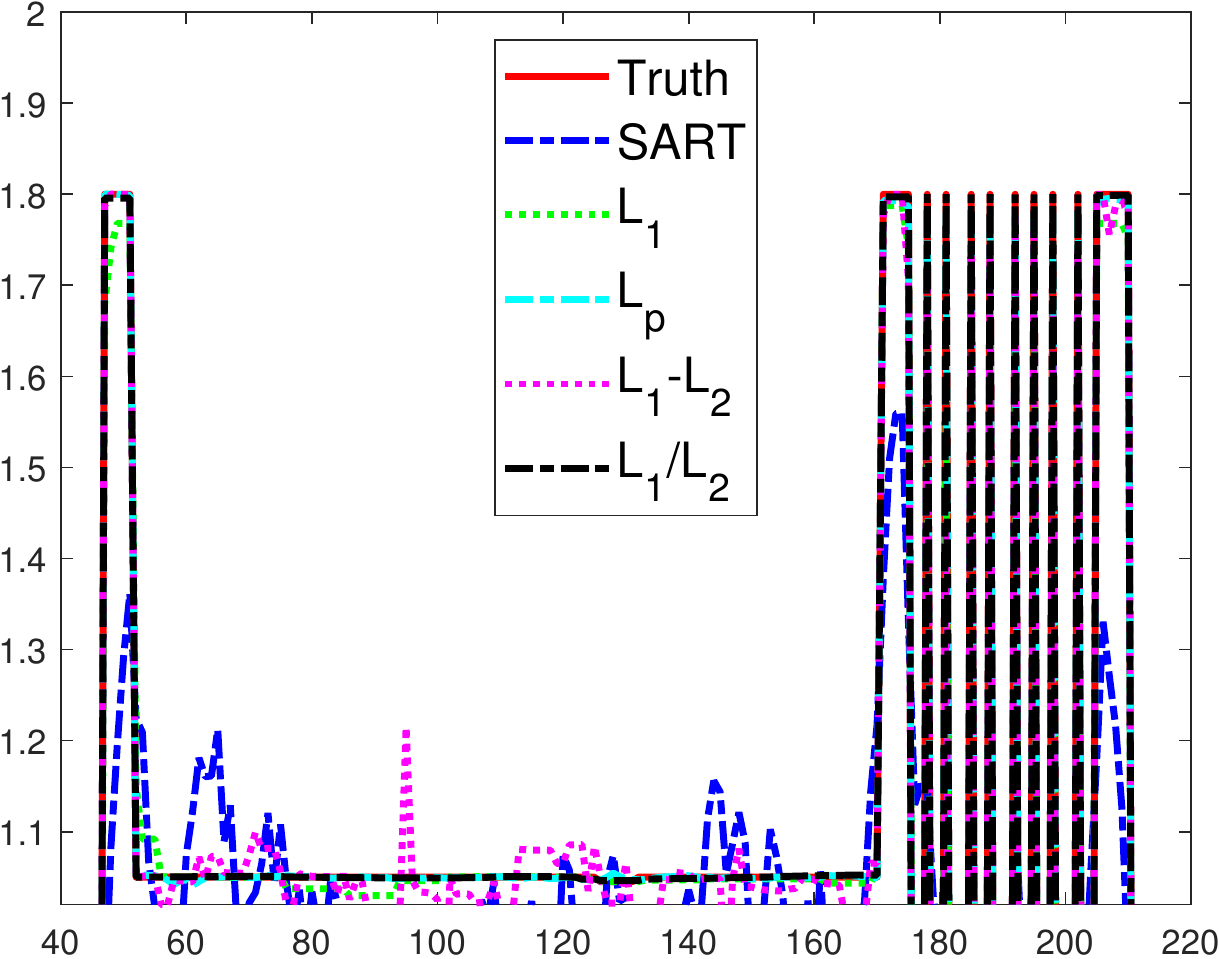} &
				\includegraphics[width=0.4\textwidth]{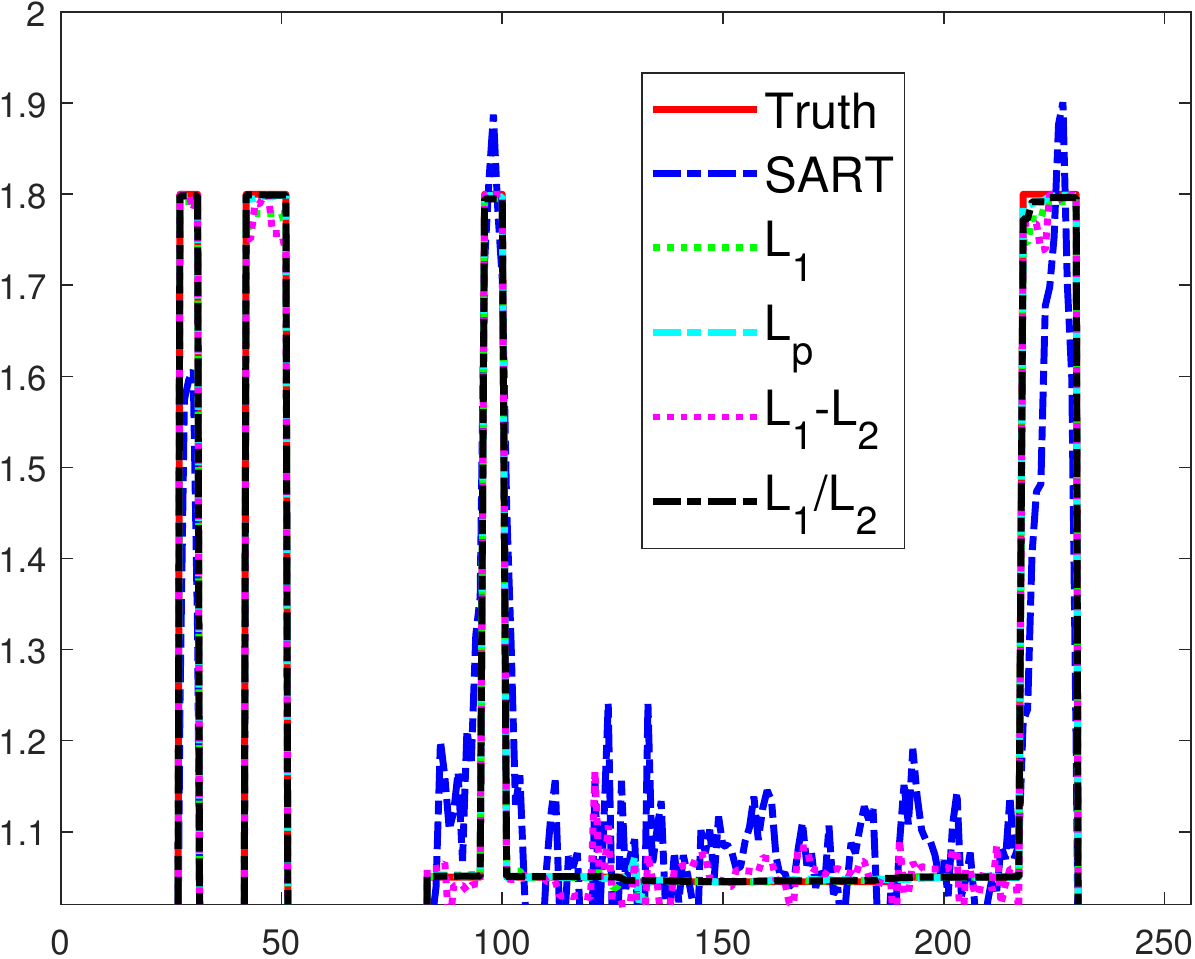} 
			\end{tabular}
		\end{center}
		\caption{Horizontal and vertical profiles generated via SART, $L_1$, $L_p$, $L_1$-$L_2$, and $L_1/L_2$ in the range of projection $90^\circ$ (top) and $150^\circ$ (bottom) for the FB phantom. }\label{fig:FB150_vh}
	\end{figure}
	
	\begin{table}[htp]
		\begin{center}
			\scriptsize
			\caption{Parallel beam CT reconstruction of the  FB phantom by SART, $L_1$, $L_p$, $L_1$-$L_2$, and $L_1/L_2$. }
			\begin{tabular}{c|c|cc|cc|cc|cc|cc}
				\hline
				\multirow{2}{*}{noise} & \multirow{2}{*}{range} & \multicolumn{2}{c|}{SART}& \multicolumn{2}{c|}{$L_1$ } &\multicolumn{2}{c|}{$L_p$ } & \multicolumn{2}{c}{$L_1$-$L_2$ } & \multicolumn{2}{|c}{$L_1/L_2$ }  \\ \cline{3-12} 
				&  & SSIM &   RMSE & SSIM & RMSE & SSIM &  RMSE & SSIM &   RMSE & SSIM &   RMSE  \\ \hline
				\multirow{2}{*}{0.5\%}&  $90^{\circ}$ & 0.26 & 0.275 &  0.82 & 0.135 & 0.77  & 0.101 &  0.65  & 0.169 & {\bf 0.91}  & {\bf 0.080}     \\ \cline{2-12} 
				& $150^{\circ}$ & 0.28 & 0.206 &  0.90 & 0.059 & 0.70  & 0.078 &   0.70  & 0.107 & {\bf 0.95}  & {\bf 0.028}   
				 \\ \hline	 
				\multirow{2}{*}{0.1\%}&  $90^{\circ}$ & 0.30 & 0.266 &  0.93 & 0.101 & 0.97  & 0.049 &  0.78  & 0.123 & {\bf 0.99}  &{\bf 0.012} \\ \cline{2-12} 
				& $150^{\circ}$ & 0.32 & 0.192 &  0.99 & 0.026 & {\bf 1.00}  & 0.003 &   0.94  & 0.026 & {\bf 1.00} & {\bf 0.002}  
				 \\ \hline
			\end{tabular}\label{Tab:noise_FB}
			\medskip
		\end{center}
	\end{table}

We present the visual results of FB-$90^\circ$ and FB-$150^\circ$ with $0.1\%$ Gaussian noise
in \Cref{fig:FB150}. None of the methods can get satisfactory recovery results under the gray scale window of $[1.03, 1.10]$.  Large fluctuations inside of the skull are produced by the competing methods, among which
 $L_1/L_2$ can restore the most details of the image.
	%Since the FB image is of low contrast,. 
	%With such a narrow window of [1.0, 1.2],  
	%Among the competing methods, 
	%$L_1/L_2$ .  
Furthermore, we plot the horizontal and vertical profiles in \Cref{fig:FB150_vh}, which illustrates that $L_1/L_2$ leads to the smallest fluctuations  compared to  others. In contrary to the simple SL phantom, $L_p$ does not work well for FB. We also observe a well-known artifact of 
	the $L_1$ method, i.e., loss of contrast, as its profile  fails to reach the height of jump on the intervals such as $[160,180]$ in the left plot and $[220, 230]$ in the right plot of \Cref{fig:FB150_vh}, while $L_1/L_2$ has a good recovery in these regions.  As shown in \Cref{fig:FB150}, errors in these low contrast regions are being magnified when  we display the restored image in a narrow gray scale window. Furthermore, the profile plots in \Cref{fig:FB150_vh} confirm that  our approach performs  very well for  high contrast details \cite{wang2021minimizing}. We report the quantitative results of FB in \Cref{Tab:noise_FB}. Comparing \Cref{Tab:noise_SL,Tab:noise_FB} shows that all the methods yield better performance for smaller noise level and 
a larger range of scanning angle.
In addition, the recovery results of FB is much worse than the ones of SL,
%	Although the proposed approach yields  salt-and-pepper type of artifacts \tb{[you didn't show any images for this set?]}, it gives better results compared to the other competing methods. 
 %we observe worse recovery results of the FB phantom  than SL, 
 which is largely due to  low contrast structures in FB.

		\begin{figure}[t]
		\begin{center}
			\begin{tabular}{ccccc}
				%		Sinogram & $L_1$-grad (RE $= 0.72\%$) & $L_1/L_2$-grad (RE $= 0.04\%$)\\
				SART & $L_1$ & $L_p$ &  $L_1$-$L_2$ & $L_1/L_2$\\
				\includegraphics[width=0.15\textwidth]{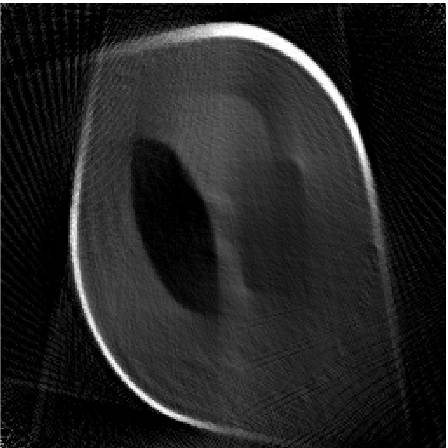} &
				\includegraphics[width=0.15\textwidth]{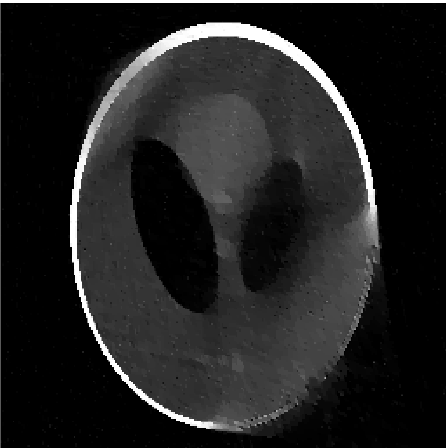} &
				\includegraphics[width=0.15\textwidth]{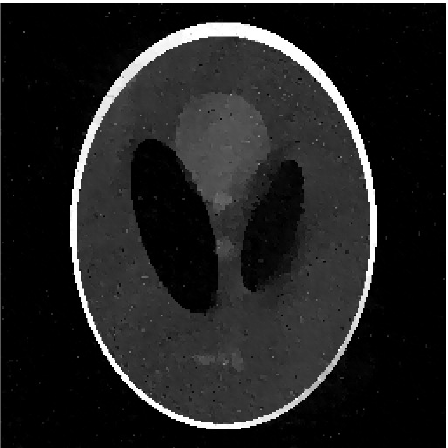} &
				\includegraphics[width=0.15\textwidth]{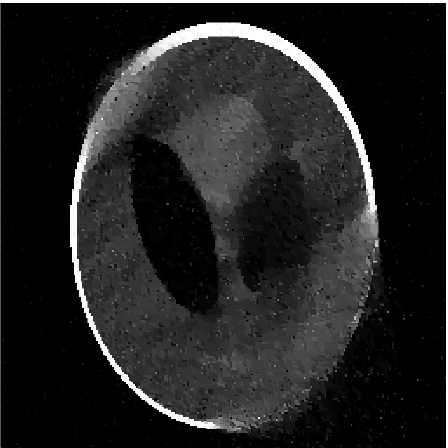}  &
				\includegraphics[width=0.15\textwidth]{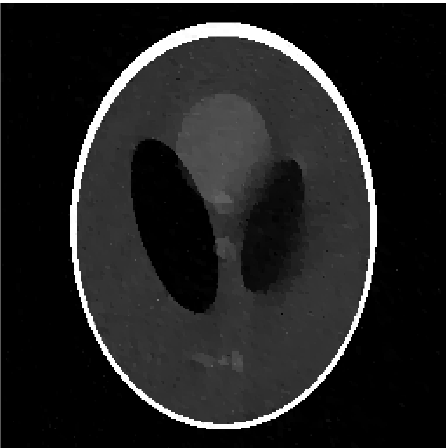}  \\
				\includegraphics[width=0.15\textwidth]{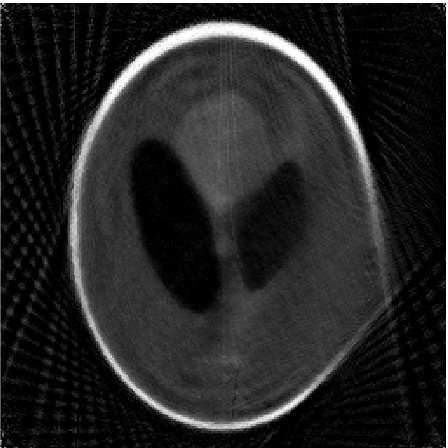} &
				\includegraphics[width=0.15\textwidth]{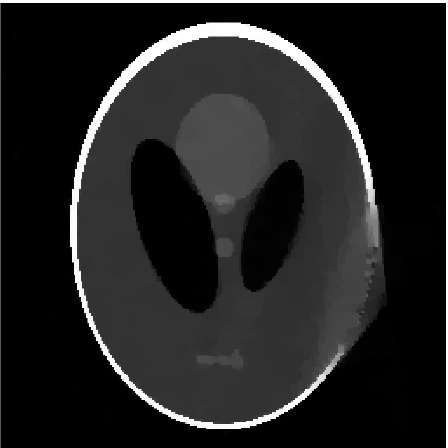} &
				\includegraphics[width=0.15\textwidth]{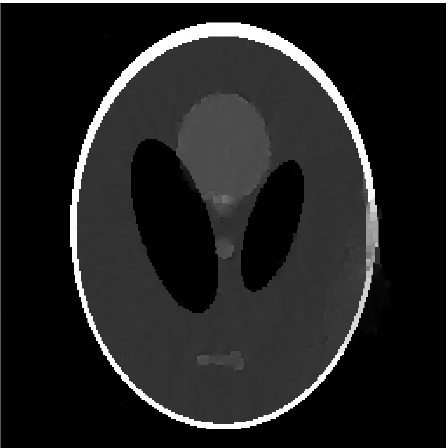} &
				\includegraphics[width=0.15\textwidth]{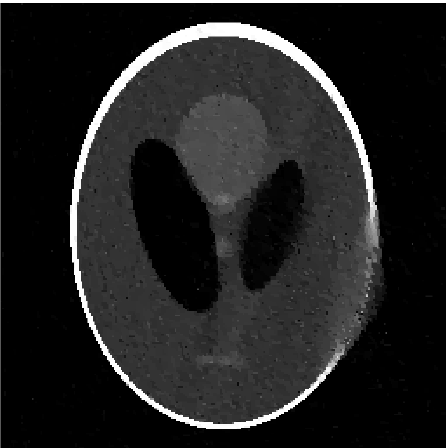}  &
				\includegraphics[width=0.15\textwidth]{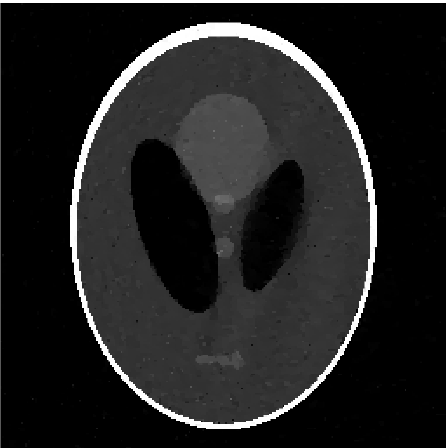}  
			\end{tabular}
		\end{center}
		\caption{CT reconstruction from $90^\circ$ (top) and $150^\circ$ (bottom) fan beam  projection  for the SL phantom with 0.5\% Gaussian noise. The gray scale window is $[0,1]$. }\label{fig:SL150fan}
	\end{figure}

We then test   fan beam CT reconstruction using the SL phantom with  $0.5\%$ Gaussian noise. 
%The projection range is $90^\circ$ and $150^\circ.$ 
Note that  fan beam with same scanning angle is more ill-posed than in the cases of   parallel beam.  \Cref{fig:SL150fan} illustrates that the ellipse shape of skull can not be completely recovered except for the proposed method. In the case of SL-$150^\circ$, $L_1/L_2$ recovers the image with  RMSE of 0.014, while RMSEs of other approaches all exceed 0.020. Overall, the proposed $L_1/L_2$ approach achieves significant improvements over SART, $L_1$, and $L_1$-$L_2$. Here  $L_p$ is comparable to $L_1/L_2$ only in the case of wider scanning ranges and ground-truth images with simple geometries.

	Lastly, we consider a more realistic noise statistics, i.e., Poisson noise, for the CT problem.
% 	In particular, the projection data is generated by taking the exponential of the negative values of the noise-free sonogram. 
% 	Then a Poisson-distributed noise is added with two different dose levels $I=10^4$ and $10^5$. \yl{[unclear about the dose level. any ref to add poisson noise?]}
Under such noise model, we also examine a popular data fitting term, called \tr{ weighted least-squares (WLS)} \cite{thibault2007three}, to measure the data misfit. In fact, \tr{WLS}  replaces the LS term in \eqref{eq:grad_uncon} by $\frac{\lambda}{2} \|A u -f\|_W^2:=\frac{\lambda}{2} (A u -f)^T W (A u -f)$, where $W = \mathrm{diag}(\mathrm{exp}(-f))$. As a result, we can simply modify the LS implementations to fit in \tr{WLS}. 
		We present one example of reconstructing the SL phantom from $150^\circ$ fan beam projection with noise level $I_0=10^5$.
		\cref{fig:SL150poi5} shows similar results of LS and \tr{WLS}. Specifically, LS gives a better recovery of the skull, while \tr{WLS} has less fluctuations inside the brain. We further compare the two data terms under different noise levels in \Cref{Tab:pois},  reporting  minor improvements of \tr{WLS} over LS for all the regularization methods.

				\begin{figure}[t]
		\begin{center}
			\begin{tabular}{cccc}
				%		Sinogram & $L_1$-grad (RE $= 0.72\%$) & $L_1/L_2$-grad (RE $= 0.04\%$)\\
				 $L_1$ & $L_p$ &  $L_1$-$L_2$ & $L_1/L_2$\\
				\includegraphics[width=0.15\textwidth]{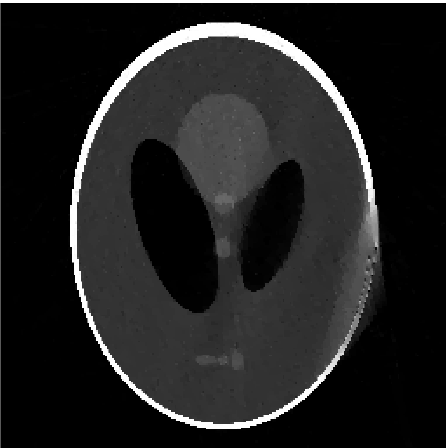} &
				\includegraphics[width=0.15\textwidth]{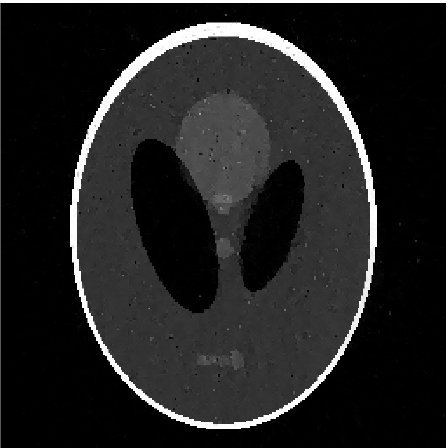} &
				\includegraphics[width=0.15\textwidth]{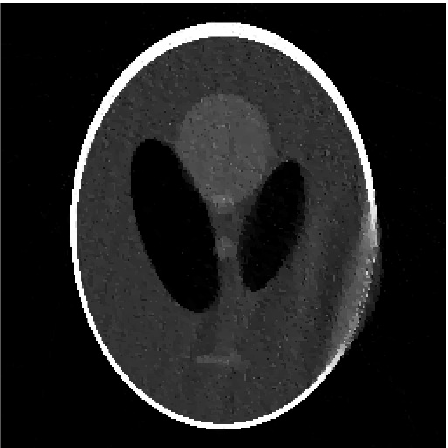}  &
				\includegraphics[width=0.15\textwidth]{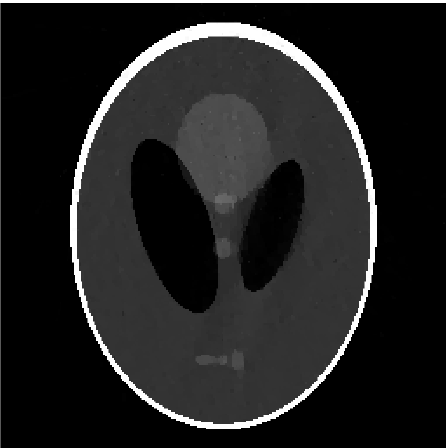}  \\
				\includegraphics[width=0.15\textwidth]{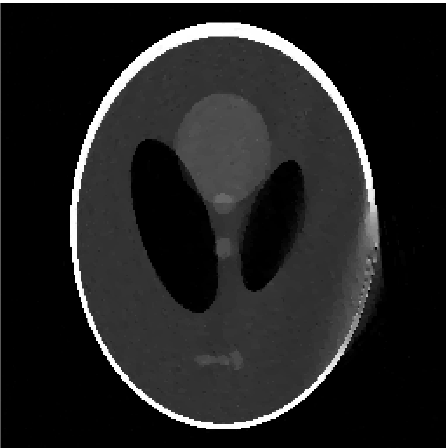} &
				\includegraphics[width=0.15\textwidth]{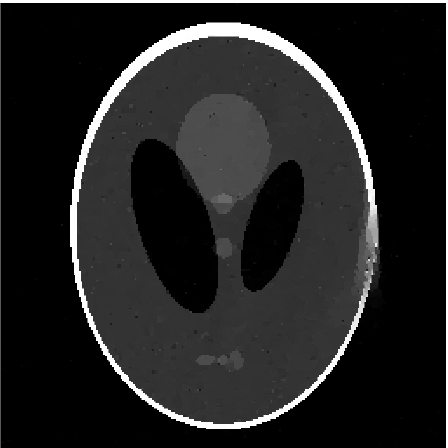} &
				\includegraphics[width=0.15\textwidth]{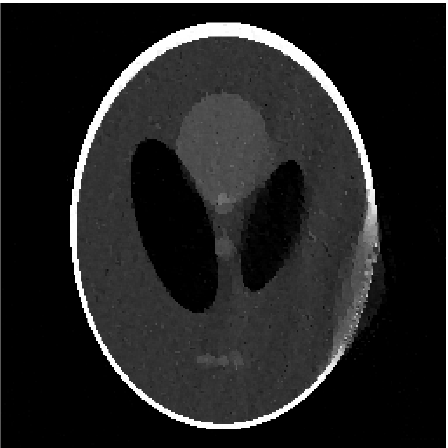}  &
				\includegraphics[width=0.15\textwidth]{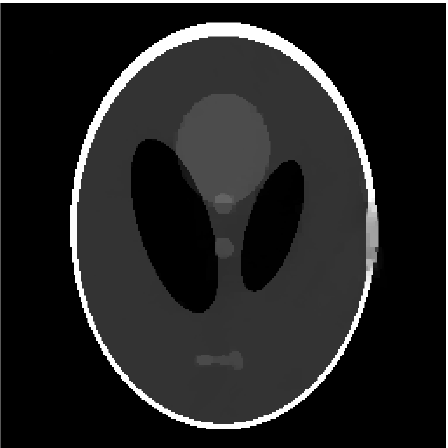}  
			\end{tabular}
		\end{center}
		\caption{CT reconstruction  from the $150^\circ$ fan beam projection for the SL phantom with Poisson noise $I_0=10^5$ using LS (top) and \tr{WLS} (bottom) data-fidelity term. The gray scale window is $[0,1]$. }\label{fig:SL150poi5}
	\end{figure}
	
		\begin{table}[t]
		\begin{center}
			\scriptsize
			\caption{CT reconstruction from  $150^\circ$ fan beam projection for  the SL phantom with Poisson noise  by $L_1$, $L_p$, $L_1$-$L_2$, and $L_1/L_2$. }
			\begin{tabular}{c|c|cc|cc|cc|cc}
				\hline
				\multirow{2}{*}{$I_0$} & \multirow{2}{*}{Data-fitting} &  \multicolumn{2}{c|}{$L_1$ } & \multicolumn{2}{c|}{$L_p$ } & \multicolumn{2}{c}{$L_1$-$L_2$ } & \multicolumn{2}{|c}{$L_1/L_2$ }  \\ \cline{3-10} 
				&  &   SSIM &   RMSE & SSIM & RMSE & SSIM &  RMSE & SSIM &   RMSE  \\ \hline
				\multirow{2}{*}{$10^4$}&  LS &   0.93 & 0.057 & 0.90  & 0.054 &  0.79  & 0.068  & 0.85  & 0.053     \\ \cline{2-10}
				& \tr{WLS} &  0.91 & 0.056 & 0.88  & 0.051 &  0.78  & 0.066 & {\bf 0.95}  & {\bf 0.051}   
				 \\ \hline	 
				\multirow{2}{*}{$10^5$}&  LS &   0.96 & 0.042  & 0.95  & 0.017 & 0.91  & 0.046 & 0.99  & 0.009   \\ \cline{2-10} 
				& \tr{WLS} &   0.97 & 0.041 & 0.98  & 0.028 &  0.93  & 0.045 & {\bf 0.99}  & {\bf 0.007}    
				 \\ \hline
			\end{tabular}\label{Tab:pois}
			\medskip
		\end{center}
	\end{table}			
		
	\subsection{Experimental dataset}\label{sect:exp_real}
	We set up a limited-angle CT problem from two experimental datasets \cite{lotusdata,walnut_ct}. 	The reference image of the walnut is of size $164\times 164,$ while the one of the lotus is  $128\times 128$. The sinogram for walnut  is  $f\in\mathbb{R}^{164\times 120}$ ($3^\circ$ per projection), and the projection matrix $A\in \mathbb{R}^{19680\times 26896}.$ In the lotus case, $f\in\mathbb{R}^{429\times 120}$ and  $A\in \mathbb{R}^{51480\times 16384}.$  When we perform the limited-angle CT reconstruction, we take partial data from $f.$  Specifically, we consider  $150^\circ$  scanning angle by selecting the first 50-projections, i.e., extracting the corresponding rows of $A$ and the columns of sinogram to generate the projection matrix and sinogram, respectively. Since the real data  contains noise generated by the CT machine, we do not add additional noise in the sinogram.  The reference images  shown in \Cref{fig:SL_FB} are reconstructed from the complete scanning data by using the Tikhonov regularization. We further impose a region of interest (ROI) when computing the quantitative evaluation metrics. The ROI is a circle with radius of 62 for walnut and 72 for lotus. 
	
	We consider a $[0,0.5]$ box constraint on all the regularization methods  ($L_1$, $L_p,$ $L_1$-$L_2$, and $L_1/L_2$), which is estimated from the reference images. We do not assume any noise type (nor noise level), and we only consider LS as the data fitting term. The optimal parameters are selected based on the ``eye-ball'' norm of the restored image, focusing on textures and details such as the shell of walnut and its inner structure. The reconstruction results are presented in \Cref{fig:walnut150,fig:lotus150} for walnut and lotus, respectively, within the corresponding ROIs and under a gray scale window of $[0,0.6].$
	In \Cref{fig:walnut150}, SART produces a lot of artifacts. The $L_1$ model gets a good recovery, but losing some details on the bottom-left corner with blurring inner texture. All these nonconvex regularization models have sharper images than $L_1,$ while $L_1/L_2$ can have a higher contrast especially for the internal region of the walnut. 
	The lotus is more difficult to reconstruct, as its root is filled with attenuating objects that causes severe metal artifacts.
	 In \Cref{fig:lotus150}, the restored image via our proposed model has less streaking artifacts than the ones by other approaches. Lastly, we provide some quantitative analysis in \Cref{Tab:real}. %Note that reference images are the ones via  Tikhonov regularization in the full scanning data, instead of those in \Cref{fig:SL_FB} owing to the consistency of resolution, To avoid the ring effect of CT reconstruction, we set the values out of ROI as zero.  
	All these regularization methods have similar performance in terms of SSIM and RMSE, while $L_1$ has the best results. As reference images have some obvious streaking artifacts, the method with the best quantitative measures does not grant the optimal performance.

		\begin{figure}[t]
		\begin{center}
			\begin{tabular}{ccc}
				%		Sinogram & $L_1$-grad (RE $= 0.72\%$) & $L_1/L_2$-grad (RE $= 0.04\%$)\\
				Reference & SART & $L_1$ \\
				\includegraphics[width=0.2\textwidth]{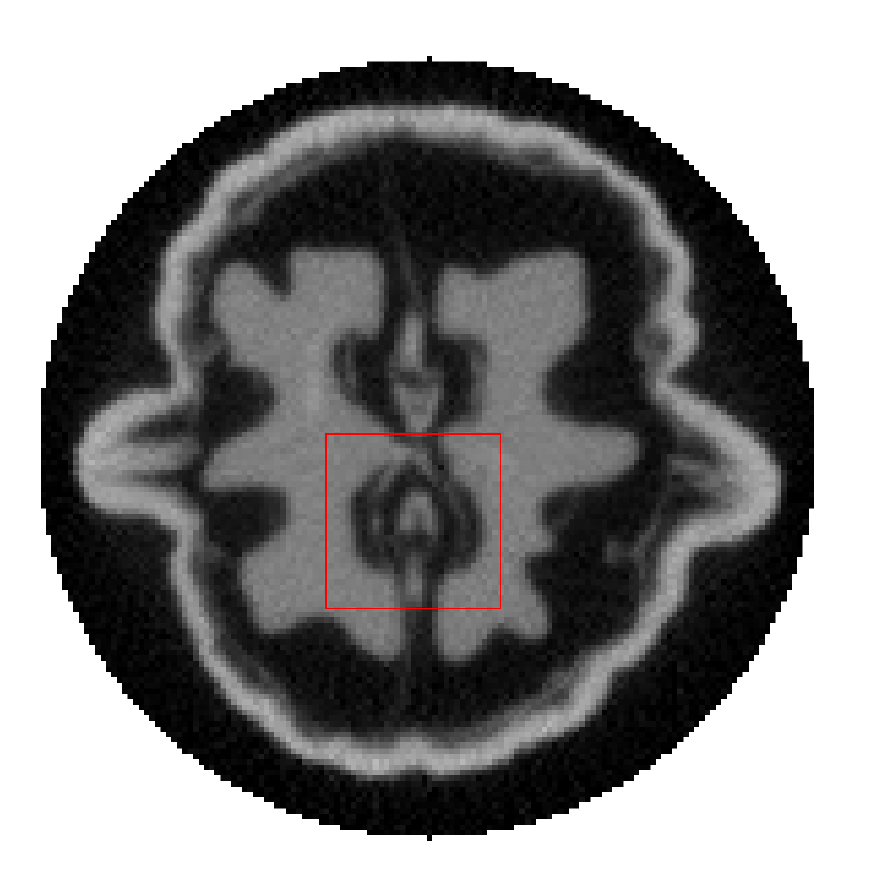} &
				\includegraphics[width=0.2\textwidth]{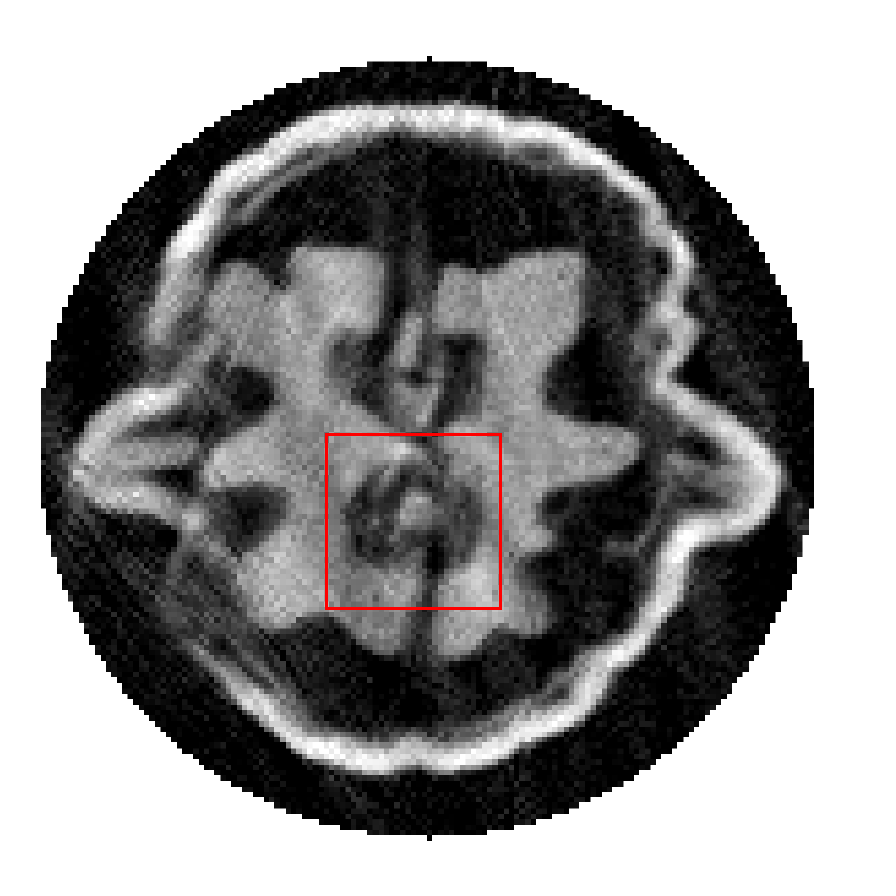} &
				\includegraphics[width=0.2\textwidth]{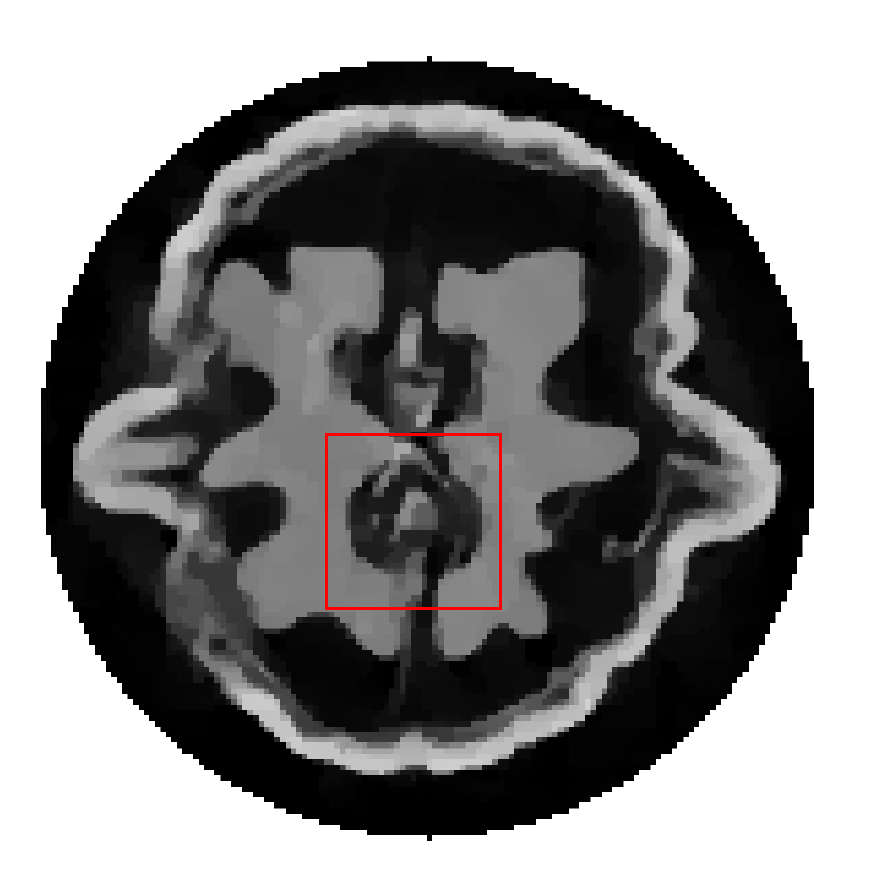} \\
				$L_p$  & 	$L_1$-$L_2$ & $L_1/L_2$\\
				\includegraphics[width=0.2\textwidth]{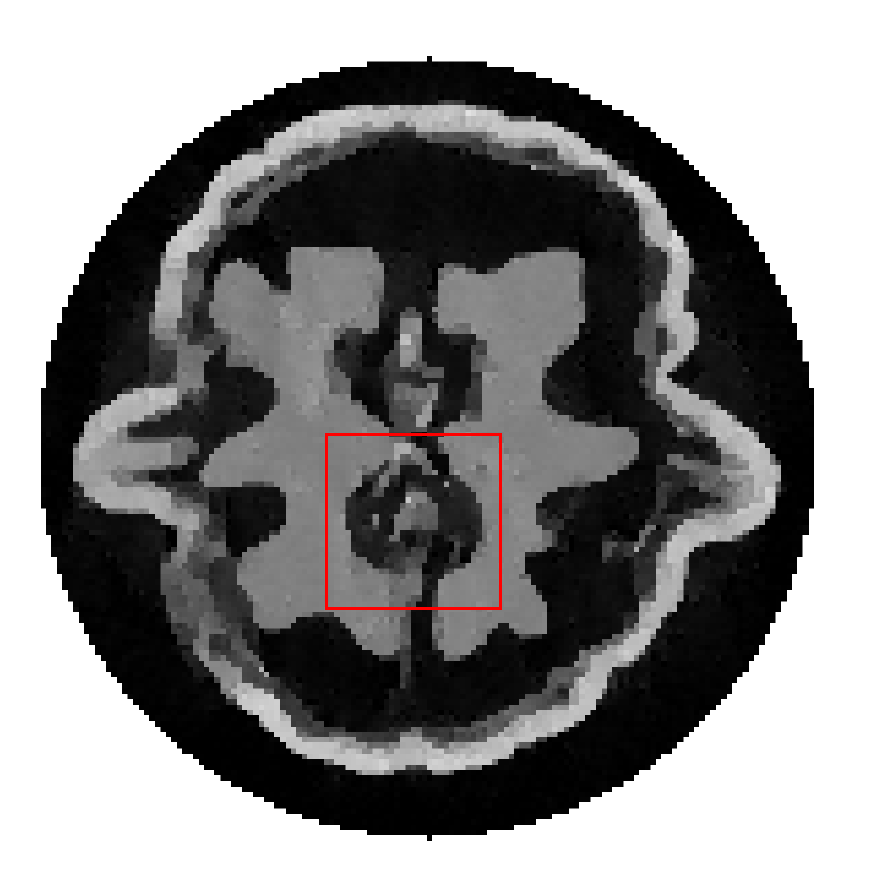} &
				\includegraphics[width=0.2\textwidth]{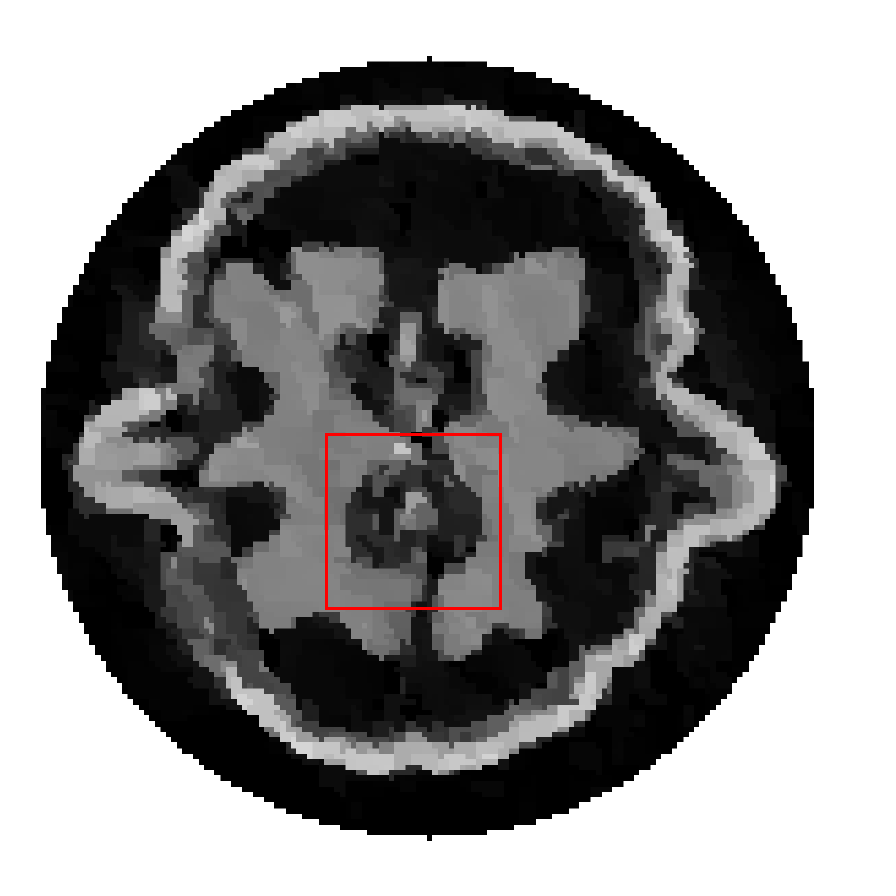} &
				\includegraphics[width=0.2\textwidth]{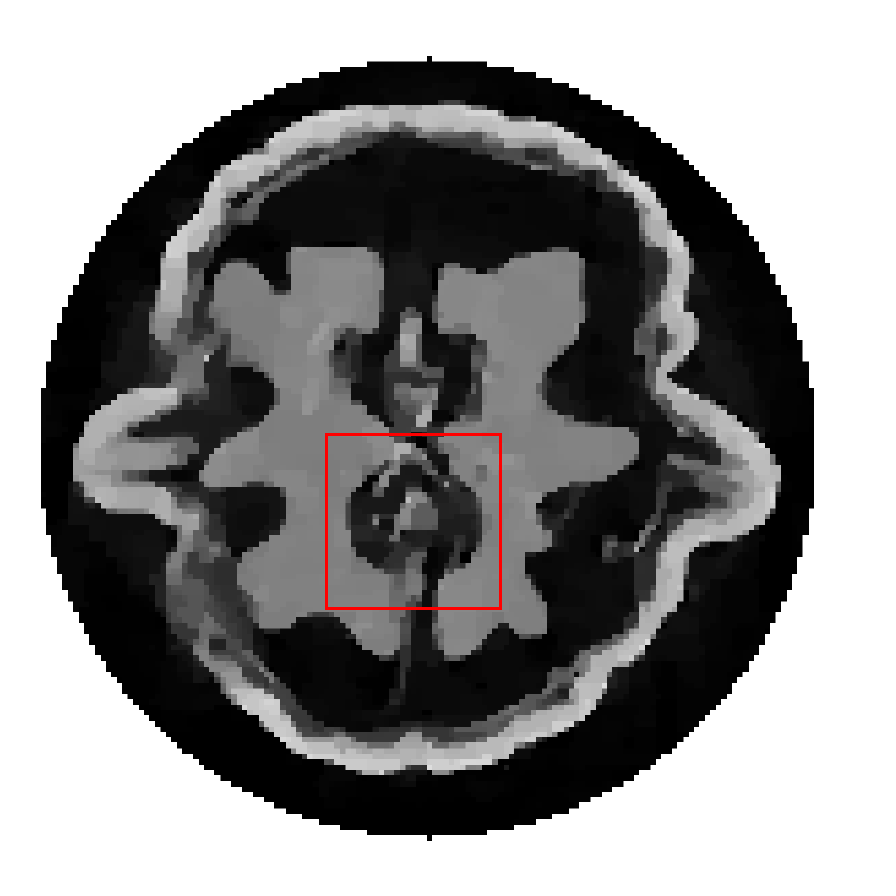} \\
				Reference & SART & $L_1$ \\
				\includegraphics[width=0.2\textwidth]{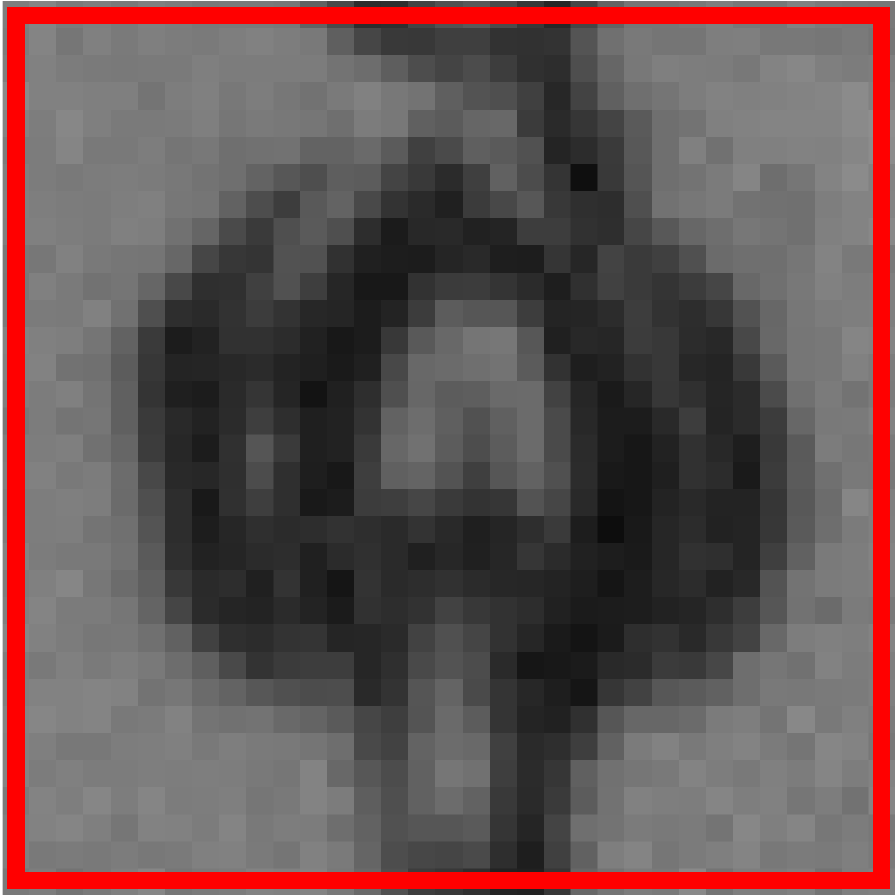} &
			     \includegraphics[width=0.2\textwidth]{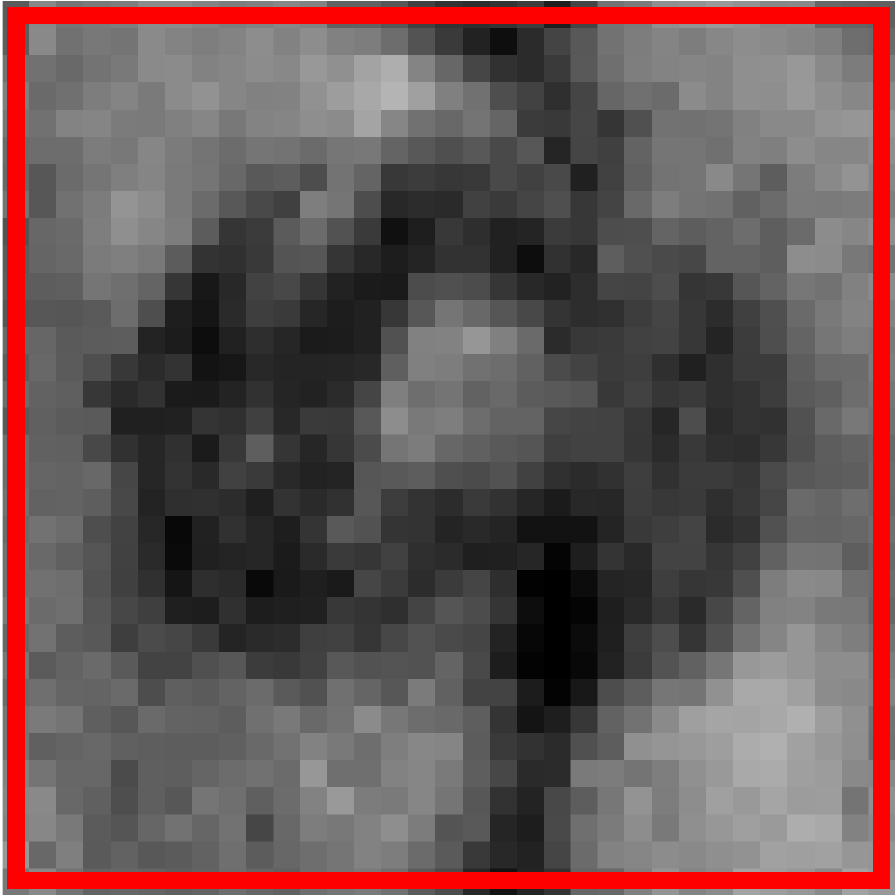} &
				\includegraphics[width=0.2\textwidth]{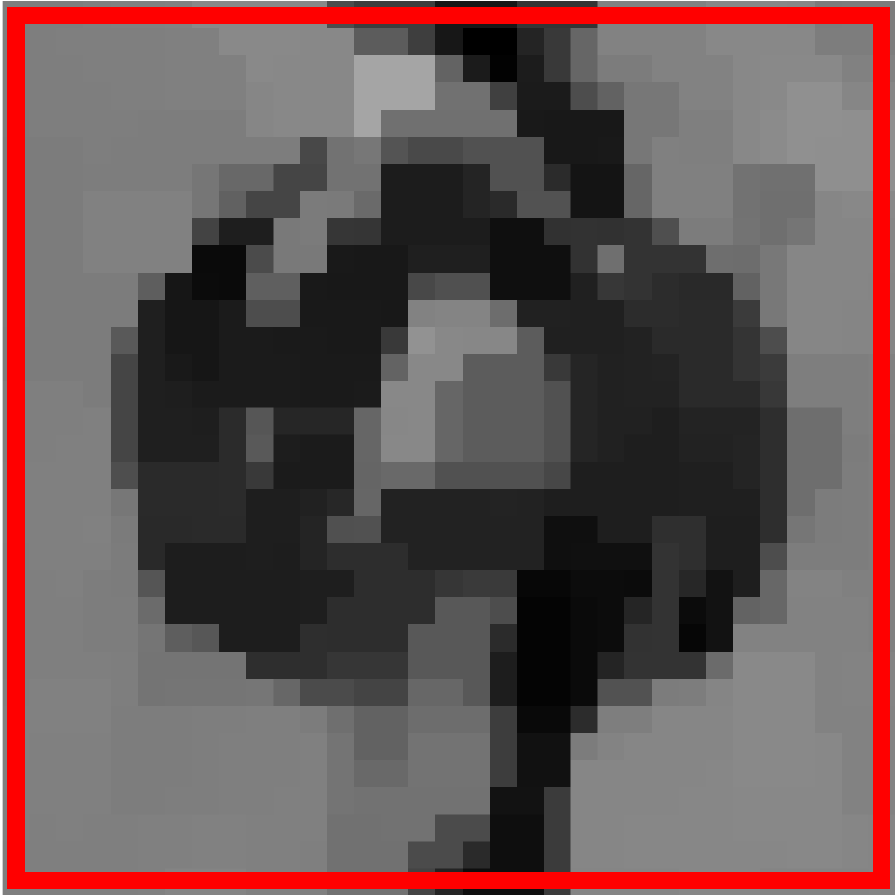} \\
				$L_p$  & 	$L_1$-$L_2$ & $L_1/L_2$\\
				\includegraphics[width=0.2\textwidth]{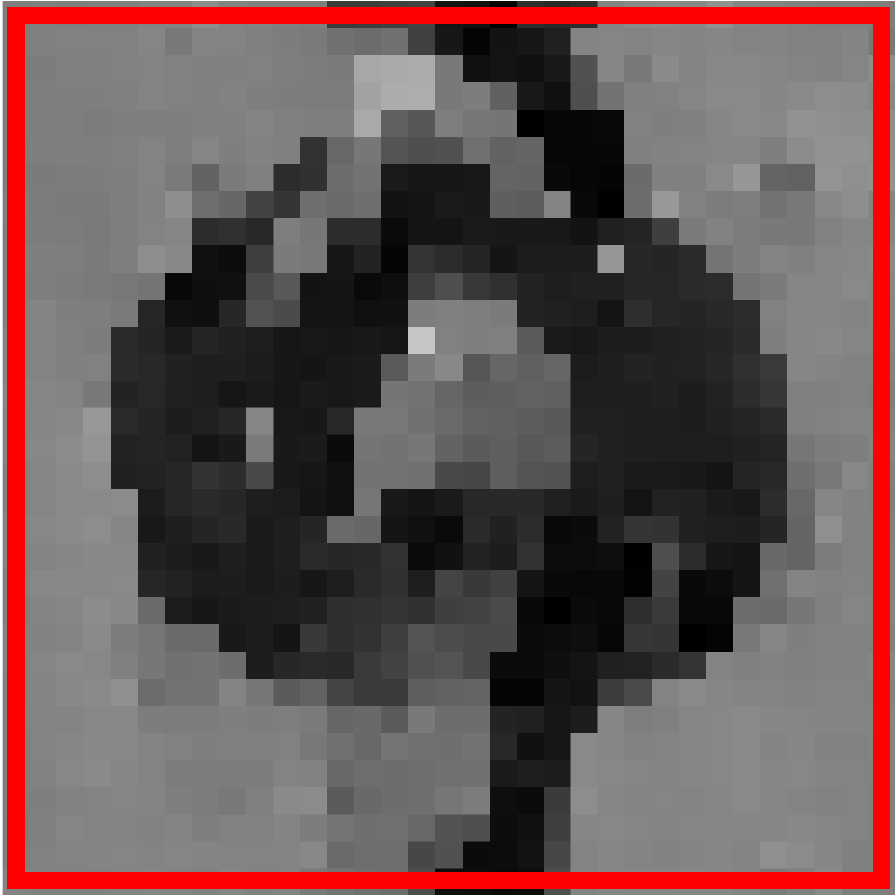} & 
				\includegraphics[width=0.2\textwidth]{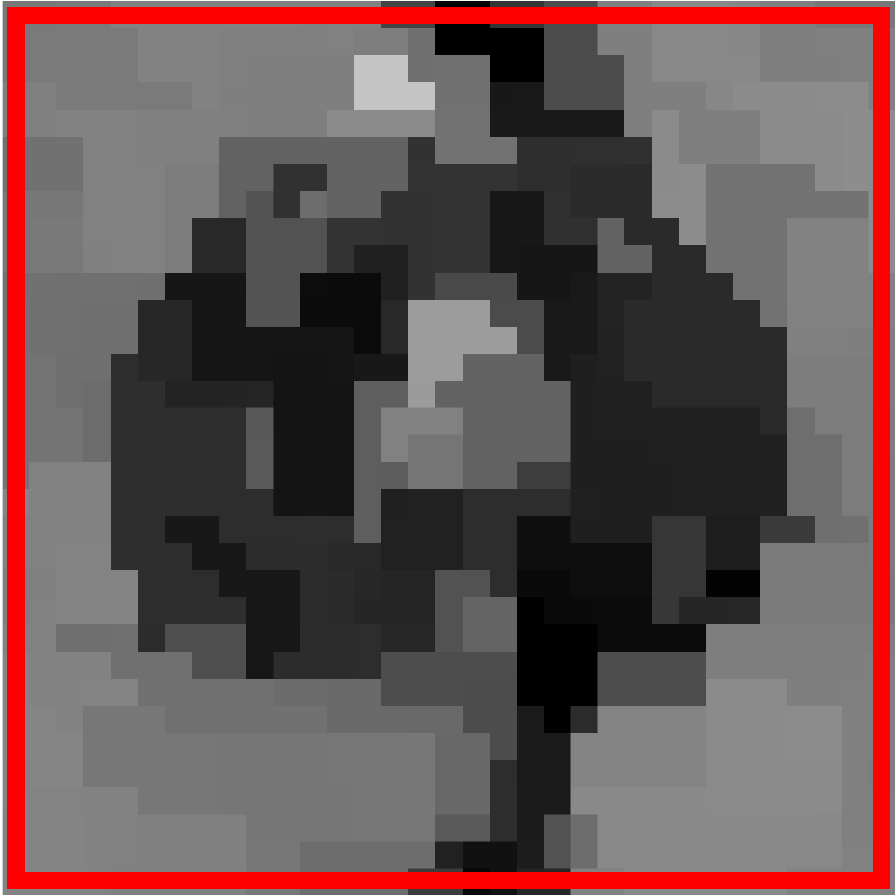} &
				\includegraphics[width=0.2\textwidth]{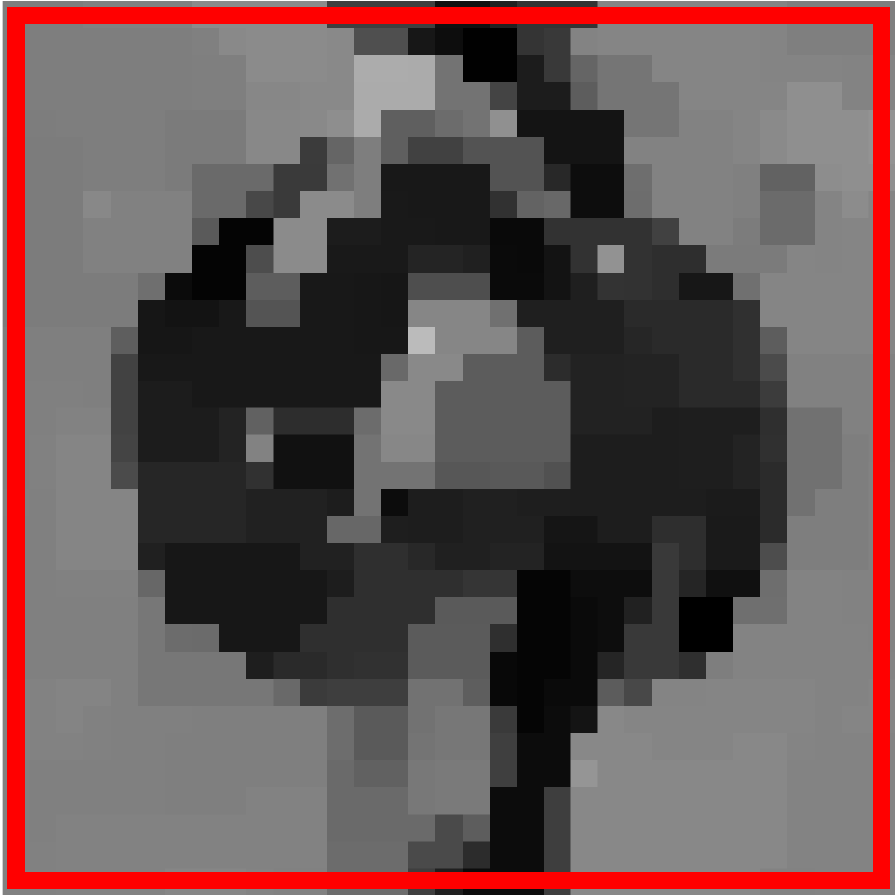} 
			\end{tabular}
		\end{center}
		\caption{CT reconstruction of a walnut in the $150^\circ$ projection range. The internal region of the walnut is zoomed-in and highlighted in red. The display window is $[0, 0.6]$.  }\label{fig:walnut150}
	\end{figure}

		\begin{figure}[t]
		\begin{center}
			\begin{tabular}{ccc}
				%		Sinogram & $L_1$-grad (RE $= 0.72\%$) & $L_1/L_2$-grad (RE $= 0.04\%$)\\
				Reference & SART & $L_1$ \\
				\includegraphics[width=0.2\textwidth]{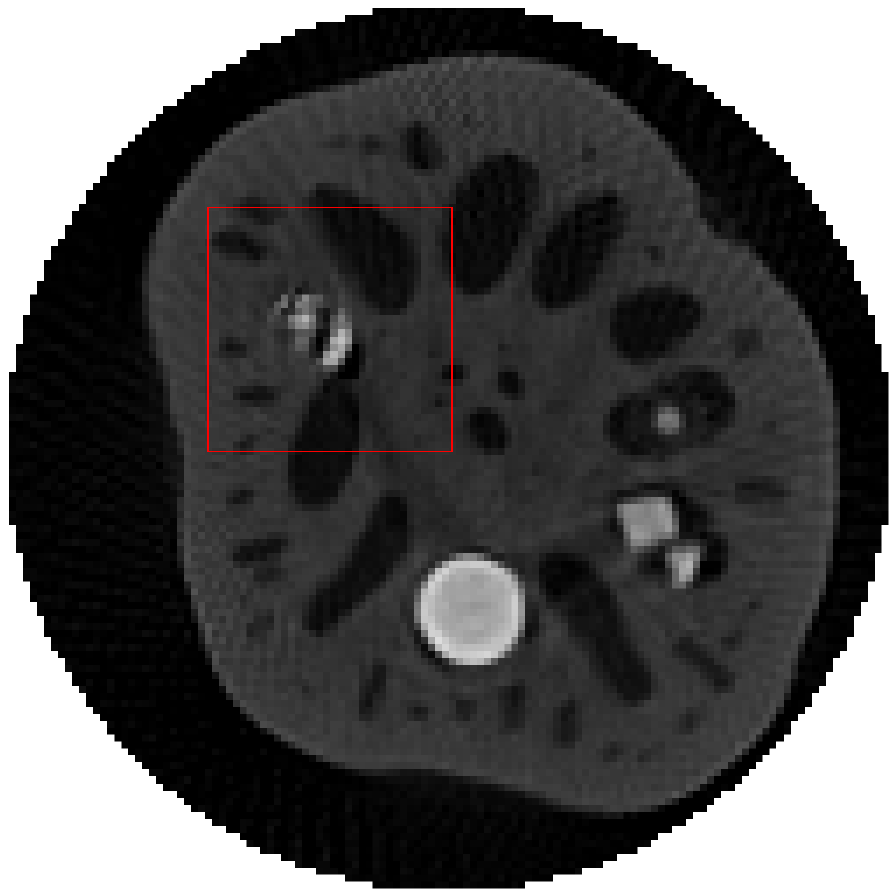} &
				\includegraphics[width=0.2\textwidth]{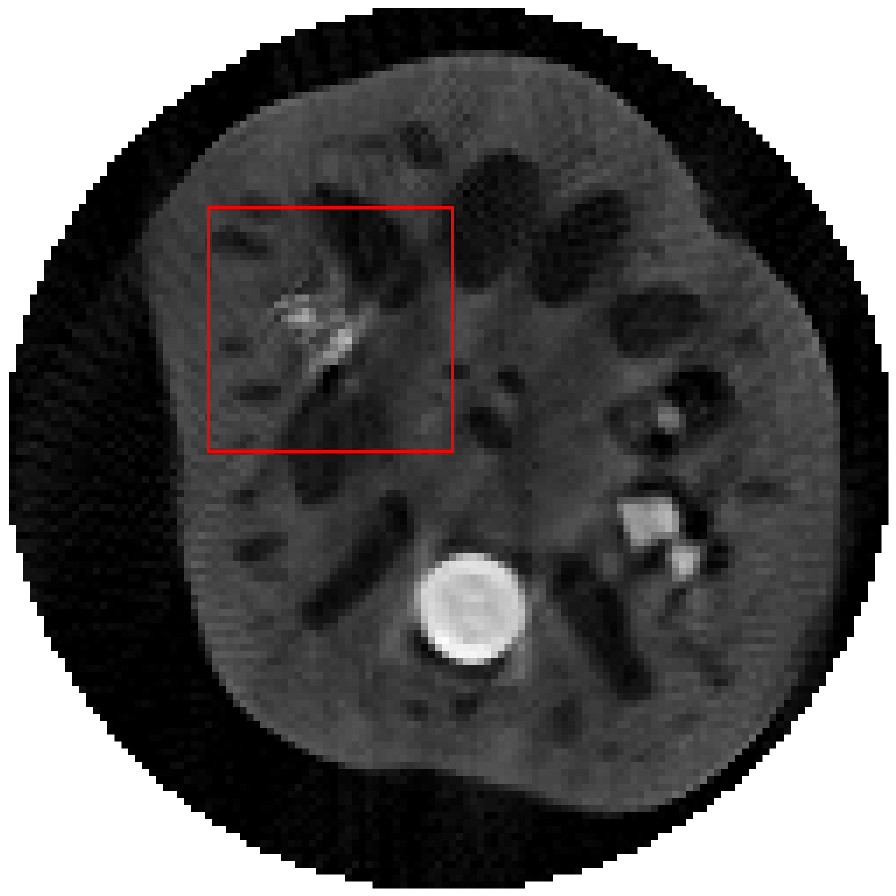} &
				\includegraphics[width=0.2\textwidth]{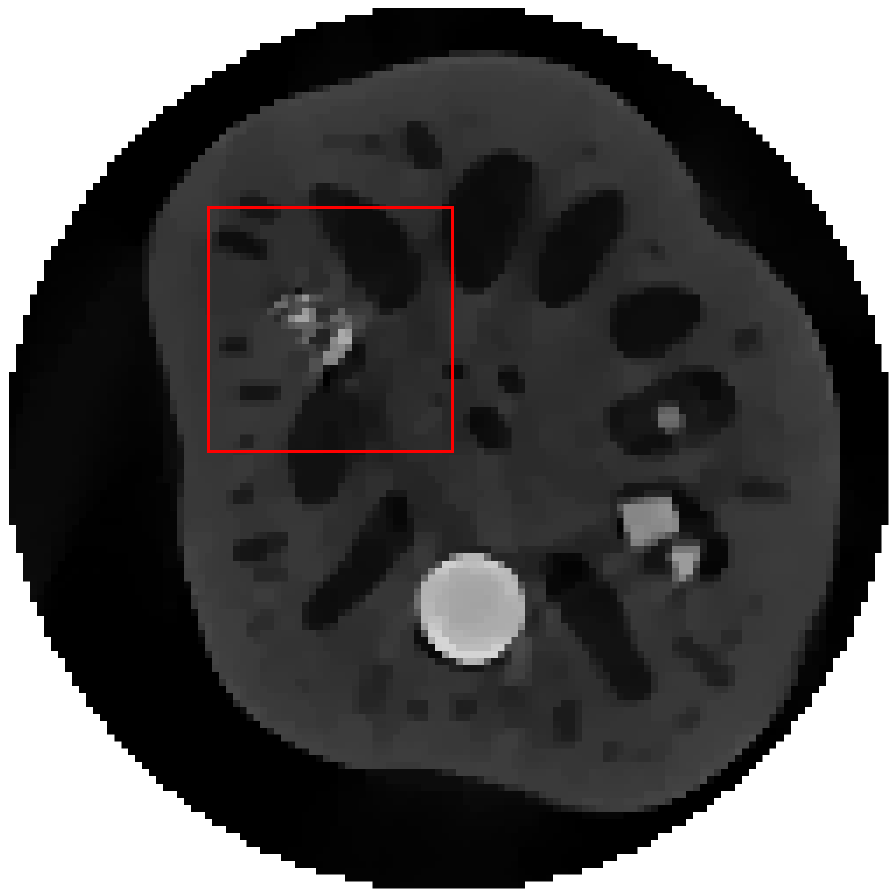} \\
				$L_p$  & 	$L_1$-$L_2$ & $L_1/L_2$\\
				\includegraphics[width=0.2\textwidth]{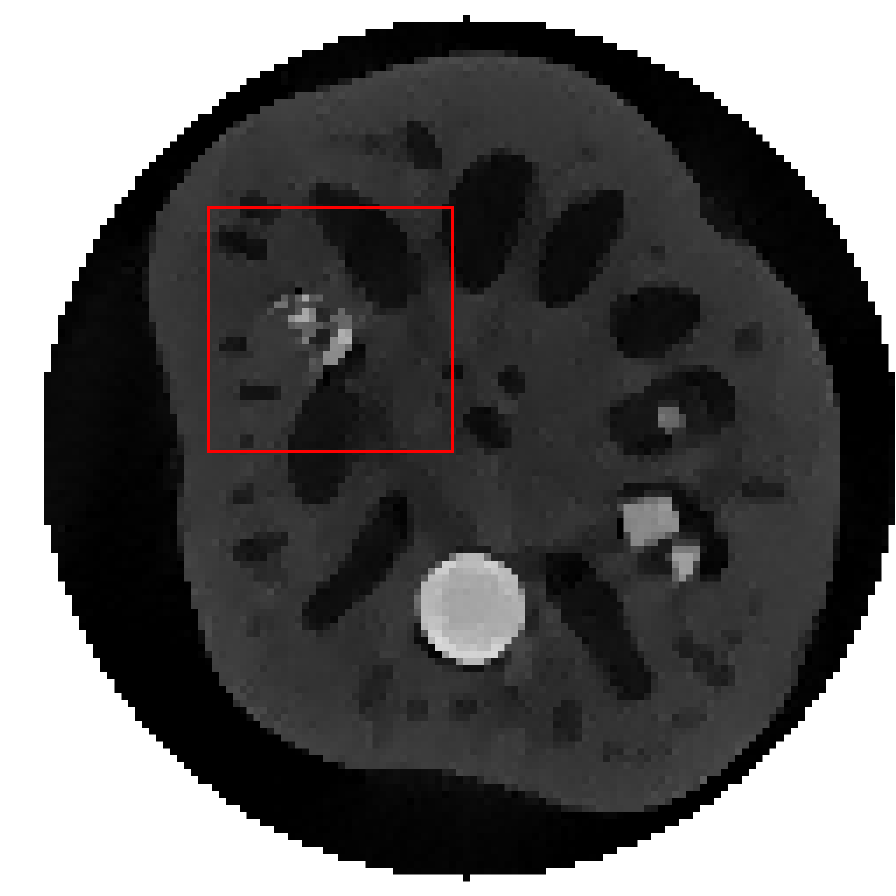} &
				\includegraphics[width=0.2\textwidth]{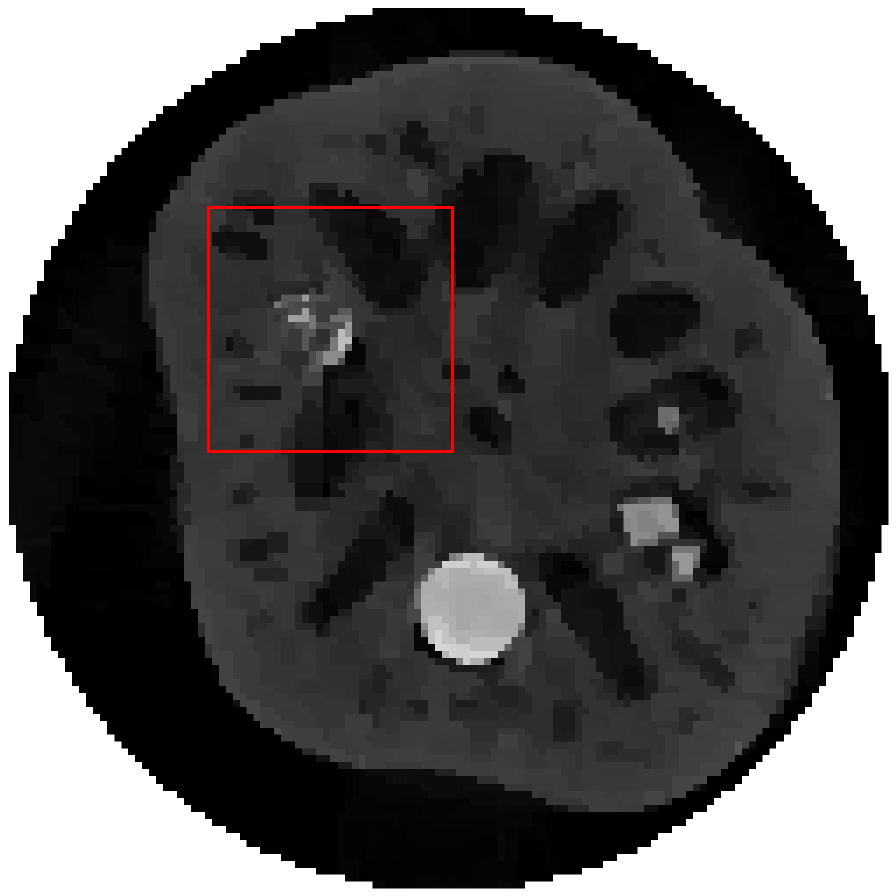} &
				\includegraphics[width=0.2\textwidth]{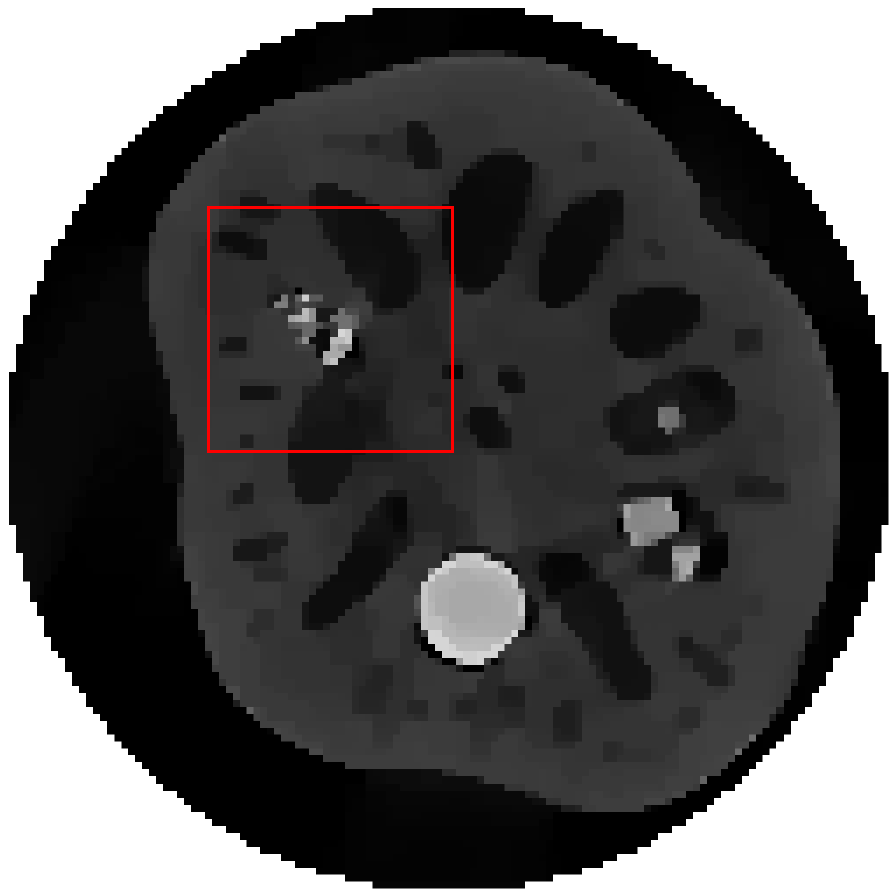} \\
				Reference & SART & $L_1$ \\
				\includegraphics[width=0.2\textwidth]{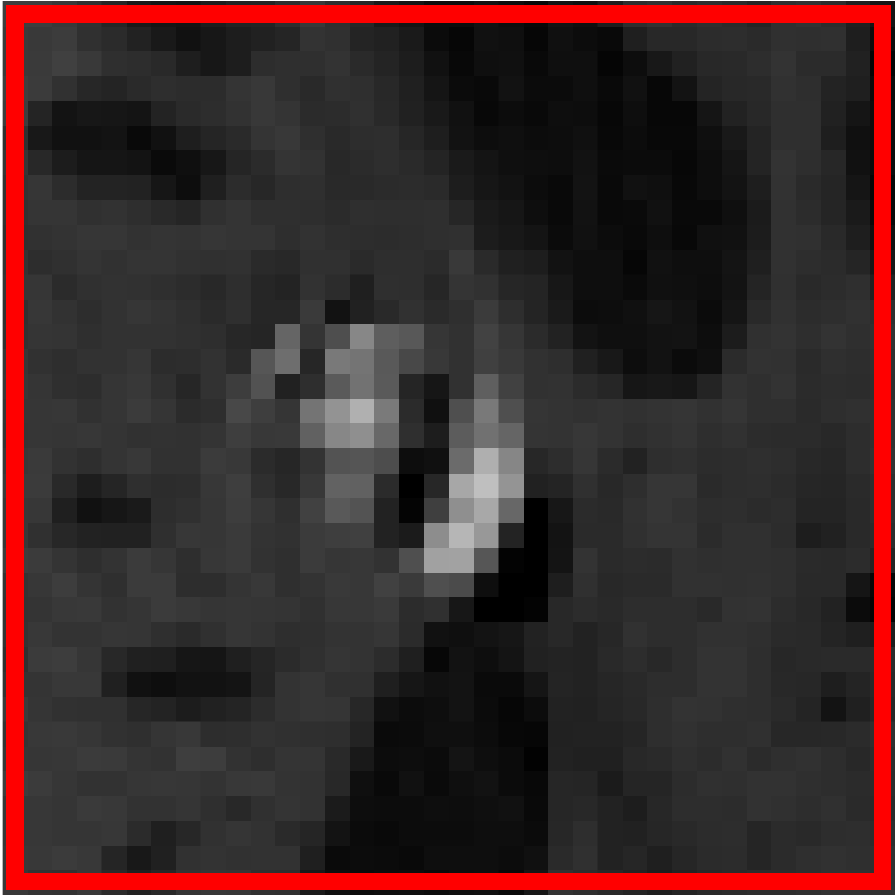} &
			     \includegraphics[width=0.2\textwidth]{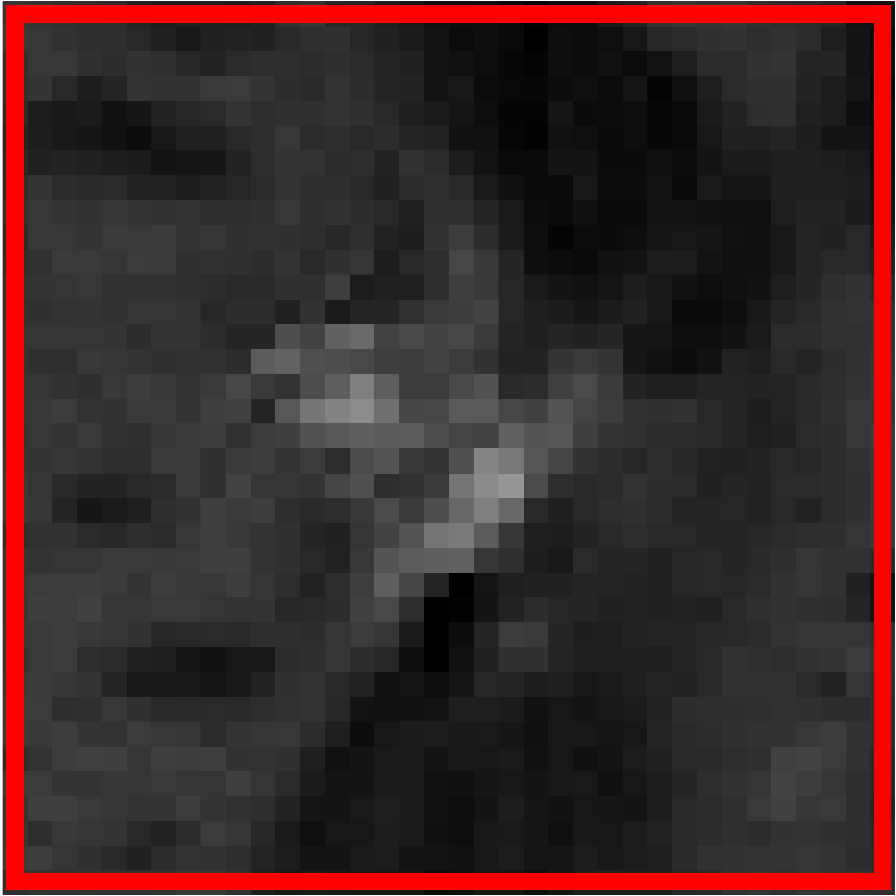} &
				\includegraphics[width=0.2\textwidth]{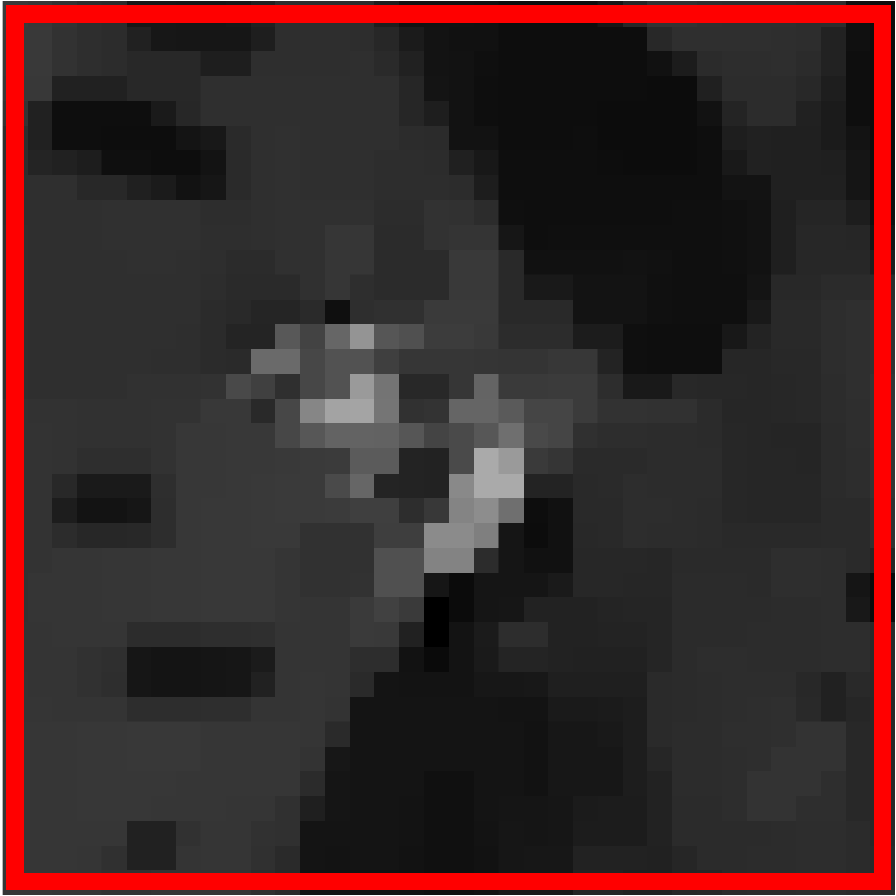} \\
				$L_p$  & 	$L_1$-$L_2$ & $L_1/L_2$\\
				\includegraphics[width=0.2\textwidth]{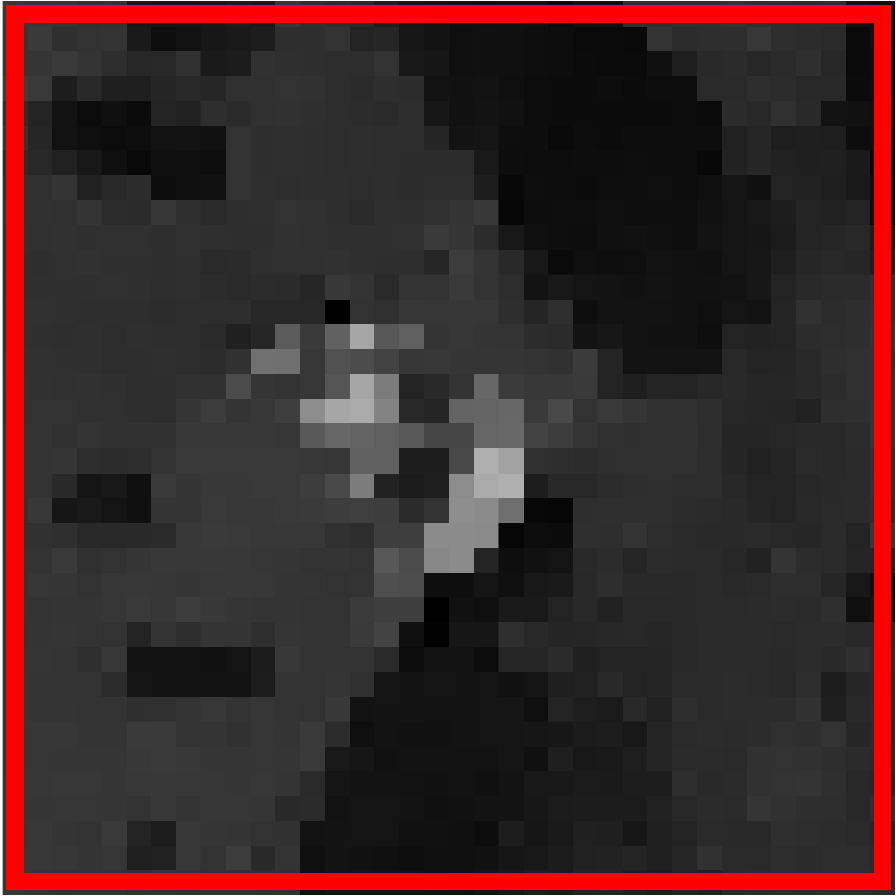} & 
				\includegraphics[width=0.2\textwidth]{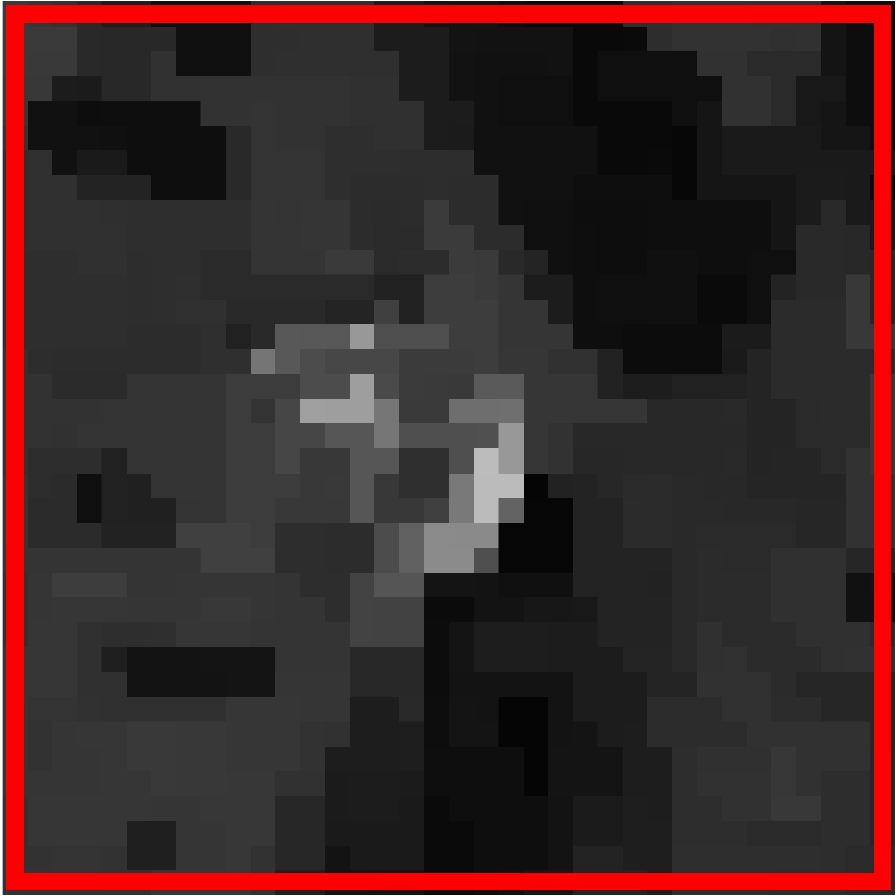} &
				\includegraphics[width=0.2\textwidth]{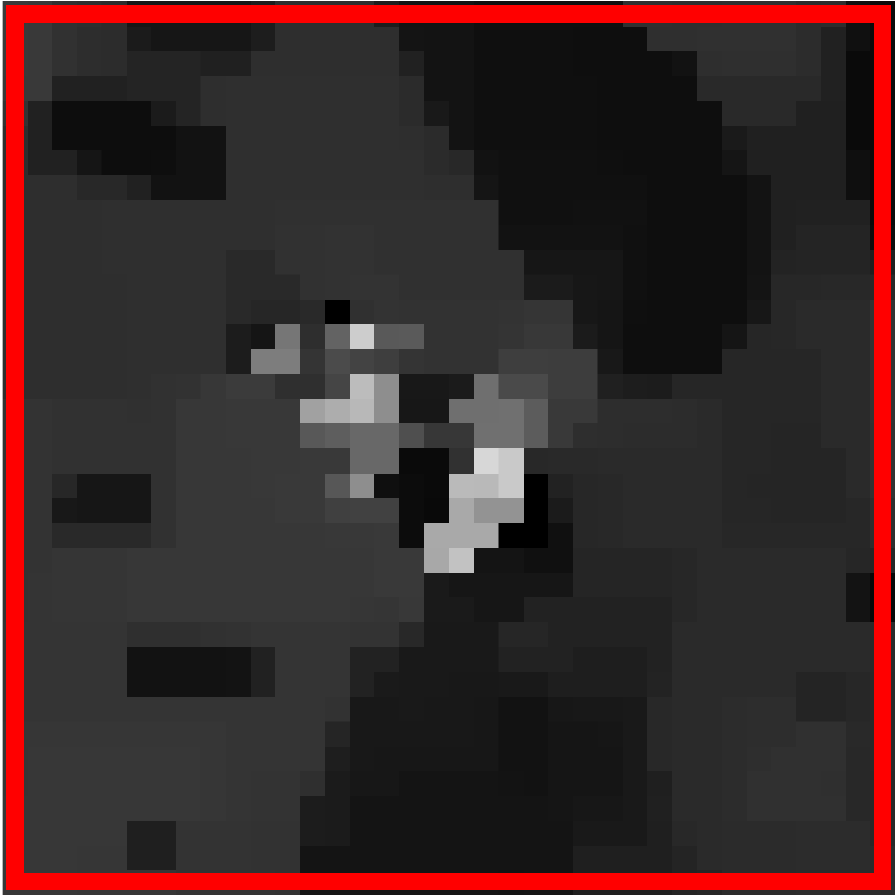} 
			\end{tabular}
		\end{center}
		\caption{CT reconstruction of a lotus in the $150^\circ$ projection range. One hole filled with attenuating objects   is zoomed-in and highlighted in red. The display window is $[0, 0.6].$}\label{fig:lotus150}
	\end{figure}
	
			\begin{table}[t]
		\begin{center}
			\caption{CT reconstruction of experimental data. }
			 \scriptsize
			\begin{tabular}{c|cc|cc|cc|cc|cc}
				\hline
				  \multirow{2}{*}{Dataset} & \multicolumn{2}{c|}{SART}& \multicolumn{2}{c|}{$L_1$ } &\multicolumn{2}{c|}{$L_p$ } & \multicolumn{2}{c}{$L_1$-$L_2$ } & \multicolumn{2}{|c}{$L_1/L_2$ }  \\ \cline{2-11} 
				 & SSIM &   RMSE & SSIM &   RMSE & SSIM & RMSE & SSIM &   RMSE & SSIM &   RMSE  \\ \hline
				  walnut & 0.87 & 0.049  &  {\bf 0.92} & {\bf 0.036}  & 0.91 & 0.040 & 0.90 & 0.041  &  0.91 & 0.041       \\ \cline{1-11} 
				 lotus & 0.93 & 0.022  &  {\bf 0.96} & {\bf 0.015}  &  0.96 & 0.016 & 0.95 & 0.017  &  0.96 & 0.017    
				 \\ \hline 
			\end{tabular}\label{Tab:real}
			\medskip
		\end{center}
	\end{table}

	\section{Conclusions and future works}
	\label{sec:conclusions}
	
	Following a preliminary work \cite{l1dl2}, we considered the use of $L_1/L_2$ on the gradient as a  regularization for imaging applications. We formulated an unconstrained model, which is novel and suitable when the noise is present. We also incorporated a box constraint that is reasonable and yet helpful for the CT reconstruction problem. We provided convergence guarantees for the proposed algorithms under mild conditions. We conducted extensive experiments to demonstrate that our approaches outperform the state-of-the-art in the limited-angle CT reconstruction subject to  either Gaussian noise or Poisson noise. Specifically, we validated the \tr{effectiveness and efficiency} of our approach by two experimental datasets. 
	
	%In reality, Poisson distribution is more appropriate than Gaussian distribution to describe the noise statistics of the projection data. Consequently, one future direction involves using different data fitting terms such as reweighted least-squares  \cite{thibault2007three} and Kullback-Leibler divergence for the CT reconstruction. 
	As both $L_1$ and $L_1/L_2$ models take about 10 minutes to run on MATLAB, we will 
% 	investigate the linearized ADMM \cite{nien2014fast} and 
	implement the algorithms on the GPU for fast computation. The extensions to a higher dimension as well as to  other medical and biological applications with real data, e.g., MRI, cone-beam CT, positron emission tomography (PET), and transmission electron microscopy (TEM),  are worth exploring in the future. 
	%Some conclusions here.
	
	\appendix
	\section{Proofs} 
	%\lipsum[71]
	%\input{appendix_v7.tex}
	
	To prepare for convergence analysis, we summarize some equivalent conditions for strong convexity and Lipschitz smooth functions in \Cref{lemma:strong_convex} and \Cref{lemma:lip}, respectively.
	\begin{lemma}\label{lemma:strong_convex}
		A function $f( x)$ is called strongly convex with parameter $\mu$ if and only if one of the following conditions holds
		\begin{enumerate}
			\item[(a)] $g(x) = f( x)-\frac \mu 2\|x\|_2^2$ is convex;
			\item[(b)] $\langle \nabla f(x)-\nabla f(y), x-y\rangle \geq \mu\| x- y\|_2^2, \ \forall x,  y;$
			\item[(c)] $f(y)\geq f( x)+\langle \nabla f( x), y- x\rangle + \frac \mu 2 \| y- x\|_2^2, \ \forall  x, y.$
		\end{enumerate}
	\end{lemma}
	
	\begin{lemma}\label{lemma:lip}
		The gradient of $f(x)$ is Lipschitz continuous with parameter $L>0$ if and only if one of the following conditions holds
		\begin{enumerate}
			\item[(a)]
			$\| \nabla f( x)-\nabla f( y)\|_2 \leq L\|x-y\|_2, \ \forall  x, y;$
			\item[(b)] $g( x) = \frac L 2\| x\|_2^2-f( x)$ is convex;
			\item[(c)] $f(y)\leq f( x)+\langle \nabla f( x),  y- x\rangle + \frac L 2 \| y- x\|_2^2, \ \forall x, y.$
		\end{enumerate}	
	\end{lemma}
	
	We show in \Cref{lem4.4} that the gradient of the function $f(\h x)=\frac{1}{\|\h x\|_2}$  is  Lipschitz continuous
	on  a set with a lower bound.

	\begin{lemma}\label{lem4.4} Given a function $f(\h x)=\frac{1}{\|\h x\|_2}$ and a set ${\cal M}_{\epsilon}:=\left\{\h x |\|\h x\|_2\ge \epsilon \right\}$ for a positive constant $\epsilon >0,$  we have
		\begin{eqnarray}\nn\|\nabla f(\h x)-\nabla f(\h y)\|_2\le \frac{2 }{\epsilon ^3}\|\h x -\h y\|_2,\;\forall\; \h x,\h y\in {\cal M}_{\epsilon}.
		\end{eqnarray}
		%where $c_1,\;c_2>0$.
	\end{lemma}
	
	\begin{proof}
% 		Some calculations lead to
% 		%	\begin{eqnarray*}
% 		\begin{equation*}
% 			\begin{split}
% 				& \left\|\nabla f(\h x)- \nabla f(\h y)\right\|_2 = \left\|\frac{\h x}{\|\h x\|^3_2}-\frac{\h y}{\|\h y\|^3_2} \right\|_2 
% 				= \frac{ 1}{\|\h x\|^3_2\|\h y\|^3_2} \Big\|\|\h y\|_2^3 \h x - \|\h x\|_2^3 \h y\Big\|_2\\
% 				%		&& = \frac{ 1}{\|\h x\|^3\|\h y\|^3} \Big\|\|\h y\|_2^3 \h x - \|\h x\|_2^3 \h x +\|\h x\|_2^3 \h x - \|\h x\|_2^3 \h y\Big\|_2\\
% 				\le & \frac{ 1}{\|\h x\|_2^2\|\h y\|_2^3}\Big| \|\h y\|_2^3-\|\h x\|_2^3\Big|  +\frac{ 1}{\|\h y\|_2^3}\|\h x-\h y\|_2\nn\\
% 				\le & \frac{ 1}{\|\h x\|_2^2\|\h y\|_2^3}\Big| \|\h y\|_2-\|\h x\|_2\Big|\Big(\|\h y\|_2^2+\|\h x\|_2\|\h y\|_2 + \|\h x\|_2^2 \Big)  +\frac{ 1}{\|\h y\|^3}\|\h x-\h y\|_2\nn\\
% 				\le &  \left(\frac{1}{\|\h x\|^2_2\|\h y\|_2}+\frac{1}{\|\h x\|_2\|\h y\|_2^2  }+\frac{1}{\|\h y\|_2^3} \right)\|\h y-\h x\|_2+\frac{1}{\|\h y\|_2^3}\|\h x-\h y\|_2
% 				\le   \frac{4}{\epsilon^3}\|\h x-\h y\|_2.
% 			\end{split}
% 		\end{equation*}
		Some calculations lead to 
$ \nabla f(\h x) = -\frac{\h x}{\|\h x\|_2^3}$  and  $\nabla^2 f(\h x) =-\frac{1}{\|\h x\|_2^3}  I +3 \h x \h x^T \frac{1}{\|\h x\|_2^5}$ 
		%\end{equation*} 
		with the identify matrix $I$. Then for $\forall \h y$, one has
		\begin{equation*}
		    \h y^T\nabla^2 f(\h x) \h y = -\frac{\h y^T \h y }{\|\h x\|_2^3} +3\frac{\h y^T \h x \h x^T \h y }{\|\h x\|_2^5} \leq 2\frac{\h y^T \h y }{\|\h x\|_2^3} \leq \frac{2}{\epsilon^3} \h y^T \h y,
		\end{equation*} 
		which implies that the maximum spectral radius of Hessian of $f$ is less than $\frac{2}{\epsilon^3}.$ 
%		which leads to our conclusion. }
	\end{proof}

	\subsection{Proof of \Cref{lem43}}
	\begin{proof}It follows from the optimality condition of the $\h h$-subproblem in (\ref{ADMML1overL2all}) that
		\begin{equation}\label{eq:opt4hsub}
		-\dfrac{a^{(k+1)}}{\|\h h^{(k+1)}\|^3}\h h^{(k+1)} + \rho_2\left(\h h^{(k+1)} -\nabla u^{(k+1)}-\h b_2^{(k)}\right) =0,
		\end{equation}
		where $a^{(k)}:=\|\nabla u^{(k)}\|_1$. Using the dual update $- \h b_2^{(k+1)}=\h h^{(k+1)} -\nabla u^{(k+1)}-\h b_2^{(k)},$  we have
		\begin{equation}
		\h b_2^{(k+1)}=-\frac{a^{(k+1)}}{\rho_2}\frac{\h h^{(k+1)}}{\|\h h^{(k+1)}\|^3_2},	
		\end{equation}
		and similarly,
		\begin{equation}	\label{equ:b_2}
		\h b_2^{(k)}=-\frac{a^{(k)}}{\rho_2}\frac{\h h^{(k)}}{\|\h h^{(k)}\|^3_2}.
		\end{equation}
		We can estimate
		\begin{eqnarray}
		&&	\|\h b_2^{(k+1)}-\h b_2^{(k)}\|_2=\frac{1}{\rho_2}\left\| a^{(k+1)}\frac{\h h^{(k+1)}}{\|\h h^{(k+1)}\|_2^3}- a^{(k)}\frac{\h h^{(k)}}{\|\h h^{(k)}\|^3}\right\|_2
		\nn	\\
		&	\le& \frac{1}{\rho_2}\left( \frac{1}{\|\h h^{(k+1)}\|_2^2}\big|a^{(k+1)}-a^{(k)}\big|+a^{(k)}\left\|\frac{\h h^{(k+1)}}{\|\h h^{(k+1)}\|_2^3}-\frac{\h h^{(k)}}{\|\h h^{(k)}\|_2^3}\right\|_2\right).\label{ineq:lemma4b}
		%	&\le &\frac{1}{\rho_2}\left(\frac{2\sqrt{2}n}{ \varepsilon^{2}}\|u^{(k)}-u^{k+1}\|+\frac{4M}{\varepsilon^3}\|h^{(k)}-h^{(k+1)}\| \right)\nn\\
		\end{eqnarray}
		For the first term in \eqref{ineq:lemma4b}, we use the facts that $\|\h x\|_1\leq \sqrt l\|\h x\|_2$ for a vector $\h x$ of the length of $l$ and $\|\nabla\|_2^2\leq 8$, thus leading to
		\begin{eqnarray}\label{ineq:lemma4a}
		|a^{(k+1)}-a^{(k)}|&\leq& \|\nabla (u^{(k+1)}-u^{(k)})\|_1\leq \sqrt{2mn} \|\nabla (u^{(k+1)}-u^{(k)})\|_2\nn\\
		&\leq & \sqrt{2mn}\cdot \|\nabla\|_2\cdot \|u^{(k+1)}-u^{(k)}\|_2\leq 4\sqrt{mn}\|u^{(k+1)}-u^{(k)}\|_2.
		\end{eqnarray}
		Note that $u\in\mathbb R^{m\times n}$ and $\nabla u \in\mathbb R^{m\times n\times 2}$ (thus of length $2mn$.)	
		%Since $\|\nabla u^{(k)}\|_1\leq M,$ we have $|a^{(k)}|\leq M, \forall k.$ %We also assume the feasible set of $u$ given by \eqref{eq:feasible_eps}.
		Invoking Lemma \ref{lem4.4}, 
		%	with $\kappa:=\epsilon$ and $\alpha := a^{(k)}$, 
		we get
		\begin{eqnarray}\label{ineq:lemma4h}
		a^{(k)}\left\|\frac{\h h^{(k+1)}}{\|\h h^{(k+1)}\|_2^3}-\frac{\h h^{(k)}}{\|\h h^{(k)}\|_2^3}\right\|_2\le\frac{2M}{\epsilon^3}\|\h h^{(k+1)}-\h h^{(k)}\|_2.
		\end{eqnarray}
		By putting together \eqref{ineq:lemma4b}-\eqref{ineq:lemma4h} and using the Cauchy-Schwarz inequality, we get 
		\eqref{ineq:lemma}.
	\end{proof}

	\subsection{Proof of \Cref{lem:suff_decr}}
	
	In order to prove \Cref{lem:suff_decr}, we show in \Cref{lem:suff_decr_u} that
	the augmented Lagrangian decreases sufficiently with respect to $u^{(k)}$.

	\begin{lemma}
		\label{lem:suff_decr_u}
		Under the same assumptions as in \Cref{lem:suff_decr},  there exists a constant $\bar c_1>0$ such that
		\begin{equation}\label{ineq:lem43_Lu}
		\mathcal{L}(u^{(k+1)}, \h h^{(k)}; \h b^{(k)}_2)- \mathcal{L}(u^{(k)}, \h h^{(k)}; \h b_2^{(k)})\leq - \frac {\bar c_1} 2 \|u^{(k+1)}-u^{(k)}\|_2^2,
		\end{equation}
		holds for the augmented Lagrangian corresponding to $L_1/L_2$-uncon and $L_1/L_2$-box. 
	\end{lemma}

	\begin{proof}

Denote $\sigma$ as the smallest eigenvalue of the matrix $A^T A+\nabla^T \nabla.$  We show $\sigma$ is strictly positive. If $\sigma=0,$ there exists a vector $x$ such that  $x^T(A^T A+\nabla^T \nabla) x = 0$. It is straightforward that $x^T A^T A x\ge0$ and $x^T \nabla^T \nabla x\ge0$. Therefore, one shall have $x^T A^T A x=0$ and $x^T \nabla^T \nabla x=0$, which contradicts with Assumption A1 that ${\cal N}(\nabla)\bigcap{\cal N}(A)={0}$. Therefore, we have that
\[
		v^T(A^T A + \nabla^T\nabla) v \geq  \sigma \|v\|_2^2, \quad \forall v,
		\]
which implies that $\mathcal L_{\rm uncon}(u, \h h^{(k)}; \h b_2^{(k)})$ with fixed $\h h^{(k)}$ and $\h b_2^{(k)}$ is strongly convex with parameter $\bar c_1 = \sigma \lambda$ (we can choose $\rho_2\geq \lambda$ as it is sufficiently large.) 
It follows from \eqref{eq:Lboxvsuncon} that the only difference between $\mathcal{L}_{\rm uncon}$ and  
$\mathcal{L}_{\rm box}$ is the indicator function $\Pi_{[c,d]}(u).$ 
Since the indicator function is convex, then $\mathcal{L}_{\rm box}$ is strongly convex with the same parameter $c_1$. We can unify $\mathcal{L}_{\rm uncon}$ and  
$\mathcal{L}_{\rm box}$ 		to be $\mathcal L.$
		 Then \Cref{lemma:strong_convex} leads to
		$$\mathcal{L}(u^{(k+1)}, \h h^{(k)}; \h b^{(k)}_2)\le \mathcal{L}(u^{(k)}, \h h^{(k)}; \h b_2^{(k)})-\frac{\sigma\lambda}{2}\|u^{(k+1)}-u^{(k)}\|^2_2.$$
		Therefore, we can choose $\bar c_1=\sigma \lambda$  such that the inequality \eqref{ineq:lem43_Lu} holds.
		\end{proof}		

	Now we are ready to prove for \Cref{lem:suff_decr}.
	\begin{proof}
		Denote $a=\|u^{(k+1)}\|_1$ and $L=\frac{2M}{\epsilon^3}$.  Lemma~\ref{lem4.4} and Lemma~\ref{lemma:lip} (c) lead to
		\begin{equation}\label{ineq:lem43_h}
		\frac a{\|\h h^{(k+1)}\|_2}\leq \frac a{\|\h h^{(k)}\|_2} -\left\langle \frac {a \h h^{(k)}}{\|\h h^{(k)}\|_2^3}, \h h^{(k+1)}-\h h^{(k)}\right\rangle + \frac{L}{2}\|\h h^{(k+1)}-\h h^{(k)}\|_2^2.
		\end{equation}
		Denoting $\h z = \nabla u^{(k+1)}+\h b_2^{(k)}$ and using the optimality condition of $\h h^{(k+1)}$ \eqref{eq:opt4hsub}, we get
		\begin{eqnarray}
		&&\frac {\rho_2} 2 \|\h h^{(k+1)}-\h z\|_2^2 - \frac {\rho_2} 2 \|\h h^{(k)}-\h z\|_2^2\nn\\
		&=& \frac {\rho_2} 2 \|\h h^{(k+1)}\|_2^2 - \frac {\rho_2} 2 \|\h h^{(k)}\|_2^2-\left\langle 	-\dfrac{a \h h^{(k+1)}}{\|\h h^{(k+1)}\|^3} + \rho_2\h h^{(k+1)}, \h h^{(k+1)}-\h h^{(k)}\right\rangle\nn\\
		& = & \left\langle 	\dfrac{a\h h^{(k+1)}}{\|\h h^{(k+1)}\|^3}, \h h^{(k+1)}-\h h^{(k)}\right\rangle -\frac {\rho_2}  2\|\h h^{(k+1)}-\h h^{(k)}\|_2^2\label{eq:lemma45h}.
		\end{eqnarray}
		Combining \eqref{ineq:lem43_h} and \eqref{eq:lemma45h}, we obtain
		\begin{eqnarray}
		\label{ineq:lem43_Lh}	&&\mathcal{L}(u^{(k+1)}, \h h^{(k+1)};\h b^{(k)}_2)- \mathcal{L}(u^{(k+1)}, \h h^{(k)};\h b_2^{(k)})\\
		&\leq& \left\langle 	\dfrac{a\h h^{(k+1)}}{\|\h h^{(k+1)}\|^3_2}-\dfrac{a\h h^{(k)}}{\|\h h^{(k)}\|^3_2}, \h h^{(k+1)}-\h h^{(k)}\right\rangle -\frac{\rho_2-L}{2}\|\h h^{(k+1)}-\h h^{(k)}\|_2^2\nn\\
		&\leq & \left\|\dfrac{a\h h^{(k+1)}}{\|\h h^{(k+1)}\|^3_2}-\dfrac{a\h h^{(k)}}{\|\h h^{(k)}\|^3_2}\right\|_2\left\|\h h^{(k+1)}-\h h^{(k)}\right\|_2	-\frac{\rho_2-L}{2}\left\|\h h^{(k+1)}-\h h^{(k)}\right\|_2^2\nn\\
		&\leq& -\frac{\rho_2-3L}{2}\|\h h^{(k+1)}-\h h^{(k)}\|_2^2.\nn
		\end{eqnarray}
		
		% with some manipulations, we obtain the following inequality:
		%
		%hence $\frac{a}{\|\h h\|_2}+\frac {\rho_2} 2\|\h h-\nabla u^{(k+1)}-\h b_2^{(k)}\|_2^2$ is strongly convex with parameter $\rho_2$. Using Lemma~\ref{lemma:strong_convex} and optimality condition of $\h h^{(k+1)}$, we have
		
		%	where ${\bar c_2}=\displaystyle{\rho_2-\frac{2M}{\varepsilon^3}}$.

		Lastly, from the update of $\h b_2,$ we compute
		\begin{eqnarray}\label{ineq:lem43_Lb}
		&&\mathcal{L}(u^{(k+1)}, \h h^{(k+1)};\h b^{(k+1)}_2)- \mathcal{L}(u^{(k+1)}, \h h^{(k+1)};\h b_2^{(k)})\\
		&=&\frac{\rho_2}2 \Big(\|\h b_2^{(k)}\|_2^2-\|\h b_2^{(k+1)}-2\h b_2^{(k)}\|_2^2\Big)\leq \frac{\rho_2}2\|\h b_2^{(k+1)}-\h b_2^{(k)}\|^2_2.\nn
		\end{eqnarray}
		By putting the inequalities \eqref{ineq:lem43_Lu}, \eqref{ineq:lem43_Lh}, and \eqref{ineq:lem43_Lb} together with Lemma~\ref{lem43}, we have
		$$\mathcal{L}(u^{(k+1)}, \h h^{(k+1)};\h b^{(k+1)}_2)\le \mathcal{L}(u^{(k)}, \h h^{(k)};\h b_2^{(k)})-c_1\|u^{(k+1)}-u^{(k)}\|_2^2 -c_2\|\h h^{(k)}-\h h^{(k+1)}\|_2^2,$$
		where $c_1 = \frac{\bar c_1}2-\frac{16mn}{\rho_2\epsilon^4}$ and $c_2 =\frac{\rho_2\epsilon^3-6M}{2\epsilon^3}-\frac{16M^2}{\rho_2\epsilon^6}$.
		%$c_1 = \frac{\bar c_1}2-\frac{16mn}{\rho_2\epsilon^4}$ and $c_2 =\frac{\rho_2\epsilon^3-3M}{2\epsilon^3}-\frac{32M^2}{\rho_2\epsilon^6}$.
		%{\color{red} [please check these constants.]}
		For sufficiently large $\rho_2,$  we can have $c_1,c_2>0.$
		%
		%	$$\mathcal{L}(u^{(k+1)}, \h h^{(k+1)} ;\h b^{(k+1)}_2)\le \mathcal{L}(u^{(k)}, \h h^{(k)} ;\h b_2^{(k)})+\rho_2\left(\frac{16 n^2}{\rho_2^2\varepsilon^{4}}\|u^{(k)}-u^{(k+1)}\|^2+ \frac{32M^2}{\rho_2^2\varepsilon^6}\|h^{(k)}-h^{k+1}\|^2\right),$$
		%	By setting ${\hat \rho}=\max(\lambda,\frac{32 n^2}{\varepsilon^4\lambda},\frac{\frac{2M}{\varepsilon^3}+\sqrt{\frac{4M^2}{\varepsilon^6}+\frac{256M^2}{\varepsilon^6}}}{2}),$
		%	and $c_1=\frac{1}{2}{\bar c}_1-\frac{16 n^2}{\rho_2\varepsilon^4}>0$, $c_2=\frac{1}{2}{\bar c}_2-\frac{32M^2}{\rho_2\varepsilon^6}>0$, the assertion (i) follows directly.
	\end{proof}
	
	\begin{remark}
		It seems that we need a very large value of $\rho_2$ to guarantee $c_1,c_2>0$ in \Cref{lem:suff_decr}. Fortunately, it is just a sufficient condition for convergence and we can choose a reasonable value of $\rho_2$ in practice; please refer to \Cref{sec:experiments} for parameter tuning.
	\end{remark}

	\subsection{Proof of \Cref{thm:stantionary_con}}
	
	\begin{proof}
		
		To accommodate the models (with and without box),  we express the optimality condition of (\ref{ADMML1overL2all}) as follows,
		\begin{eqnarray}\label{con:equ:sym1}
		\left\{\begin{array}{l}
		\frac{p^{(k+1)}}{\|\h h^{(k)}\|_2}+q^{(k+1)}+r^{(k+1)}+\rho_2\nabla^T(\nabla u^{(k+1)}- \h h^{(k)}+ \h b_2^{(k)})=0\\[0.1cm]
		-\frac{\|\nabla u^{(k+1)}\|_1}{\|\h h^{(k+1)}\|_2^3}\h h^{(k+1)}  +\rho_2(\h h^{(k+1)}-\nabla u^{(k+1)}- \h b_2^{(k)})=0\\[0.1cm]
		\h b_2^{(k+1)}=\h b_2^{(k)}+\nabla u^{(k+1)} - \h h^{(k+1)},
		\end{array}\right.
		\end{eqnarray}
		where $p^{(k+1)}\in \partial \| \nabla u^{(k+1)}\|_1$, $q^{(k+1)}:=\lambda A^T(Au^{(k+1)}-f)$, and $r^{(k+1)}$ either belongs to $\partial(\Pi_{[c,d]}(u^{(k+1)}))$ with the box constraint or zero otherwise.
		Let $\eta_1^{(k+1)}, \eta_2^{(k+1)}, \eta_2^{(k+1)}$ be 
		\begin{eqnarray}\label{con:equ:sym2}
		\left\{\begin{array}{l}
		\eta_1^{(k+1)}:=  \frac{p^{(k+1)}}{\|\h h^{(k+1)}\|_2}+q^{(k+1)}+r^{(k+1)}+\rho_2\nabla^T(\nabla u^{(k+1)}-\h h^{(k+1)}+\h  b_2^{(k+1)})\\[0.1cm]
		\eta_2^{(k+1)} :=  -\frac{\|\nabla u^{(k+1)}\|_1}{\|\h h^{(k+1)}\|_2^3}\h h^{(k+1)}  +\rho_2(\h h^{(k+1)}-\nabla u^{(k+1)}-\h b_2^{(k+1)})\\[0.1cm]
		\eta_3^{(k+1)} :=  \rho_2(\nabla u^{(k+1)} - \h h^{(k+1)}).
		\end{array}\right.
		\end{eqnarray}
		Clearly, we have
		\begin{equation*}
			\begin{split}
				\eta_1^{(k+1)} & \in \partial_{u}{\cal L}(u^{(k+1)},\h h^{(k+1)},\h b_2^{(k+1)})\\
				\eta_2^{(k+1)} & \in \partial_{\h h}{\cal L}(u^{(k+1)},\h h^{(k+1)},\h b_2^{(k+1)})\\
				\eta_3^{(k+1)} & \in \partial_{\h b_2}{\cal L}(u^{(k+1)},\h h^{(k+1)},\h b_2^{(k+1)}),
			\end{split}
		\end{equation*}
		for $\mathcal L=\mathcal L_{\rm{uncon}}$ or  $\mathcal L_{\rm{box}}.$ 
		Combining \eqref{con:equ:sym1} and \eqref{con:equ:sym2} leads to
		\begin{eqnarray*}
			\left\{\begin{array}{l}
				\eta_1^{(k+1)}= -\frac{p^{(k+1)}}{\|\h h^{(k)}\|_2}+\frac{p^{(k+1)}}{\|\h h^{(k+1)}\|_2}+\rho_2\nabla^T (\h h^{(k)}-\h h^{(k+1)})+\rho_2\nabla^T(\h b_2^{(k+1)}-\h b_2^{(k)})\\[0.1cm]
				\eta_2^{(k+1)}= \rho_2(\h b_2^{(k)}-\h b_2^{(k+1)})\\[0.1cm]
				\eta_3^{(k+1)}=\rho_2(\h b_2^{(k+1)}-\h b_2^{(k)}).
			\end{array}\right.
		\end{eqnarray*}
		The chain rule of subgradient \cite{hiriart-lemarechal-1996} suggests that 	 
		$\partial \|\nabla u\|_1 = \nabla^T \h q $, where 
		\begin{equation*}
			\h q = \{\h q  \ | \ \langle\h q, \nabla u \rangle_Y = \|\nabla u\|_1, \ |q_{ijk}| \leq 1, \forall i,j,k   \}.
		\end{equation*} 	
		%	 $\partial \|\nabla u\|_1 = \nabla^T \partial \|\nabla u\|_1$, where
		%	 	\[
		%	 	\partial |r|_1=\left\{\begin{array}{ll}
		%	 	[-1,1] & r=0 \ ,\\
		%	 	\mbox{sign}(r) & \mbox{otherwise} \ .
		%	 	\end{array}\right.
		%	 	\]
		Therefore, we have an upper bound for $\|p^{(k+1)}\|_2\leq \|\nabla ^T\|_2\|\h q^{(k+1)}\|_2\leq 2\sqrt{2mn}$. Simple calculations show that
		\begin{eqnarray*}\label{ineq:lemma43}
			&&	 \left\|\frac{p^{(k+1)}}{\|\h h^{(k)}\|_2}-\frac{p^{(k+1)}}{\|\h h^{(k+1)}\|_2}\right\|_2 =\left|\frac{1}{\|\h h^{(k)}\|_2}-\frac{1}{\|\h h^{(k+1)}\|_2} \right| \left\|p^{(k+1)}\right\|_2\\
			&	 \leq& \frac{1}{\epsilon^2}\left\|\h h^{(k+1)}-\h h^{(k)}\right\|_2\left\|p^{(k+1)}\right\|_2\le  \frac{2\sqrt{2mn}}{\epsilon^2} \left\|\h h^{(k+1)}-\h h^{(k)}\right\|_2.
		\end{eqnarray*}
		% and the inequality \eqref{ineq:lemma43} becomes
		%	\begin{equation*}
		%\left\|\frac{p^{(k+1)}}{\|\h h^{(k)}\|_2}-\frac{p^{(k+1)}}{\|\h h^{(k+1)}\|_2}\right\|_2 .
		%	\end{equation*}
		Finally, by setting $\gamma = \max\{26 \rho^2, \ 24\rho^2+ \frac{24mn}{\epsilon^4}\}$,
		%$\gamma=\max \left(17\rho_2^2,(\frac{16mn}{\varepsilon^4}+2\rho_2^2)\right)$
		%	{\color{red} [I didn't get $2\sqrt{2mn}$ and check these constants] } 
		\eqref{eq2} follows immediately. 
		%	\tb{\it CW: The constant in my calculation is  but I don't have any term for $\|u^{(k+1)}-u^{(k)}\|_2$.  } {\color{red} [This is what I get: $\gamma =\max(\frac{24mn}{\epsilon^4}+8mn\rho_2^2,24mn\rho_2^2+2\rho_2^2).$]}
	\end{proof}

	\bibliographystyle{siamplain}
	\bibliography{refer_l1dl2}
	
\end{document}